\documentclass[a4paper,twoside,10pt]{article}
\usepackage[a4paper,left=3cm,right=3cm, top=3cm, bottom=3cm]{geometry}

\usepackage{appendix}
\usepackage{amsfonts}
\usepackage{epstopdf}
\usepackage{amssymb}
\usepackage{stmaryrd}
\usepackage{float}
\usepackage{bigints}
\usepackage{color}
\usepackage{graphicx}
\usepackage{multirow}
\usepackage{algorithmic}
\usepackage{algorithm}

\newcommand{\eremk}{\hbox{}\hfill\rule{0.8ex}{0.8ex}}
\newcommand{\g}{g}
\newcommand{\V}{V}
\newcommand{\W}{W}
\newcommand{\vp}{v^+}
\newcommand{\vm}{v^-}
\newcommand{\Vg}{\V_\g}
\newcommand{\Vz}{\V_0}
\renewcommand{\a}{a}
\renewcommand{\b}{b}
\newcommand{\bE}{\b^\E}
\newcommand{\bn}{\b_\h}
\newcommand{\bnLcal}{\bn^{\Lcal}}
\newcommand{\bnH}{\bn^{H}}
\newcommand{\aE}{\a^\E}
\newcommand{\anDelta}{\a_\h^{\Delta}}
\newcommand{\anEDelta}{\a_\h^{\E, \Delta}}
\newcommand{\anLcal}{\a_\h^{\Lcal}}
\newcommand{\anELcal}{\a_\h^{\E, \Lcal}}
\newcommand{\anH}{\a_\h^{H}}
\newcommand{\anEH}{\a_\h^{\E, H}}
\newcommand{\E}{K}
\newcommand{\Ep}{\E^+}
\newcommand{\Em}{\E^-}
\newcommand{\e}{e}

\newcommand{\EE}{\mathcal E^\E}
\newcommand{\taun}{\mathcal T_\h}
\newcommand{\h}{h}
\newcommand{\hE}{\h_\E}
\newcommand{\he}{\h_\e}
\newcommand{\p}{p}
\newcommand{\pLcal}{\p^{\Lcal}}
\newcommand{\qp}{q_\p}
\newcommand{\qpDelta}{\qp^{\Delta}}
\newcommand{\VhDelta}{V_\h^\Delta}
\newcommand{\VhH}{V_\h^H}
\newcommand{\VhDeltag}{V_{\h,\g}^\Delta}
\newcommand{\VhDeltaz}{V_{\h,0}^\Delta}
\newcommand{\VhDeltaE}{\VhDelta(\E)}
\newcommand{\VhHE}{\VhH(\E)}
\newcommand{\uh}{u_\h}
\newcommand{\uhDelta}{\uh^\Delta}
\newcommand{\uhH}{\uh^H}
\newcommand{\uhLcal}{\uh^\Lcal}
\newcommand{\vh}{v_\h}
\newcommand{\vhDelta}{\vh^\Delta}
\newcommand{\vhH}{\vh^H}
\newcommand{\vhLcal}{\vh^{\Lcal}}
\newcommand{\n}{\mathbf n}
\newcommand{\nOmega}{\mathbf n_\Omega}
\newcommand{\nE}{\n_\E}
\newcommand{\nEp}{\n_{\Ep}}
\newcommand{\nEm}{\n_{\Em}}
\renewcommand{\ne}{\n_\e}
\newcommand{\malphae}{m_\alpha^\e}
\newcommand{\Eh}{\mathcal E_h}
\newcommand{\En}{\mathcal E_h}

\newcommand{\Vcaln}{\mathcal V_h}
\newcommand{\VcalE}{\mathcal V^\E}
\newcommand{\EhB}{\Eh^B}
\newcommand{\EhI}{\Eh^I}
\newcommand{\EnB}{\Eh^B}
\newcommand{\EnI}{\Eh^I}

\newcommand{\HncpDelta}{H^{nc,\Delta}_\p(\Omega,\taun)}
\newcommand{\HncpLcal}{H^{nc,\Lcal}_\p(\Omega,\taun)}
\newcommand{\HncpH}{H^{nc,H}_\p(\Omega,\taun)}
\newcommand{\uI}{u_I}
\newcommand{\uIDelta}{\uI^{\Delta}}
\newcommand{\uIH}{\uI^{H}}
\newcommand{\uILcal}{\uI^{\Lcal}}
\newcommand{\PinablaDeltap}{\Pi^{\nabla,\Delta}_\p}
\newcommand{\PiLcalp}{\Pi^{\nabla,\Lcal}_\p}
\newcommand{\PieLcal}{\Pi_p^{\e,\Lcal}}
\newcommand{\PiHp}{\Pi^{\nabla,H}_\p}
\newcommand{\PieH}{\Pi^{\e,H}_\p}

\newcommand{\SEDelta}{S^{\E,\Delta}}
\newcommand{\SELcal}{S^{\E,\Lcal}}
\newcommand{\SEH}{S^{\E,H}}
\newcommand{\Ncalh}{\mathcal N_h}
\newcommand{\Lcal}{\mathcal L}
\newcommand{\shLcal}{s^{\Lcal}_\p}

\newcommand{\ShLcal}{\mathcal S^{\Lcal}_\p}
\newcommand{\ShzLcal}{\mathcal S^{0,\Lcal}_\p}

\newcommand{\VhLcal}{V_\h^{\Lcal}}
\newcommand{\WhLcal}{W_\h^{\Lcal}}
\newcommand{\VhLcalE}{\VhLcal(\E)}
\newcommand{\subVhLcal}{\hat{V}_\h^{\Lcal}}
\newcommand{\subWhLcal}{\hat{W}_\h^{\Lcal}}

\newcommand{\G}{G}
\newcommand{\Gn}{\G_\h}
\newcommand{\GnLcal}{\Gn^{\Lcal}}
\renewcommand{\k}{k}
\newcommand{\kn}{\k_n}
\renewcommand{\i}{\mathrm{i}}
\newcommand{\im}{\i}
\newcommand{\Nbb}{\mathbb N}
\newcommand{\Rbb}{\mathbb R}
\newcommand{\Cbb}{\mathbb C}
\renewcommand{\wp}{w_\p}
\newcommand{\wpe}{w_\alpha^\e}
\newcommand{\well}{w_\ell}
\newcommand{\welle}{\well^\e}
\newcommand{\wellE}{\wp}
\newcommand{\PW}{\mathbb{PW}}
\newcommand{\PWp}{\PW_\p}
\newcommand{\PWpE}{\PWp(\E)}
\newcommand{\PWptaun}{\PWp(\Omega,\taun)}
\newcommand{\PWpe}{\PWp(\e)}
\newcommand{\Jcal}{\mathcal J}
\renewcommand{\d}{\mathbf d}
\newcommand{\dell}{\d_\ell}
\newcommand{\xbf}{\mathbf x}
\newcommand{\xbfE}{\xbf_\E}

\newcommand{\NPWe}{N_{\text{PW}}^\e}
\newcommand{\RE}{\mathbb{RE}}
\newcommand{\deltah}{\delta_\h}

\DeclareMathOperator{\Real}{Re}
\DeclareMathOperator{\Imag}{Im}

\DeclareMathOperator{\diam}{diam}
\newcommand{\xibold}{\boldsymbol{\xi}}
\newcommand{\lbold}{\boldsymbol{\ell}}
\newcommand{\mbold}{\textbf{m}}
\newcommand{\C}{\mathbb{C}}
\newcommand{\R}{\mathbb{R}}
\newcommand{\N}{\mathbb{N}}
\newcommand{\Z}{\mathbb{Z}}
\newcommand{\dbold}{\textup{\textbf{d}}}
\newcommand{\Vn}{\mathcal{V}_n}

\newcommand{\Vnh}{\widehat{\mathcal{V}}_n}

\newcommand{\lambdazero}{\lambda^{(0)}}
\newcommand{\lambdaone}{\lambda^{(1)}}
\newcommand{\lambdatwo}{\lambda^{(2)}}
\newcommand{\vn}{v_n} 
\newcommand{\x}{\textbf{\textup{x}}} 
\newcommand{\unh}{\widehat{u}_n}
\newcommand{\vnh}{\widehat{v}_n}

\newcommand{\un}{u_n} 

\newcommand{\lin}{\textup{span}}
\newcommand{\an}{\a_n} 
\newcommand{\chih}{\widehat{\chi}}

\newcommand{\djj}{\textbf{\textup{d}}_j} 
\newcommand{\PWE}{\mathbb {PW}_\p(\E)} 
\newcommand{\wjE}{w_j^{\E}} 
\newcommand{\Pip}{\PieH} 

\newcommand{\dir}{\textbf{\textup{d}}} 

\newcommand{\gammatraceBound}{\mathsf{tr}_{\partial\Omega}}
\newcommand{\gammatrace}{\mathsf{tr}}
\newcommand{\gammatracee}{\gammatrace^e}
\newcommand{\gammatraceIe}{\gammatrace_I^\e}
\newcommand{\gammatraceIE}{\gammatrace_I^\E}
\newcommand{\betamin}{\beta_{\mathsf{min}}}

\usepackage{amsthm}
\restylefloat{table}
\theoremstyle{plain}
\newtheorem{thm}{Theorem}[section]
\newtheorem{cor}[thm]{Corollary}

\newtheorem{prop}[thm]{Proposition}
\theoremstyle{definition}
\newtheorem{defn}{Definition}[section]

\theoremstyle{remark}
\newtheorem{remark}{Remark}

\begin{document}
\title{The nonconforming Trefftz virtual element method: general setting, applications, and dispersion analysis for the Helmholtz equation}
\author{Lorenzo Mascotto \and Ilaria Perugia \and Alexander Pichler
\thanks{Faculty of Mathematics, University of Vienna, 1090 Vienna, Austria \texttt{lorenzo.mascotto@univie.ac.at}, \texttt{ilaria.perugia@univie.ac.at}, \texttt{alex.pichler@univie.ac.at}}}
\maketitle

\begin{abstract}
We present a survey of the nonconforming Trefftz virtual element method for the Laplace and Helmholtz equations.
For the latter, we present a new abstract analysis, based on weaker assumptions on the stabilization,
and numerical results on the dispersion analysis, including comparison with the plane wave discontinuous Galerkin method.
\end{abstract}



\section{Introduction} \label{section:introduction}
In this chapter, we present a survey of a methodology, which dovetails the nonconforming virtual element setting with the Trefftz paradigm.

The \emph{nonconforming virtual element method} is an extension of the nonconforming finite element method to polytopal meshes, which is based on the virtual element method (VEM) framework.
Notably, the continuity constraint of functions in the approximation spaces are imposed in a weak sense only.
Since its inception~\cite{nonconformingVEMbasic}, the nonconforming VEM has received an increasing attention, and has been analysed and applied to several problems:
general elliptic problems~\cite{cangianimanzinisutton_VEMconformingandnonconforming};
Stokes and Navier-Stokes equations~\cite{CGM_nonconformingStokes, nc_VEM_NavierStokes, liu2017nonconformingStokes, zhao2019divergence, zhao2020nonconforming};
eigenvalue problems~\cite{gardini2019nonconforming}; the plate bending problem~\cite{zhao2016nonconforming};
equations involving the biharmonic and $2m$-th
operators~\cite{VEM_fullync_biharmonic, polyharmonic:VEM:LongChen};
anisotropic error estimates~\cite{chen_anisotropic_nonconforming};
the linear elasticity problem~\cite{zhang2019nonconforming};
parabolic and fractional-reaction subdiffusion problems~\cite{zhao2019nonconforming-parabolic, li2019nonconforming-subdiffusion};
the VEM with a SUPG stabilization for advection-diffusion-reaction~\cite{berrone2018supg}; fourth order singular perturbation problems~\cite{zhang2020nonconforming};
the Kirchhoff plate contact problem~\cite{wang2019conforming};
the medius error analysis for the Poisson and biharmonic problem~\cite{huang2020medius}.
Its comparison with other skeletal methods such as the hybridized discontinuous Galerkin method (HDG) and the hybrid high-order (HHO) method is investigated in~\cite{di2018discontinuous}.

\emph{Trefftz methods} are Galerkin-type methods for the approximation of linear partial differential equations (PDEs) with piecewise constant coefficients,
where the test and/or trial functions belong to the kernel of the differential operator defining the PDE to be approximated.
Trefftz methods have been applied mainly to time-harmonic wave propagation problems, but also to advection-diffusion problems and to wave problems in the time-domain.
Typically, Trefftz methods are obtained by combining these functions (\emph{Trefftz function}) with the discontinuous Galerkin method (dG) or with the partition of unity method (PUM).
Out of the former category, restricting ourselves to the Helmholtz problem, we recall several approaches, which trace back to the ultra weak variational formulation \cite{cessenatdespres_basic}:
 the wave based method \cite{wavebasedmethod_overview}; discontinuous methods based on Lagrange multipliers \cite{farhat2001discontinuous} and on least square formulation \cite{monk1999least};
the plane wave discontinuous Galerkin (PWDG) method \cite{GHP_PWDGFEM_hversion, TDGPW_pversion}; the variational theory of complex rays \cite{riou2008multiscale}; see \cite{PWDE_survey} for an overview of such methods.
We also mention the quasi-Trefftz dG method for the case of smoothly varying coefficients, where functions that ``almost'' belong to the kernel of the operator appearing in the PDE are employed~\cite{DespresImbert2014,ImbertGerard2015}.
Instead, the latter category consists of methods based on
approximation spaces of continuous functions given by the product of
pure Trefftz functions with partition of unity, low order, hat
functions.
Amongst them, we highlight the classical PUM~\cite{BabuskaMelenk_PUMintro,babumelenk_harmonicpolynomials_approx}
and its virtual element version~\cite{Helmholtz-VEM}. 

More recently, the Trefftz gospel has been combined with the nonconforming VEM setting for the Laplace equation~\cite{ncHVEM},
and the Helmholtz equation with constant~\cite{TVEM_Helmholtz, TVEM_Helmholtz_num} and piecewise constant wave number~\cite{fluidfluid_ncVEM}.
Albeit the nonconforming Trefftz VEM is not an $H^1$ conforming
method, the interelement continuity is imposed weakly within the approximation spaces, unlike in the dG setting. Moreover, unlike in the PUM setting, its basis functions are exacly Trefftz.
\medskip 

In this contribution, we review the methods presented in~\cite{ncHVEM, TVEM_Helmholtz, TVEM_Helmholtz_num}, and elaborate a common framework for nonconforming Trefftz VEMs.
We start by considering the simplest case of the Laplace equation in Section~\ref{section:Laplace}.
Then, we extrapolate the core idea of the nonconforming Trefftz VEM approach and extend it to general linear differential operators of the second order; see Section~\ref{section:general-framework}.
In Section~\ref{section:Helmholtz}, we recast the case of the Helmholtz equation studied in~\cite{TVEM_Helmholtz} into the setting of Section~\ref{section:general-framework}.
Additionally, we present a new abstract analysis of the method, which
is based on weaker assumptions on the stabilization than those in~\cite{TVEM_Helmholtz}.
While we refer to~\cite{TVEM_Helmholtz_num} for the implementation details and an extended numerical testing of the nonconforming Trefftz VEM for the Helmholtz problem,
we present in Section~\ref{section:NR-Helmholtz} unpublished work on its numerical dispersion analysis, where the performance of the nonconforming Trefftz VEM are compared to those of the PWDG method that have been studied in~\cite{gittelson}.

\paragraph*{Notation.}
We employ standard notation for Sobolev spaces. Given~$s\in \Nbb$ and
a domain~$\Omega$, we denote the Sobolev space of order~$s$ taking
values in the complex field~$\Cbb$ by~$H^s(\Omega)$.
In the special case~$s=0$, $H^s(\Omega)$ reduces to the Lebesgue space~$L^2(\Omega)$. We introduce the Sobolev sesquilinear forms, seminorms, and norms
\[
(\cdot, \cdot)_{s,\Omega}, \quad\quad\quad \vert \cdot \vert_{s,\Omega}, \quad\quad\quad \Vert \cdot \Vert_{s,\Omega}.
\]
We define Sobolev spaces of order~$s \in \Rbb$ by interpolation. Analogously, we denote the Sobolev spaces on~$\partial \Omega$ by~$H^s(\partial \Omega)$.
If we consider Sobolev spaces of functions taking values only in~$\Rbb$, we employ the same notation~$H^s(\Omega)$ thanks to the trivial embedding~$\Rbb \hookrightarrow \Cbb$.

Assume that the domain~$\Omega$ is Lipschitz. Then, we can define the standard Dirichlet trace operator~$\gammatraceBound: H^s(\Omega) \rightarrow H^{s- \frac{1}{2}}(\partial \Omega)$ for all~$s \in (1/2, 3/2)$.
Thanks to this operator, we are allowed to introduce affine Sobolev
spaces with boundary conditions: given~$g \in H^{\frac{1}{2}}(\partial \Omega)$,
\[
H^1_g(\Omega) := \left\{ v \in H^1(\Omega)  \mid \gammatraceBound(v) = g      \right\}.
\]

Henceforth, as standard in the VEM literature, a quantity is said to be \emph{computable} if it can be evaluated using the degrees of freedom of the trial and test spaces under consideration.

\section{Polygonal meshes and broken Sobolev spaces} \label{section:meshes}

We denote a family of polygonal meshes over a polygonal domain~$\Omega \subset \mathbb R^2$ by~$\{\taun\}_{h>0}$, and the sets of edges and vertices of $\taun$ by~$\En$ and~$\Vcaln$, respectively.
In particular, we split~$\En$ into the sets of boundary and internal
edges~$\EnB$ and~$\EnI$, respectively.
Given a polygon~$\E \in \taun$, we denote its barycenter, size, set of
edges, set of vertices, and outward normal to $\partial K$
by~$\xbfE$, $\hE$, $\EE$, $\VcalE$, and~$\nE$, respectively,
and given an edge~$\e \in \En$, we denote its size by~$\he$.

As customary in polygonal methods, we demand the following
shape-regularity assumption
on~$\{ \taun \}_{\h}$:
\begin{equation}\label{eq:shape-regularity}
  \begin{split}
    &\hspace{-0.5cm}\text{there exists a positive constant~$\gamma > 0$ such that, for all~$\E\in\taun$,}\\
  i)&\text{\ $\E \in \taun$ is star-shaped with respect to a ball of radius~$\gamma  \hE$};\\
 ii)&\text{\ every edge~$\e \in \EE$ is such that~$\he \le \hE \le \gamma \he$}.
\end{split}
\end{equation}
We introduce the broken Sobolev space associated with the mesh~$\taun$
\begin{equation} \label{broken:Sobolev}
H^{1}(\Omega,\taun) := \{ v \in L^2(\Omega) \mid v\in H^1(\E) \, \forall \E \in \taun   \},
\end{equation}
and endow it with the broken seminorm
\begin{equation} \label{broken-seminorm}
\vert v \vert ^2 _{1,\taun} := \sum_{\E \in \taun} \vert v_{\E} \vert_{1,\E} ^2 \quad \quad \forall v \in H^1(\Omega,\taun).
\end{equation}
Given~$\e \in \EhI$, we define the jump operator across~$\e$ as follows:
\begin{equation} \label{jump}
\llbracket v \rrbracket_\e =
\vp{}_{|\e} \ \nEp + \vm{}_{|\e} \ \nEm  	\qquad \text{if } \e\subset \partial \Ep\cap \partial \Em\\
\end{equation}
for all~$v$ in~$H^{1}(\Omega,\taun)$, where~$\vp:=v_{|_{\Ep}}$ and~$\vm:=v_{|_{\Em}}$.

\section{The nonconforming Trefftz virtual element method for the Laplace problem} \label{section:Laplace}

In this section, we focus on the approximation of a two dimensional
Laplace problem by means of the nonconforming Trefftz virtual element method that was originally introduced in~\cite{ncHVEM}; see also~\cite{conformingHarmonicVEM} for its conforming version.
\paragraph*{The continuous problem.}
Let~$\Omega \subset \Rbb^2$ be a polygonal domain and~$\g \in H^{\frac{1}{2}}(\partial \Omega)$.
Introduce the following notation:
\[
\Vg:= H^1_\g(\Omega), \quad\quad \Vz:=H^1_0(\Omega), \quad\quad \a(\cdot,\cdot) := (\nabla \cdot, \nabla\cdot)_{0,\Omega}.
\]
We consider the following Laplace problem: find a sufficiently smooth~$u:\Omega \rightarrow \Rbb$ such that
\[
\begin{cases}
\Delta u = 0 	& \text{in } \Omega \\
u = \g 		& \text{on } \partial \Omega,\\
\end{cases}
\]
which in weak formulation reads
\begin{equation} \label{Laplace:weak}
\begin{cases}
\text{find } u\in \Vg \text{ such that }\\
\a(u,v) = 0 \quad \quad \forall v \in \Vz. \\
\end{cases}
\end{equation}

\paragraph*{An explicit discontinuous space.}
Let~$\p \in \mathbb N$.
Given a sequence~$\{\taun\}_h$ of polygonal decompositions over~$\Omega$ as in Section~\ref{section:introduction},
we define the corresponding sequence of discontinuous, piecewise harmonic polynomials over~$\taun$:
\[
\mathcal S^{0,\Delta}_\p(\Omega,\taun) := \left\{ \qpDelta \in L^2(\Omega) \mid \qpDelta{}_{|\E} \in \mathbb H_\p(\E) \; \forall \E \in \taun      \right\},
\]
where, for all~$\E \in \taun$,
\[
\mathbb H_\p(\E) := \left\{ \qpDelta \in \mathbb P_\p(\E) \mid \Delta \qpDelta= 0   \right\}.
\]

We recall the following approximation property of discontinuous, piecewise harmonic polynomials for harmonic functions; see, e.g., \cite[Theorem~2.9]{melenk1999operator}.
\begin{prop} \label{proposition:best-harmonic}
Under the shape regularity assumption~\eqref{eq:shape-regularity} with constant~$\gamma$, given a harmonic function~$u \in H^{s+1}(\Omega)$, $s>0$, there exists~$\qpDelta \in \mathcal S^{0,\Delta}_\p(\Omega,\taun) $ such that
\[
\vert u - \qpDelta\vert_{1,h} \le c \h^{s} \Vert u \Vert_{s+1,\Omega}.
\]
The positive constant~$c$ depends on~$\gamma$ and on the polynomial degree~$\p$.
\end{prop}
The importance of Proposition~\ref{proposition:best-harmonic} resides in the fact that there exists a subset of the space of piecewise polynomials of degree at most~$\p$ having optimal approximation properties for harmonic functions.
This subset is the space of piecewise harmonic polynomials of degree at most~$\p$, whose local dimension in 2D is~$2\p+1$, while the local
dimension of the space of complete polynomials of degree at most~$\p$ is~$(\p+1)(\p+2)/2$.

\paragraph*{Design of the VE Trefftz space.}
Here, we recall from~\cite{ncHVEM} the definition of local and global nonconforming Trefftz spaces for the Laplace problem.
Given~$\E \in \taun$, define
\begin{equation} \label{local-hVEM-space}
\begin{split}
\VhDeltaE := \{ & \vhDelta \in H^1(\E) \mid \Delta \vhDelta = 0 \text{ in } \E,\\
					&  \forall \e \in \EE\ \exists \qpDelta \in \mathbb H _\p^\Delta(\E)\ \text{s.t.}\ \ne \cdot \nabla \vhDelta{}_{|\e}  = \ne \cdot \nabla \qpDelta {}_{|\e}\}.
\end{split}
\end{equation}
Equivalently, we are requiring that the Neumann traces of functions in~$\VhDeltaE$ belong to~$\mathbb P_{\p-1} (\e)$ for all~$\e \in \EE$.
It is more convenient to define~$\VhDeltaE$ as in~\eqref{local-hVEM-space} in view of the general setting presented in Section~\ref{section:general-framework} below.

The idea behind the definition of~$\VhDeltaE$ is as follows.
According to the Trefftz gospel, we consider a local space, which consists of Trefftz functions, i.e., harmonic functions in our case.
A possible way to pick a finite dimensional subspace~$\VhDeltaE$ is to require that, on each~$\e\in\EE$, a suitable trace of any element in~$\VhDeltaE$ belongs to a suitable explicit finite element space.
In our case, we require that the Neumann traces belong to $\mathbb P_{\p-1}(\e)$, the space of polynomials of degree at most~$\p-1$ ($\dim(\mathbb P_{\p-1}(\e))=p$).
By doing this, harmonic polynomials are included in the space~$\VhDeltaE$, which yields good approximation properties; see Proposition~\ref{proposition:best-interpolation-harmonic} below.

For any edge~$\e \in \EE$, let~$\{ \malphae \}_{\alpha=1}^\p$ be a
basis of~$\mathbb P_{\p-1}(\e)$.
Consider the following set of linear functionals on $\VhDeltaE$:
\begin{equation} \label{polynomial-edge-dof}
\vhDelta \in \VhDeltaE\ \mapsto\ 
\frac{1}{\he} \int_\e \vhDelta \malphae \quad \quad\quad \forall \alpha=1,\dots,\p,\; \forall \e \in \EE.
\end{equation}
\begin{prop} \label{proposition:unisolvency-Laplace}
The set of functionals in~\eqref{polynomial-edge-dof} is a set of unisolvent degrees of freedom.
\end{prop}
\begin{proof}
The proof is standard and can be found, e.g., in \cite[Section~3.1]{ncHVEM}. For the sake of completeness, we recall it here.
The number of the functionals in~\eqref{polynomial-edge-dof} is smaller than or equal to the dimension of~$\VhDeltaE$.
Thus, it suffices to show the unisolvence of such a set of functionals.

Assume that~$\vhDelta\in \VhDeltaE $ has the moments in~\eqref{polynomial-edge-dof} all equal to zero. Then, we have
\[
\int_{\partial \E} \vhDelta = 0.
\]
Consequently, in order to prove the unisolvence, i.e., that~$\vhDelta=0$, it is enough to show that~$\vhDelta$ has zero gradient.
This is a consequence of an integration by parts, and the fact
that~$\vhDelta$ is harmonic, that~$\nE \cdot \nabla \vhDelta {}_{|\e}$
is a polynomial of degree at most~$\p-1$ on each edge~$\e \in \EE$, and that the functionals in~\eqref{polynomial-edge-dof} are zero:
\begin{equation} \label{IBP-harmonic}
  \vert \vhDelta \vert_{1,\E}^2
  =-\int_\E\Delta \vhDelta \,\vhDelta+\int_{\partial\E} \nE \cdot \nabla \vhDelta \, \vhDelta
  = \sum_{\e \in \EE}\int_{\e} \nE \cdot \nabla \vhDelta \, \vhDelta = 0.
\end{equation}
\end{proof}
In the proof of Proposition~\ref{proposition:unisolvency-Laplace}, the choice of the local polynomial traces in the definition of~\eqref{local-hVEM-space} is important.
More precisely, we fixed polynomial Neumann traces, for they appear in the integration by parts~\eqref{IBP-harmonic}.
At the same time, the choice of the ``Dirichlet''-type degrees of freedom in~\eqref{polynomial-edge-dof} is relevant as well, and will play a role in the construction of the global space.
\medskip

We define the infinite dimensional, nonconforming spaces
{\[
\HncpDelta \!\! :=\!\! \left\{ v \in H^{1}(\Omega,\taun)\! \mid
    \! \int_\e \llbracket v \rrbracket_\e \cdot \ne \, \malphae =0 \ \, \forall \alpha=1,\dots,\p,\,  \forall \e \in \EhI  \right\}\!,
\]

where the broken Sobolev spaces~$H^{1}(\Omega,\taun)$ and the jump operator~$\llbracket \cdot \rrbracket$ are defined
in~\eqref{broken:Sobolev} and~\eqref{jump}, respectively, and
we introduce the global nonconforming Trefftz virtual element space for the Laplace problem:
\[
\VhDelta := \left\{  \vhDelta \in \HncpDelta \mid \vhDelta{}_{|\E} \in \VhDeltaE \; \forall \E \in \taun   \right\}.
\]
We obtain the set of global degrees of freedom of the space~$\VhDelta$ by patching the local ones in~\eqref{polynomial-edge-dof}.
In particular, we use the Dirichlet edge moments of~\eqref{polynomial-edge-dof} in the definition of the infinite dimensional, nonconforming space~$\HncpDelta$ in order to weakly impose the interelement continuity.\medskip

We summarize the main features of the space $\VhDelta$, highlighting a ``duality'' between Dirichlet moments and local Neumann traces, as follows.
\begin{center}
  \framebox[\textwidth]{Trefftz spaces $\quad$ {\bf contain} $\quad$ harmonic functions}\par
\framebox[\textwidth]{nonconformity $\quad$
 {\bf imposed through} $\quad$ Dirichlet moments}\par
\framebox[\textwidth]{unis. of DOFs in~\eqref{polynomial-edge-dof} $\quad$ {\bf implied by} $\quad$ pol. Neumann traces in~\eqref{local-hVEM-space}}
\end{center}

For the design of the method, we define spaces that incorporate Dirichlet boundary conditions. More precisely, given~$\g \in H^{\frac{1}{2}}(\partial \Omega)$, we define
\[
\begin{split}
& \VhDeltag := \left\{ \vhDelta \in \VhDelta \mid \int_\e (\vhDelta - \g) \malphae =0 \quad \forall \e \in \EhB,\, \forall \alpha=1,\dots,\p    \right\} .    \\
\end{split}
\]
The seminorm $\vert\cdot \vert_{1,\taun} $ defined in~\eqref{broken-seminorm} is actually a norm in $\VhDeltaz$.

\paragraph*{Interpolation properties.}
An interesting property of the nonconforming Trefftz virtual element
space for the Laplace problem is that, given a harmonic function~$u \in H^1(\Omega)$, there exists~$\uIDelta \in \VhDelta$,
which approximates~$u$ better than any \emph{discontinuous},
  piecewise harmonic
polynomial of degree at most $\p$.
This property was originally shown in~\cite[Proposition~3.1]{ncHVEM}.
\begin{prop} \label{proposition:best-interpolation-harmonic}
Given a harmonic function~$u \in H^1(\Omega)$, there exists~$\uIDelta \in \VhDelta$ such that
\[
\vert u - \uIDelta \vert_{1,\taun} \le \vert u - \qpDelta \vert_{1,\taun} \quad \quad \forall \qpDelta \in \mathcal S^{0,\Delta}_\p(\Omega,\taun) ,
\]
where the broken Sobolev seminorm is defined in~\eqref{broken-seminorm}.
\end{prop}
\begin{proof}
Define~$\uIDelta \in \VhDelta$ as the interpolant of~$u$, i.e.,
\begin{equation} \label{definition-uIDelta}
\int_\e (u-\uIDelta) \malphae=0 \quad \quad \forall \e \in \Eh, \; \forall \alpha=1,\dots,\p,
\end{equation}
and let $\qpDelta$ be any function in $\mathcal
  S^{0,\Delta}_\p(\Omega,\taun)$.
For any $\E \in \taun$, since both $(\nE \cdot \nabla \uIDelta){}_{|_e}$
and $(\nE \cdot \nabla \qpDelta){}_{|_e}$
 belong to $\mathbb P_{\p-1}(\e)$ for all $\e\in\EE$,
 definition~\eqref{definition-uIDelta} implies
 \begin{equation}\label{eq:revised}
\int_e \nE \cdot \nabla \uIDelta (u-\uIDelta)=
 \int_e \nE \cdot \nabla \qpDelta (u-\uIDelta)=0\qquad\forall \e\in\EE.
   \end{equation}
Therefore,
by integrating by parts twice 
and using~\eqref{eq:revised}, as well as $\Delta u=\Delta \uIDelta=\Delta \qpDelta=0$,
we deduce that
\[
\begin{split}
\vert u - \uIDelta \vert_{1,\E}^2
& = -\int_{\E} \underbrace{\Delta (u-\uIDelta)}_{=0} \, (u-\uIDelta) +
\sum_{\e\in\EE} \int_{\e}
\nE \cdot \nabla  (u-\uIDelta)
\, (u-\uIDelta)  \\
& \overset{\eqref{eq:revised}}{=} -\int_{\E} \underbrace{\Delta (u-\qpDelta)}_{=0} \, (u-\uIDelta) +
\sum_{\e\in\EE} \int_{\e}
\nE \cdot \nabla (u-\qpDelta)
\, (u-\uIDelta)  \\
						& = (\nabla(u-\qpDelta), \nabla(u-\uIDelta))_{0,\E} \le \vert u-\qpDelta\vert_{1,\E} \vert u - \uIDelta \vert_{1,\E},
\end{split}
\]
whence the assertion follows.
\end{proof}

\paragraph*{Projections and stabilizations.}
For future convenience, split
\[
\a(u,v) = \sum_{\E \in \taun} \aE(u_{|\E}, v_{|\E}) :=  \sum_{\E \in \taun} \left( \nabla (u_{|\E}), \nabla (v_{|\E}\right))_{0,\E}.
\]
Since the functions in the virtual element space~$\VhDelta$ are not
known in closed form, we cannot compute the local bilinear forms~$\aE(\cdot, \cdot)$ applied to functions in~$\VhDelta$.
Rather, we introduce computable bilinear forms as in the standard virtual element approach of~\cite{VEMvolley}.
\medskip

To this aim, we need two main ingredients. The first one is a local projection into harmonic polynomial spaces.
Define~$\PinablaDeltap : \VhDeltaE \rightarrow \mathbb H_\p(\E)$ as follows:
\begin{equation} \label{harmonic-Pinabla}
\begin{cases}
\aE(\vhDelta - \PinablaDeltap \vhDelta , \qpDelta) = 0\\
\int_{\partial \E} (\vhDelta - \PinablaDeltap \vhDelta)  = 0
\end{cases}
\quad \quad \forall \qpDelta \in \mathbb H_\p(\E),\; \forall
\vhDelta\in \VhDeltaE.
\end{equation}
This is a typical VEM projection. Here, we project into the subspace of harmonic polynomials of degree at most~$\p$ whereas, in the standard setting~\cite{VEMvolley}, the projection is into the full space of polynomials of degree at most~$\p$.

The definition of the degrees of freedom in~\eqref{polynomial-edge-dof} allows us to compute the projector~$\PinablaDeltap$.
This is clear for the second condition in~\eqref{harmonic-Pinabla}.
As for the first condition, we observe that
\[
\aE(\PinablaDeltap\vhDelta, \qpDelta)= \aE(\vhDelta, \qpDelta) = - \int_\E \vhDelta\, \underbrace{\Delta \qpDelta}_{=0} + \sum_{\e\in\EE} \int_\e \vhDelta \, \underbrace{\nE \cdot \nabla \qpDelta}_{\in \mathbb P_{\p-1}(\e)},
\]
where the right-hand side is computable using~\eqref{polynomial-edge-dof}.
\medskip

The second ingredient is a computable stabilization on each element,
which is needed since the bilinear form~$\aE(\cdot, \cdot)$ is not computable on~$\ker(\PinablaDeltap) \times \ker(\PinablaDeltap)$.
More precisely, for all~$\E \in \taun$, let~$\SEDelta: \ker(\PinablaDeltap) \times \ker(\PinablaDeltap) \rightarrow \Rbb$ be a bilinear form
that is computable via the degrees of freedom in~\eqref{polynomial-edge-dof} and that satisfies the following property:
there exist two positive constant~$\alpha_*$ and~$\alpha^*$
independent of the mesh size
such that
\begin{equation} \label{stabilization-harmonic}
\alpha_* \vert \vhDelta \vert_{1,\E}^2 \le \SEDelta (\vhDelta, \vhDelta) \le \alpha^* \vert \vhDelta \vert_{1,\E}^2 \quad \quad \forall \vhDelta \in \ker(\PinablaDeltap).
\end{equation}
We allow~$\alpha_*$ and~$\alpha^*$ to depend on the shape regularity constant~$\gamma$ in~\eqref{eq:shape-regularity}.

Then, we define
\[
\begin{split}
\anDelta (\uhDelta,\vhDelta) 	& := \sum_{\E \in \taun} \anEDelta (\uhDelta{}_{|\E}, \vhDelta{}_{|\E}) \\
						& := \sum_{\E \in \taun} \aE (\PinablaDeltap \uhDelta{}_{|\E}, \PinablaDeltap \vhDelta{}_{|\E}) \\
						& \quad + \SEDelta ( (I-\PinablaDeltap) \uhDelta{}_{|\E} , (I-\PinablaDeltap) \vhDelta{}_{|\E}  ).
\end{split}
\]
As in~\cite{VEMvolley, ncHVEM}, the discrete bilinear form~$\anDelta(\cdot, \cdot)$ is coercive and continuous with constants~$\min(1,\alpha_*) $ and~$\max(1,\alpha^*)$.

\begin{remark} 
We refer to~\cite[Section~3.3]{ncHVEM} for an explicit stabilization satisfying~\eqref{stabilization-harmonic}.
There, stability bounds are proven, which are explicit also in terms of the polynomial degree .
\eremk
\end{remark}

\paragraph*{The method.}
We have introduced all the ingredients needed for the design of the nonconforming Trefftz VEM for the Laplace problem:
\begin{equation} \label{ncHVEM}
\begin{cases}
\text{find } \uhDelta \in \VhDeltag \text{ such that}\\
\anDelta(\uhDelta, \vhDelta) = 0 \quad \quad \forall \vhDelta \in \VhDeltaz.
\end{cases}
\end{equation}
The well-posedness of the method follows from the coercivity and the continuity of the discrete bilinear form~$\anDelta(\cdot, \cdot)$.

\paragraph*{Convergence analysis.}
The abstract error analysis of method~\eqref{ncHVEM} is carried out in~\cite[Theorem~3.3]{ncHVEM} and is based on the second
Strang's lemma.
The result is that the error of the method is controlled by the sum of two terms: the best approximation error in the space of \emph{discontinuous}, piecewise polynomials and a term that
measures the nonconformity of the method. The latter is expressed in terms of the bilinear form~$\Ncalh: H^1(\Omega) \times
\HncpDelta \rightarrow \Rbb$ defined as
\begin{equation} \label{nonconforming-bf}
\Ncalh (u,v) = \sum_{\e \in \Eh} \int_\e \nabla u \cdot \llbracket v \rrbracket.
\end{equation}
\begin{thm}\label{th:abstractLaplace}
Let~$u$ and~$\uhDelta$ be the solutions to~\eqref{Laplace:weak}
and~\eqref{ncHVEM}, respectively.
Under the shape regularity assumption~\eqref{eq:shape-regularity}, the following bound is valid:
\[
\vert u - \uh \vert_{1,\taun} \le \frac{\alpha^*}{\alpha_*} \left\{  6 \inf_{\qpDelta \in \mathcal S^{0,\Delta}_\p(\Omega,\taun)} \vert u -\qpDelta \vert_{1,\taun} + \sup_{0\not=\vhDelta \in \VhDeltaz} \frac{\Ncalh (u,\vhDelta)}{\vert \vhDelta \vert_{1,\taun}}   \right\}.
\]
\end{thm}
As a consequence of Theorem~\ref{th:abstractLaplace}, Proposition~\ref{proposition:best-harmonic}, and estimates of~$\Ncalh$ derived by standard computations that are typical in nonconforming Galerkin methods,
the convergence of the method follows; see~\cite[Section~3.5]{ncHVEM} for more details.
\begin{cor}
Let~$u$ and~$\uhDelta$ be the solutions to~\eqref{Laplace:weak} and~\eqref{ncHVEM}, respectively, with~$u \in H^{s+1}(\Omega)$.
Under the shape regularity assumption~\eqref{eq:shape-regularity} with constant~$\gamma$,
the following convergence result is valid:
\[
\vert u - \uh \vert_{1,\taun} \le c \h ^s \Vert u \Vert_{s+1,\Omega}.
\]
Here, $c$ is a positive constant, which depends on~$\gamma$ and on the polynomial degree~$\p$.
\end{cor}

\emph{Overall, the nonconforming Trefftz VEM for the Laplace problem is a modification of the standard nonconforming VEM, in the sense that it encodes certain properties of the solution to the problem within the definition of the VE spaces.
 The resulting method has significantly fewer degrees of freedom than a standard VEM based on complete polynomial spaces, yet keeping the same convergence properties.}

\section{General structure of nonconforming Trefftz virtual element methods} \label{section:general-framework}

In this section, we pinpoint the structure lying behind the nonconforming Trefftz VEM for the Laplace equation and extend it to a more general and abstract setting.
In particular, given a homogeneous, linear partial differential
equation, we highlight which ideas we can extend to the new setting and which not.
For the sake of presentation, we assume that the solution to the involved partial differential equation has to be sought in~$H^1$-type Sobolev spaces with values in the field of complex numbers~$\Cbb$,
although generalizations to other problems are possible as well.

\paragraph*{The continuous problem.}
Let~$\Omega \subset \Rbb^2$ be a polygonal domain and~$\g \in H^{s}(\partial \Omega)$, where~$s\in \Rbb$.
In typical cases, we have~$s\in \{-1/2, 1/2\}$. 
Let~$\Lcal: H^1(\Omega) \rightarrow H^{-1}(\Omega)$ be a linear differential operator of the second order and~$\gammatraceBound:H^1(\Omega) \rightarrow H^s(\partial \Omega)$, $s$ as above, a trace operator.

Consider the following \emph{abstract} problem: find  $u:\Omega \rightarrow \Cbb$  such that
\[
\begin{cases}
\Lcal u = 0 				& \text{in } \Omega \\
\gammatraceBound (u)=\g 	& \text{on } \partial \Omega,\\
\end{cases}
\]
which, in weak formulation, reads
\begin{equation} \label{abstract:problem}
\begin{cases}
\text{find } u\in \V \text{ such that }\\
\a(u,v) + \b(u,v)= \G(v) \quad \quad \forall v \in \W. \\
\end{cases}
\end{equation}
Here, we have introduced an affine space~$\V \subseteq H^1(\Omega)$, a test space~$\W \subseteq H^1(\Omega)$,
sesquilinear forms~$\a: \V \times \W \rightarrow \Cbb$ and~$\b: \V \times \W \rightarrow \Cbb$, and an antilinear functional~$\G: \W \rightarrow \Cbb$.  
The form~$\b(\cdot, \cdot)$ and the functional~$\G(\cdot)$ accommodate the treatment of several types of boundary conditions.
In particular, they are defined only on~$\partial \Omega$.
In what follows, we assume that~$\a(\cdot,\cdot)$, $\Lcal$, and~$\gammatraceBound$ are related by the following identity: for sufficiently smooth~$u$ and~$v$,
\begin{equation} \label{assumption:abstract}
\a(u,v) = -(\Lcal u, v)_{0,\Omega} + \G(v) - \b(u,v).
\end{equation}
After splitting
\[
\a(u,v) = \sum_{\E \in \taun} \aE(u{}_{|\E},v{}_{|\E}),
\]
we demand that, for all~$\E \in \taun$ and all sufficiently smooth~$u$ and~$v$,
{\begin{equation} \label{assumption:abstract:local}
\aE(u{}_{|\E},v{}_{|\E}) \!=\! -(\Lcal u{}_{|\E}, v{}_{|\E})_{0,\E} \!+\!\! \sum_{\e\in \EE}(\gammatracee(u{}_{|\E}), v{}_{|\E})_{0,\e} \!-\bE(u{}_{|\E}, v{}_{|\E}),
\end{equation}
where~$\gammatracee$ denotes the restriction
to an edge~$\e \in \Eh$
of a trace operator~$\gammatrace$, which is not necessarily of the same type as $\gammatraceBound$, 
and where~$\bE(\cdot, \cdot)$ is a local sequilinear form.

For example, in Section~\ref{section:Laplace}, we fixed
\[
\begin{split}
& \V = H^1_\g(\Omega),\qquad \W = H^1_0(\Omega), \qquad \Lcal = \Delta,\qquad \gammatrace = \text{Dirichlet trace operator} ,\\
& \a(u,v) = (\nabla u, \nabla v)_{0,\Omega}, \qquad \b(u,v)=0, \qquad \G(v) = 0.
\end{split}
\]
In the abstract formulation~\eqref{abstract:problem}, we can deal with the boundary datum~$\g$ by tuning either the trial space~$\V$ or the right-hand side~$\G(v)$.

\paragraph*{An explicit discontinuous space.}
The basic tool in the construction of a Trefftz VEM is the
existence of finite dimensional space consisting of globally discontinuous, piecewise smooth functions,
which lie in the kernel of the operator~$\Lcal$ and possess suitable approximation properties for solutions to problem~\eqref{abstract:problem}.
We denote such approximation space by~$\ShzLcal(\Omega,\taun)$ and its local counterpart on every element~$\E$ by~$\ShLcal(\E)$,
where the index $\p\in\mathbb N$ is related to the local space dimension.

For instance, in Section~\eqref{section:Laplace}, we considered
as~$\ShzLcal(\Omega,\taun)$ the space of \emph{discontinuous}, piecewise
harmonic polynomials of degree at most~$\p$, which has optimal approximation properties in terms of the mesh size; see Proposition~\ref{proposition:best-harmonic}.

\paragraph*{Design of the VE Trefftz space.}
Given~$\E \in \taun$, we define the local Trefftz virtual element space on~$\E$ as follows:
\begin{equation} \label{local-abstract-VEM-space}
\begin{split}
\VhLcalE := \{ \vhLcal & \in H^1(\E) \mid  \; \Lcal (\vhLcal) = 0 \text{ in } \E,\\
& \forall \e \in \EE\ \exists \shLcal \in \ShLcal (\E)\ \text{s.t.}\  \gammatracee(\vhLcal{}_{|\e})  =  \gammatracee (\shLcal{}_{|\e}) \}.
\end{split}
\end{equation}
The idea behind the construction of~$\VhLcalE$ hinges upon the existence of an infinite dimensional, local space, which consists of functions in the kernel of the operator~$\Lcal$ (Trefftz space).
We define the finite dimensional subspace $\VhLcalE$ by requiring that, on each $e\in\EE$, a suitable trace belongs to a suitable explicit finite element space having good approximation properties for functions in the kernel of $\Lcal$.
More precisely, we require that the trace $\gammatracee$ on each edge~$\e \in \EE$
of any function in~$\VhLcalE$ belongs to~$\gammatracee(\ShLcal(\E))$.
In this way, we include the functions in~$\ShLcal(\E)$ within the space~$\VhLcalE$.
The hope is that this will yield good interpolation properties of the
local virtual element space.

In the setting of Section~\ref{section:Laplace}, we obtained a local space of harmonic functions, whose Neumann trace~$\gammatracee$
belongs to the space of Neumann traces of harmonic polynomials on every edge~$\e$, namely to~$\mathbb P_{\p-1}(\e)$.

As for the degrees of freedom, for all edges~$\e \in \EE$, let~$\{ \malphae \}_{\alpha=1}^{\pLcal}$ be a basis of~$\gammatrace(\ShLcal(\E))$.
Consider the following set of antilinear functionals on $\VhLcalE$:
\begin{equation} \label{abstract-edge-dof}
\vhLcal \in \VhLcalE\ \mapsto\  c(\he) \int_\e \vhLcal \overline\malphae \quad \quad \forall \alpha=1,\dots,\pLcal,\; \forall \e \in \EE,
\end{equation}
where~$c(\he)$ is a constant depending only on~$\he$ and providing a suitable scaling of the degrees of freedom.

We aim at getting the following result. 
\begin{prop} \label{proposition:unisolvency-abstract}
The set of functionals in~\eqref{abstract-edge-dof} is a set of unisolvent degrees of freedom.
\end{prop}
The number of the functionals in~\eqref{abstract-edge-dof} is smaller than or equal to the dimension of~$\VhLcalE$.
Thus, in order to prove
Proposition~\ref{proposition:unisolvency-abstract}, it suffices to
show the unisolvence of such a set of functionals.

Using assumption~\eqref{assumption:abstract:local}, we get
\[
\aE(\vhLcal,\vhLcal) + \bE(\vhLcal,\vhLcal) = -(\underbrace{\Lcal \vhLcal}_{=0}, \vhLcal)_{0,\E} + \sum_{\e \in \EE}  (\underbrace{\gammatracee(\vhLcal)}_{\in \gammatracee (\ShLcal(\E))}, \vhLcal)_{0,\e} = 0.
\]
In general, this is not enough to prove the unisolvence of the DOFs.
In case of the Dirichlet-Laplace problem, this is indeed sufficient; see Proposition~\ref{proposition:unisolvency-Laplace}.
Notwithstanding, in the case of the Helmholtz problem in Section~\ref{section:Helmholtz} below, we also need assumptions on the
size of the mesh elements; see Proposition~\ref{proposition:unisolvency-Helmholtz}.

Importantly, while the definition of the local spaces is problem-dependent, as it depends on the elliptic operator and suitable traces associated with the problem under consideration,
the choice of the degrees of freedom is fixed, and always consists of
(suitably scaled) Dirichlet moments; see~\eqref{abstract-edge-dof}.
\medskip

Next, we construct global nonconforming VE Trefftz spaces for problem~\eqref{abstract:problem}.
We define the infinite dimensional, nonconforming spaces
\[
\HncpLcal  \!\!:=\!\!\left\{ v \in H^{1}(\Omega,\taun) \!\mid\!\!
    \int_\e \llbracket v \rrbracket_\e \cdot \ne \, \overline\malphae
    =0 \ \, \forall \alpha=1,\dots,\pLcal, \,  \forall\e \in \EhI \right\}\!,
\]
where~$H^{1}(\Omega,\taun)$ and~$\llbracket \cdot \rrbracket$ are defined in~\eqref{broken:Sobolev} and~\eqref{jump}, respectively.
Then, the global nonconforming Trefftz virtual element space for
problem~\eqref{abstract:problem} is defined as
\[
\VhLcal := \left\{  \vhLcal \in \HncpLcal \mid \vhLcal{}_{|\E} \in \VhLcalE \; \forall \E \in \taun   \right\}.
\]
We obtain the set of global degrees of freedom of the space~$\VhLcal$ by patching the local ones in~\eqref{abstract-edge-dof}.
In particular, we use the Dirichlet edge moments in the definition of the infinite dimensional, nonconforming space~$\HncpLcal$ in order to weakly impose the interelement continuity.
\medskip

We can summarize the Trefftz feature of the space $\VhLcal$ and the ``duality'' between Dirichlet moments and the local $\gammatracee$-type trace as follows.

\begin{center}
  \framebox[\textwidth]{Trefftz spaces $\quad$ {\bf contain} $\quad$ functions in~$\ker(\Lcal)$}\par
\framebox[\textwidth]{nonconformity $\quad$ {\bf imposed through} $\quad$ Dirichlet moments}\par
\framebox[\textwidth]{unis. of DOFs  in~\eqref{abstract-edge-dof} $\quad$ {\bf implied by} $\quad$ traces of the type~$\gammatracee$ in~\eqref{local-abstract-VEM-space}\qquad\quad}
\end{center}
As already mentioned, we incorporate the boundary conditions within the method either by enforcing them in the trial space or by suitably tuning the functional~$G(\cdot)$ on the right-hand side.

\paragraph*{Interpolation properties.}
A desirable property of the local space~$\VhDeltaE$ is that the following result is valid.

\begin{prop} \label{proposition:best-interpolation-abstract}
Given a function~$u \in H^1(\Omega)$ in the kernel of~$\Lcal$, there exists~$\uILcal \in \VhLcal$ such that
\[
\vert u - \uILcal \vert_{1,\taun} \le c \vert u - \shLcal \vert_{1,\taun} \quad \quad \forall \shLcal \in \ShzLcal(\Omega,\taun),
\]
where the broken Sobolev seminorm is defined in~\eqref{broken-seminorm} and $c$ is a positive constant independent of the mesh size.
\end{prop}
In particular, we wish that functions in the kernel of $\Lcal$ can be
approximated in~$\VhLcal$
not worse than
in the explicit space~$\ShzLcal(\Omega,\taun)$.
For the Laplace problem the constant is~$1$; see Proposition~\ref{proposition:best-interpolation-harmonic}.
Moreover, we expect that~$\uI$ in
Proposition~\ref{proposition:best-interpolation-abstract} can be
defined as the interpolant of~$u$ through the degrees of freedom in~\eqref{abstract-edge-dof}.

In what follows, we assume that the local sesquilinear forms~$\bE(\cdot, \cdot)$ appearing in~\eqref{assumption:abstract:local} satisfy
\begin{equation} \label{computability:bE}
\begin{split}
& \text{whenever}\ \shLcal\in \ShLcal(K),\, \vhLcal\in \VhLcal(K)\\
& \text{ then } \quad \bE(\shLcal, \vhLcal) \text{ and } \bE(\vhLcal, \shLcal) \quad\text{are computable.}
\end{split}
\end{equation}

\paragraph*{Projections and stabilizations.}
Recall the splitting
\[
\a(u,v) = \sum_{\E \in \taun} \aE(u_{|\E}, v_{|\E}).
\]
We consider the following discretizations of~$\a$ and~$\aE$:
\[
\begin{split}
& \anLcal (\uhLcal, \vhLcal)  := \sum_{\E \in \taun} \anELcal(\uhLcal{}_{|\E}, \vhLcal{}_{|\E}) \\
& := \!\!\sum_{\E \in \taun} \!\!\aE (\PiLcalp \uhLcal{}_{|\E}, \PiLcalp \vhLcal{}_{|\E}) \!+\! \SELcal ( (I\!-\! \PiLcalp) \uhLcal{}_{|\E}, (I\!-\!\PiLcalp) \vhLcal{}_{|\E}  ),
\end{split}
\]
where we have to define the projector~$\PiLcalp$ and the sequilinear form~$\SELcal(\cdot, \cdot)$.

The operator~$\PiLcalp : \VhLcal(K)\to \ShLcal(K)$ is the projection operator with respect to the local sesquilinear form~$\aE(\cdot, \cdot)$. More precisely, for all~$\E\in \taun$, we set
\[
\begin{cases}
\aE(\PiLcalp \vhLcal -  \vhLcal ,   \shLcal)  = 0 \quad \forall \vhLcal \in \VhLcalE,\, \forall \shLcal \in \ShLcal(\E)\\
\text{+ computable conditions for the uniqueness of $\PiLcalp \vhLcal$}.
\end{cases}
\]
In the nonconforming Trefftz VEM for the Laplace equation, the computable condition for uniqueness was on the average on~$\partial\E$; see~\eqref{harmonic-Pinabla}.

The computability of $\PiLcalp$ follows from the definitions of the local spaces~$\VhLcalE$ and the degrees of freedom in~\eqref{abstract-edge-dof}, and from the property~\eqref{computability:bE}:
\[
\begin{split}
\aE(\shLcal, \PiLcalp \vhLcal) 
& = \aE(\shLcal, \vhLcal) \\
& = -( \underbrace{\Lcal \shLcal}_{=0}, \vhLcal)_{0,\E} + \sum_{\e \in \EE}  (\underbrace{\gammatracee(\shLcal)}_{\in \gammatracee (\ShLcal(\E))}, \vhLcal)_{0,\e} -\underbrace{\bE(\shLcal, \vhLcal)}_{\eqref{computability:bE}}.
\end{split}
\]}}

The choice of the degrees of freedom in~\eqref{abstract-edge-dof} allows us to compute the $L^2$-edge projector $\PieLcal: \VhLcalE{}_{|\e} \rightarrow \gammatracee (\ShLcal(\E){}_{|\e})$, which is defined as
\[
(\PieLcal \vhLcal{}_{|\e} -  \vhLcal{}_{|\e}, \gammatracee(\shLcal{}_{|\e}))_{0,\e}   \quad \quad \forall \e \in \EE, \; \forall \vhLcal\in\VhLcal,\; \forall \shLcal \in \ShLcal(\Omega, \taun).
\]
We need this projector for the discretization of the boundary terms, i.e., the sesquilinear form~$\b(\cdot, \cdot)$ and the right-hand side~$\G(\cdot)$ appearing in~\eqref{abstract:problem}.
Such terms do not appear in the Dirichlet-Laplace setting of Section~\ref{section:Laplace}.
They would appear in the case of the Laplace problem with inhomogenous Neumann boundary conditions.

For future convenience, we introduce the approximations
\[
\bnLcal(\uhLcal,\vhLcal) \approx \b(\uhLcal,\vhLcal) \quad \quad \GnLcal(\vhLcal) \approx \G(\vhLcal).
\]
As for the stabilization~$\SELcal$, we require it to satisfy two properties: it has to be computable via the DOFs and it must lead to a well-posed problem.
For instance, in Section~\ref{section:Laplace}, we considered stabilizations leading to coercive and continuous discrete sesquilinear forms.
This is not necessary in all situations. We will see in Section~\ref{section:Helmholtz} below that, in the Helmholtz case, the
coercivity is not required.

\paragraph*{The method.}
With $\subVhLcal\subseteq \VhLcal$ and $\subWhLcal\subseteq \WhLcal$, which may or not contain information about the boundary conditions depending on the choice of~$V$, $W$, $\b$, and~$\G$ in~\eqref{abstract:problem}, 
the nonconforming Trefftz VEM for problem~\eqref{abstract:problem} reads
\begin{equation} \label{ncVEM:abstract}
\begin{cases}
\text{find } \uhLcal \in \subVhLcal \text{ such that}\\
\anLcal(\uhLcal, \vhLcal) + \bnLcal(\uhLcal, \vhLcal) = \GnLcal(\vhLcal) \quad \quad \forall \vhLcal \in \subWhLcal.
\end{cases}
\end{equation}
The well-posedness of the method relies on suitable properties of the stabilization~$\SELcal (\cdot, \cdot)$.

\paragraph*{Convergence analysis.}
Here, we state the Strang-type result that would be the target of the error analysis for method~\eqref{ncVEM:abstract}.
\begin{thm} \label{theorem:abstract-abstract}
Let~$u$ and~$\uhLcal$ be the solutions to~\eqref{abstract:problem} and~\eqref{ncVEM:abstract}, respectively.
Under the shape regularity assumption~\eqref{eq:shape-regularity}, the following bound is valid:
\[
\vert u - \uhLcal \vert_{1,\taun} \le c(\SELcal) \left\{ A + B + C   \right\},
\]
where
\begin{itemize}\setlength\itemsep{0.1em}
\item $c(\SELcal)$ is a constant possibly depending on the stabilization~$\SELcal$;
\item $A$ is the best approximation of~$u$ in the explicit space~$\ShzLcal(\Omega,\taun)$, i.e.,
\[
\inf_{\shLcal \in \ShzLcal(\Omega,\taun)} \Vert u -\shLcal  \Vert_{\text{NORM}},
\]
where~$\Vert \cdot \Vert_{\text{NORM}}$ is a suitable norm;
\item $B$ is a term addressing the nonconformity of the global space; 
\item $C$ is a term involving the approximation of the boundary terms.
\end{itemize}
\end{thm}
In the light of Theorem~\ref{theorem:abstract-abstract}, we could deduce an optimal convergence result from best approximation estimates in~$\ShzLcal(\Omega,\taun)$, and bounds on the nonconformity at interior and boundary edges.

In the setting of the nonconforming Trefftz VEM for the Laplace problem, we had~$C=0$ and
\[
c(\SELcal) = \frac{\alpha^*}{\alpha_*},\quad A= \inf_{\qpDelta \in \mathcal S^{0,\Delta}_\p(\Omega,\taun)} \vert u -\qpDelta \vert_{1,\taun},
\quad B = \sup_{ 0\not= \vhDelta \in \VhDeltaz} \frac{\Ncalh
    (u,\vhDelta)}{\vert \vhDelta \vert_{1,\taun}}
\]
where~$\Ncalh (u,v) $ is defined in~\eqref{nonconforming-bf}.

An important tool in the proof of Theorem~\ref{theorem:abstract-abstract} is Proposition~\ref{proposition:best-interpolation-abstract}, which allows us to absorb a best approximation term in~$ \mathcal S^{0,\Delta}_\p(\Omega,\taun)$ in the term~$A$.

\paragraph*{Common and problem-related features.}
We conclude this section with a summary of the common features in nonconforming Trefftz VEM and the differences depending on the problem under consideration.
\medskip

\textbf{Common features:}
\begin{itemize}\setlength\itemsep{0.1em}
\item in the definition of the local spaces, the existence of an underlying finite dimensional Trefftz space and the characterization through a given edge trace;
\item the definition of Dirichlet-type degrees of freedom;
\item the fact that we can control best interpolation errors in VEM spaces by best approximation errors in explicit discontinuous spaces.
\end{itemize}

\textbf{Problem-related features:}
\begin{itemize}\setlength\itemsep{0.1em}
\item in the definition of the local spaces, the kind of trace used in the characterization;
\item how to prove the unisolvence of the DOFs (additional assumptions might be needed);
\item required properties on the stabilization form;
\item the definition and the well-posedness of the projections;
\item the imposition of the boundary conditions.
\end{itemize}

\section{The nonconforming Trefftz virtual element method for the Helmholtz problem} \label{section:Helmholtz}

In this section, according to the framework established in Section~\ref{section:general-framework}, we describe the construction and the main steps of the analysis of a nonconforming Trefftz VEM for the Helmholtz equation.
We follow the framework of~\cite{TVEM_Helmholtz, TVEM_Helmholtz_num}. However, we propose a slightly different analysis, based on milder assumptions on the stabilization form.

\paragraph*{The continuous problem.}
Let~$\Omega \subset \Rbb^2$ be a polygonal domain, $\g \in H^{-\frac{1}{2}}(\partial \Omega)$, and~$\k >0$.
Introduce the following space of complex-valued functions and the following sesquilinear forms:
\[
V:= H^1(\Omega), \quad\quad \a(\cdot,\cdot) := (\nabla \cdot, \nabla\cdot)_{0,\Omega} - \k^2(\cdot ,\cdot)_{0,\Omega}, \quad \quad \b(\cdot,\cdot) := \i\k (\cdot ,\cdot)_{0,\partial \Omega}.
\]

We consider the following Helmholtz problem endowed with impedance boundary conditions: find a sufficiently smooth~$u:\Omega \rightarrow \mathbb C$ such that
\[
\begin{cases}
\Delta u + \k^2 u = 0 					& \text{in } \Omega\\
\i\k u + \nOmega \cdot \nabla u = \g 		& \text{on } \partial \Omega,\\
\end{cases}
\]
which in weak formulation reads
\begin{equation} \label{Helmholtz:weak}
\begin{cases}
\text{find } u\in V \text{ such that }\\
\a(u,v) + \b(u,v) = (\g,v)_{0,\partial \Omega} \quad \quad \forall v \in V. \\
\end{cases}
\end{equation}
Observe that~\eqref{Helmholtz:weak} falls in the broader abstract setting~\eqref{abstract:problem}.

\paragraph*{An explicit discontinuous space.}
Let~$\p \in \mathbb N$. Given~$\{\taun\}_\h$ a sequence of polygonal decompositions over~$\Omega$ as in Section~\ref{section:introduction},
we introduce the corresponding sequence of piecewise plane waves over~$\taun$:
\[
\PWptaun := \left\{ \wp\in L^2(\Omega) \mid \wp{}_{|\E} \in \PWpE \; \forall \E \in \taun      \right\},
\]
where, for all~$\E \in \taun$, the local space of plane waves~$\PWpE$ is constructed as follows.

Introduce the set of indices $\Jcal:= \{ 1,\dots, 2\p+1 \}$ and the set of pairwise different and normalized directions~$\{ \dell  \}_{\ell \in \Jcal}$.
In each~$\E \in \taun$, consider the set of plane waves
\begin{equation} \label{plane-wave-and-directions}
\well(\xbf) := e^{\i\k\dell \cdot (\xbf-\xbfE)} \quad \quad \quad \forall \ell \in \Jcal,\, \forall \xbf \in \E,
\end{equation}
and define
\[
\PWpE := \text{span} \{ \well, \ell \in \Jcal  \}.
\]
These plane waves belong to the kernel of the Helmholtz operator, i.e.,
\[
\Delta \wp + \k^2 \wp = 0 \quad \quad \quad \forall \wp \in \PWpE, \, \forall \E \in \taun.
\]
Introduce the weighted broken norms and seminorms
\[
\begin{split}
& \vert \cdot \vert_{s,\taun} ^2:= \sum_{\E \in \taun} \vert \cdot \vert_{s,\E}^2,\quad \  \Vert \cdot \Vert_{s,\k,\taun} ^2:= \sum_{\E   \in \taun} \Vert \cdot \Vert_{s,\k,\E}^2 \\
& \text{with}\ \ \Vert \cdot \Vert_{s,\k,\E} ^2:= \sum_{j=0}^s \k^{2(s-j)}  \vert \cdot \vert_{j,\E}^2\ \ \forall \E \in \taun.
\end{split}
\]
For future convenience, we demand that the directions~$\dell$ are uniformly separated. More precisely, we ask that
\begin{equation}\label{eq:ass_directions}
  \begin{split}
& \text{there exists~$\delta \in (0,1]$ such that the angle between $\d_{\ell_1}$ and $\d_{\ell_2}$}\\
&\text{is larger than or equal to $\delta (2 \pi/\p)$ for every~$\ell_1,\ell_2 \in \Jcal$, $\ell_1\not=\ell_2$}.
\end{split}
\end{equation}
This assumption allows us to recall the following approximation property of discontinuous, piecewise plane waves for functions in the kernel of the Helmholtz operator; see, e.g., \cite[Theorem~5.2]{moiola2011ZAMP}.
\begin{prop} \label{proposition:best-PW}
Let~$u \in H^{s+1}(\Omega)$, $s>0$, belong to the kernel of the Helmholtz operator.
Under the shape-regularity assumption~\eqref{eq:shape-regularity} with constant~$\gamma$ and assumption~\eqref{eq:ass_directions} on the directions $\dell$,
for all~$L\in \Rbb$ with~$1\le L \le \min(\p,s)$, there exists~$\wp \in \PWptaun$ such that the following estimate is valid: for every~$0\le j\le L$,
\[
\Vert u - \wp \Vert_{j,\k,\taun} \le c_{pw}(\h \k) \h^{L+1-j} \Vert u \Vert_{L+1,\k,\Omega},
\]
where
\[
c_{pw}(t) := C e^{b\, t} (1+t^{j+q+8}), \quad \quad \b,\,C \in \Rbb.
\]
The constant $C > 0$ depends on $\p$, $j$, $L$, $\gamma$,
and the directions~$\{ \dell \}$, but is independent of~$\k$, $\h$, and~$u$.
On the other hand, the constant~$b$ depends on the geometric properties of the mesh only.
Observe that~$c_{pw}(\h \k)$ remains bounded as~$\h \to 0$.
\end{prop}
The importance of Proposition~\ref{proposition:best-PW} resides in the
fact that there exists a finite dimensional space, whose dimension is lower than that of the piecewise polynomial space of degree at most~$\p$, locally~$2\p+1$
instead of~$(\p+1)(\p+2)/2$ in 2D, but with the same approximation rates for functions in the kernel of the Helmholtz operator.

\paragraph*{Design of the VE Trefftz space.}
Here, we recall from~\cite{TVEM_Helmholtz} the definition of local and global nonconforming Trefftz spaces for the Helmholtz problem.
Given~$\E \in \taun$, for all~$\e \in \EE$, introduce the space
\[
  \PWpe := \{  \welle   \mid \welle = \wellE{}_{|\e} \text{ for some } \wellE \in \PWpE    \}.
\]
We have that~$\dim(\PWpe) \le \dim(\PWpE)$.
More precisely, the dimension of~$\PWpe$ gets smaller whenever the restrictions to~$\e$ of two basis functions of~$\PWpE$ coincide.
Below, we use the notation~$\NPWe := \dim(\PWpe)$.

Given~$\E \in \taun$ and~$\e\in \EE$, introduce the local impedance trace operators
\[
\gammatraceIE (v) := \i\k v + \nE \cdot \nabla v , \quad \gammatraceIe (v) :=  \i\k v_{|_{\e}} + \ne \cdot (\nabla v)_{|_{\e}}   \quad \quad \forall v \in H^1(\E),
\]
and define the local space
\begin{equation} \label{local-Helmholtz-VEM-space}
\begin{split}
\VhHE := \{ \vhH 
& \in H^1(\E) \mid \Delta \vhH + \k^2 \vhH = 0 \text{ in } \E, \\
&  \forall \e \in \EE\ \exists \wellE \in \PWpE\ \text{s.t.}\ 
\gammatraceIe(\vhH) =  \gammatraceIe(\wellE) \}.
\end{split}
\end{equation}
Equivalently, we are requiring that the impedance trace of functions in~$\VhHE$ belongs to~$\PWpe$ for all~$\e \in \EE$.

The idea behind the definition of~$\VhHE$ is exactly the one described in Section~\ref{section:general-framework} after formula~\eqref{local-abstract-VEM-space}, with $\gammatracee=\gammatraceIe$.
The inclusion of $\PWpE$ within $\VhHE$ yields good interpolation properties of the space; see Proposition~\ref{proposition:best-interpolation-Helmholtz} below.

For all edges~$\e \in \EE$, let $\{ \wpe \}_{\alpha=1}^{\NPWe}$ be a basis of~$\PWpe$.
Consider the following set of antilinear functionals on $\VhHE$:
\begin{equation} \label{PW-edge-dof}
  \vhH\in \VhHE\ \mapsto\ \frac{1}{\he} \int_\e \vhH \overline{\wpe} \quad \quad \forall \alpha=1,\dots,\NPWe,\; \forall \e \in \EE.
\end{equation}

So far, the construction falls in the abstract setting detailed in
Section~\ref{section:general-framework}.
The first big difference with respect to the case of the Laplace problem in Section~\ref{section:Laplace} is in the proof of the
unisolvence of the degrees of freedom, which requires the following additional assumption: for all~$\E \in \taun$,
\begin{equation}\label{eq:additional_assumption}
\text{$\k$ is such that~$\k^2$ is not a Dirichlet-Laplace eigenvalue on~$\E$.}
\end{equation}
As discussed, e.g., in~\cite[Section~3.1]{TVEM_Helmholtz}, \text{the} condition~\eqref{eq:additional_assumption} boils down to a threshold condition on the mesh size.

\begin{prop} \label{proposition:unisolvency-Helmholtz}
Under assumption~\eqref{eq:additional_assumption}, the set of functionals~\eqref{PW-edge-dof} is a set of unisolvent degrees of freedom.
\end{prop}
\begin{proof}
The proof can be found, e.g., in \cite[Lemma~3.1]{TVEM_Helmholtz}. For the sake of completeness, we recall it here.
The number of the functionals in~\eqref{PW-edge-dof} is smaller than or equal to the dimension of~$\VhDeltaE$.
Thus, it suffices to show the unisolvence of such a set of functionals.

Let~$\vhH \in \VhHE$ be such that the functionals~\eqref{PW-edge-dof} are zero in $\vhH$. An integration by parts and the properties of functions in~$\VhHE$ yield
\[ 
\begin{split}
\vert \vhH \vert_{1,\E}^2 
& - \k^2 \Vert \vhH \Vert_{0,\E}^2 - \i\k \Vert \vhH \Vert_{0, \partial \E}^2 \\
& = \underbrace{\int_\E \vhH \overline{(-\Delta \vhH - \k^2 \vhH)}}_{=0} +  \sum_{\e \in \EE} \int_\e \vhH \overline{\underbrace{\gammatraceIE (\vhH)}_{\in \PWpe}}  = 0.
\end{split}
\]
By taking the imaginary part on both sides, we deduce that~$\vhH$ has
zero trace on~$\partial \E$. Since $\vhH$ also satisfies $\Delta  \vhH+\k^2 \vhH=0$, assumption~\eqref{eq:additional_assumption} implies that $\vhH=0$, whence the unisolvence follows.
\end{proof}

In the present  case, the local forms $\bE$ in~\eqref{assumption:abstract:local} of the abstract setting of Section~\ref{section:general-framework} are given by
\begin{equation*}
\bE(\uhH, \vhH) := \i \k \int_{\partial \E} \uhH \overline{\vhH}.
\end{equation*}
The forms~$\bE(\cdot,\cdot)$ fulfil assumption~\eqref{computability:bE} for all~$\E \in \taun$.

Next, we construct global nonconforming Trefftz VE spaces for problem~\eqref{Helmholtz:weak}.
Recalling the definition of the broken Sobolev
spaces~$H^{1}(\Omega,\taun)$ and of the jump operator~$\llbracket \cdot \rrbracket$
in~\eqref{broken:Sobolev} and~\eqref{jump}, respectively, and
that~$\{ \wpe \}_{\alpha=1}^{\NPWe}$ denotes a basis of~$\PWpe$,
we set
{\[
\HncpH  \!\!:=\!\! \left\{ \!v \!\in\! H^{1}(\Omega,\taun) \!\mid\!\!
    \int_\e \llbracket v \rrbracket_\e \!\cdot\! \ne \, \overline{\wpe} =0
    \ \, \forall \alpha\!=\!1,\dots, \NPWe,  \forall \e \in \EhI \right\}\!.
\]
Then, we define the global nonconforming Trefftz virtual element space for the problem~\eqref{Helmholtz:weak} as
\[
\VhH := \left\{  \vhH \in \HncpH \mid \vhH{}_{|\E} \in \VhHE \; \forall \E \in \taun   \right\}.
\]
We obtain the set of global degrees of freedom of the space~$\VhH$ by patching the local ones in~\eqref{PW-edge-dof}.
In particular, we use the Dirichlet edge moments in the definition of the infinite dimensional, nonconforming space~$\HncpH$ in order to weakly impose the interelement continuity.
\medskip

As in Section~\ref{section:general-framework}, we summarize the features of the space $\VhHE$, including
the ``duality'' between Dirichlet moments and the local impedance traces~$\gammatraceIE$ as follows.
\begin{center}
  \framebox[\textwidth]{\qquad\qquad\qquad\qquad
    Trefftz spaces $\quad$ {\bf contain} $\quad$
    functions in $\ker(\Delta + \k^2)$
  }\par
\framebox[\textwidth]{nonconformity $\quad$ {\bf imposed through} $\quad$ Dirichlet moments}\par
\framebox[\textwidth]{\; unis. of DOFs in~\eqref{abstract-edge-dof} $\quad$ {\bf implied by} $\quad$
  traces of the type $\gammatraceIE$
  in~\eqref{local-Helmholtz-VEM-space}\qquad\quad}
\end{center}
Differently from the Laplace case, the boundary conditions are incorporated within the weak formulation of the problem, and not in the trial and test spaces.

\paragraph*{Interpolation properties.}
Similarly to the Laplace case, we can approximate any target function~$u \in H^1(\Omega)$ by functions in the space~$\VhH$ better than by functions in the space of discontinuous, piecewise plane waves.
This was shown in~\cite[Theorem~4.2]{TVEM_Helmholtz}.
\begin{prop} \label{proposition:best-interpolation-Helmholtz}
Let assumptions~\eqref{eq:shape-regularity}, \eqref{eq:ass_directions}, and~\eqref{eq:additional_assumption} be valid.
Moreover, let~$\h \k$ be sufficiently small; see~\cite[equation~(4.17)]{TVEM_Helmholtz}. Given a function~$u \in H^1(\Omega)$, there exists~$\uIH \in \VhH$ such that
\[
\Vert u - \uIH \Vert_{1,\k,\taun} \le c_{BA}(\h\k)   \Vert u - \wp \Vert_{1,\k,\taun} \quad \quad \forall \wp \in \PWptaun,
\]
where
\[
c_{BA}(t):= 2 c_1 (1+c_2 t^2)  (2 + c_3 t^2),
\]
for three positive constants~$c_1$, $c_2$, and~$c_3$.
\end{prop}
\begin{proof}
The proof is rather
technical. Therefore, we refer to~\cite[Theorem~4.2]{TVEM_Helmholtz}
for details. There, the constants~$c_1$, $c_2$, and~$c_3$ are provided explicitly.
We can define the function~$\uIH$ as the interpolant of~$u$ through the degrees of freedom~\eqref{PW-edge-dof}.
\end{proof}

The target function~$u$ in the statement of Proposition~\ref{proposition:best-interpolation-Helmholtz} does not need to belong to the kernel of the Helmholtz operator. We only require that it belongs to~$H^1(\Omega)$.
Clearly, we need to require that~$u$ belongs to the kernel of the
Helmholtz operator if we want to combine
Proposition~\ref{proposition:best-PW} together with
Proposition~\ref{proposition:best-interpolation-Helmholtz} in order to
recover high-order approximation rates
in virtual element spaces.

\paragraph*{Projections and stabilizations.}
Recall the splitting
\[
\a(u,v) = \sum_{\E \in \taun} \aE(u_{|\E}, v_{|\E}).
\]
We consider the following discretizations of~$\a$ and~$\aE$:
\begin{equation} \label{discrete-bf-Helmholtz}
\begin{split}
\anH 	& (\uhH, \vhH)  := \sum_{\E \in \taun} \anEH(\uhH{}_{|\E}, \vhH{}_{|\E}) \\
& := \sum_{\E \in \taun} \aE(\PiHp \uhH{}_{|\E}, \PiHp \vhH{}_{|\E}) \\
& \quad + \SEH ( (I- \PiHp) \uhH{}_{|\E}, (I-\PiHp) \vhH{}_{|\E}  ),
\end{split}
\end{equation}
where we still have to define the projector~$\PiHp$ and the sesquilinear form $\SEH(\cdot, \cdot)$.
The operator~$\PiHp: \VhHE\to \PWpE$ is the projection operator with respect to the local sesquilinear form~$\aE(\cdot, \cdot)$. More precisely, for all~$\E\in \taun$, we set
\[
\aE(\PiHp \vhH -  \vhH,   \wellE)  = 0 \quad \forall \vhH \in \VhHE,\, \forall \wellE \in \PWpE.
\]
The computability of such a projector follows from the definition of
the local spaces~$\VhHE$ and of the degrees of freedom in~\eqref{abstract-edge-dof}:
\[
  \begin{split}
&  \aE(\PiHp \vhH, \well)= \aE(\vhH, \well)\\ 
& = -( \vhH, \underbrace{(\Delta + \k^2)\well}_{=0})_{0,\E} + \sum_{\e \in \EE}  (\vhH, \!\!\!\underbrace{\gammatraceIe(\well)}_{\in \gammatraceIe (\PWpE)}\!\!\!)_{0,\e} - \i\k \sum_{\e \in \EE}  (\vhH, \!\!\! \!\!\!\underbrace{\well}_{\in \gammatraceIe   (\PWpE)}\!\!\! \!\!\!)_{0,\e}.
\end{split}
\]
In order to have the well-posedness of the projector~$\PiHp$, we do not need to impose any additional computable condition.
Rather, we need to require a threshold condition on the mesh size that, in addition to~\eqref{eq:additional_assumption}, also guarantees that $k^2$ is not a Neumann-Laplace eigenvalue.
In particular, the following result is valid; see~\cite[Proposition~3.1]{TVEM_Helmholtz} and~\cite[Propositions~2.1 and~2.3]{Helmholtz-VEM}.
\begin{prop} \label{proposition:well-posedness-projector-Helmholtz}
Let \text{the} assumptions~\eqref{eq:shape-regularity}, \eqref{eq:ass_directions}, and~\eqref{eq:additional_assumption} be valid.
Moreover, assume that~$\h$ is sufficiently small, so that~$\k^2$ is smaller that the first Neumann-Laplace eigenvalue on each~$\E \in \taun$. Then, the projector~$\PiHp$ is well-defined and continuous.
More precisely, there exists a positive constant~$\beta(\hE\k)$, uniformly bounded away from zero as~$\hE \k\to 0$, such that
\begin{equation} \label{continuity:projection:Helmholtz}
\Vert \PiHp \vhH \Vert_{1,\k,\E} \le  \frac{1}{\beta(\hE \k)} \Vert \vhH \Vert_{1,\k,\E} \quad \quad \forall \vhH \in \VhHE,\, \forall \E \in \taun.
\end{equation}
\end{prop}
\medskip

By defining
\begin{equation} \label{betamin}
\betamin := \min_{\E \in \taun} \beta(\hE \k),
\end{equation}
inequality~\eqref{continuity:projection:Helmholtz} implies
\begin{equation}\label{continuity:projection:HelmholtzNEW}
\Vert \PiHp \vhH \Vert_{1,\k,\E} \le  \frac{1}{\betamin} \Vert \vhH \Vert_{1,\k,\E} \quad \quad \forall \vhH \in \VhHE,\, \forall \E \in \taun.
\end{equation}

We observe that the choice of the degrees of
freedom~\eqref{PW-edge-dof} also allows us to compute the $L^2$-edge projector $\PieH: \VhHE{}_{|\e} \rightarrow \gammatraceIe (\PWpE){}_{|\e})$ into traces of plane waves, which is defined as follows:
for all~$\e \in \EE$,
\begin{equation} \label{projector:L2-edge-Helmholtz}
(\PieH \vhH{}_{|\e} -  \vhH{}_{|\e}, \gammatraceIe (\wellE))_{0,\e}=0   \quad  \forall \vhH \in\VhHE,\; \forall \wellE \in \PWp(\E).
\end{equation}
This projector is needed for the discretization of the boundary terms, i.e., the sesquilinear form~$\b(\cdot, \cdot)$ and the right-hand side~$(\g,\cdot)_{0,\partial \Omega}$ appearing in~\eqref{Helmholtz:weak}.
We introduce
\begin{equation}\label{eq:bH}
\begin{split}
& \bnH(\uhH, \vhH) := \i\k \sum_{\e \in \EhB} (\PieH \uhH{}_{|\e}, \PieH \vhH{}_{|\e})_{0,\e},\\
&  (\g,\vhH)_{0, \partial \Omega}   \approx \sum_{\e \in \EhB} (\g, \PieH \vhH{}_{|\e})_{0, \e}.
\end{split}
\end{equation}
With respect to the abstract setting in Section~\ref{section:general-framework}, the form~$G(\cdot)$ is here given by
\[
G(\vhH) := \int_{\partial \Omega} \g \overline{\vhH}.
\]

The last ingredient we need is a stabilization~$\SEH(\cdot, \cdot)$ for all~$\E \in \taun$, which is computable via the degrees of freedom~\eqref{PW-edge-dof} and satisfies certain properties.
So far, we have recalled the setting of~\cite{TVEM_Helmholtz}.
Here, we weaken the assumptions on the stabilization demanded there, and yet deduce the well-posedness and convergence for the method.

More precisely, for all~$\E \in \taun$, we require that
\begin{equation} \label{weak-stab:coercivity}
\SEH(\vhH, \vhH)\ge \vert \vhH \vert ^2_{1,\E} - (1+C_S)   \k^2 \Vert \vhH \Vert^2_{0,\E}  \quad \quad \forall \vhH \in \ker(\PiHp)
\end{equation}
and the continuity
\begin{equation} \label{continuity:stab:Helmholtz}
\vert \SEH(\uhH ,\vhH) \vert \le C_C \Vert \uhH \Vert_{1,\k,\E} \Vert \vhH \Vert_{1,\k,\E}     \qquad\forall \uhH ,\vhH\in \ker(\PiHp),
\end{equation}
where~$C_S$ and~$C_C$ are two positive constants independent of~$\k$, with~$C_S= C_S(\h) \to 0$ as~$\h\to 0$.

With these choices, we are in a position to prove the following ``weak'' version of the G\r arding inequality.
\begin{prop} \label{prop:weak-Garding}
For every~$\E \in \taun$, let the stabilization~$\SEH(\cdot , \cdot)$ satisfy~\eqref{weak-stab:coercivity}. Then, the following G\r arding-type inequality is valid:
\begin{equation} \label{weak:Garding}
\begin{split}
\RE 
&[\anH(\vhH ,\vhH)  + \bnH(\vhH , \vhH)]  + 2\k^2  \Vert \vhH \Vert^2_{0,\Omega} \\
& + C_S \k^2 \sum_{\E \in\taun} \Vert (I-\PiHp) \vhH \Vert^2_{0,\E} \ge \Vert \vhH \Vert^2_{1,\k,\taun}    \; \quad \forall \vhH \in \VhH,
\end{split}
\end{equation}
where~$C_S=C_S(\h) \to 0$ as~$\h\to 0$ is the constant in~\eqref{weak-stab:coercivity}.
\end{prop}
\begin{proof}
The proof follows along the same lines as the one of~\cite[Proposition~4.2]{Helmholtz-VEM}. For the sake of completeness, we carry out the details here.
From~\eqref{eq:bH}, \eqref{discrete-bf-Helmholtz}, and
simple algebra, we get
\[
\begin{split}
& \RE [\anH(\vhH ,\vhH)  + \bnH(\vhH , \vhH) ] + 2\k^2 \Vert  \vhH \Vert^2_{0,\Omega}  = \anH(\vhH, \vhH) + 2\k^2 \Vert \vhH \Vert ^2_{0,\Omega} \\
& = \sum_{\E \in \taun} \left\{  \aE (\PiHp \vhH, \PiHp \vhH) + 2\k^2 \Vert \PiHp \vhH \Vert^2_{0,\E}   \right\}\\
& \quad  + \sum_{\E \in \taun} \left\{ \SEH ( (I\!-\!\PiHp) \vhH, (I\!-\!\PiHp)\vhH)  + 2\k^2 \Vert (I\!-\!\PiHp) \vhH \Vert^2_{0,\E}  \right\} \\
& \quad + \sum_{\E \in \taun} 4\k^2 \RE\left[ \int_\E \PiHp \vhH \overline{ (I\!-\!\PiHp) \vhH}   \right].
\end{split}
\]
Then, using~\eqref{weak-stab:coercivity} and simple calculations, we deduce
\[
\begin{split}
& \RE [\anH	 (\vhH ,\vhH)  + \bnH(\vhH , \vhH)] + 2\k^2 \Vert \vhH \Vert^2_{0,\Omega}  \\
& \ge \sum_{\E \in \taun} \left\{  \vert \PiHp \vhH \vert _{1,\E}^2 + \k^2 \Vert \PiHp \vhH \Vert^2_{0,\E}   \right\}\\
& \quad + \sum_{\E \in \taun} \left\{ \vert (I-\PiHp)\vhH \vert^2_{1,\E} + (1-C_S) \k^2 \Vert (I-\PiHp) \vhH \Vert^2_{0,\E}    \right\} \\
& \quad + \sum_{\E \in \taun} 2\, \RE \left[ \int_\E \nabla \PiHp \vhH \overline{ \nabla (I-\PiHp) \vhH}   \right] \\
& \quad + \sum_{\E \in \taun} 2\k^2 \RE\left[ \int_\E \PiHp \vhH \overline{ (I-\PiHp) \vhH}   \right] .\\
\end{split}
\]
Thus, we have
\[
\begin{split}
\RE [\anH(\vhH ,\vhH) &+ \bnH(\vhH , \vhH)] + 2\k^2 \Vert \vhH \Vert^2_{0,\Omega} + C_S \k^2 \sum_{\E \in \taun}\Vert (I-\PiHp)\vhH \Vert^2_{0,\E} \\
& \ge \sum_{\E \in \taun} \bigg \{  \vert \PiHp \vhH  \vert^2_{1,\E}
  + \vert (I-\PiHp) \vhH \vert^2_{1,\E} \\
& \qquad\qquad + 2\, \RE
  \left[ \int_\E \nabla \PiHp \vhH \cdot \overline{ \nabla (I-\PiHp) \vhH }    \right]    \bigg\} \\
& \quad + \sum_{\E \in \taun} \bigg\{  \k^2 \Vert \PiHp \vhH \Vert ^2_{0,\E} + \k^2 \Vert (I-\PiHp) \vhH \Vert^2_{0,\E} \\
& \qquad\qquad+2\k^2 \mathbb {RE} \left[ \int_\E \PiHp \vhH \overline{(I-\PiHp)\vhH} \right]   \bigg\}\\
& = \sum_{\E \in \taun} \left(  \vert \vhH \vert^2_{1,\E} + \k^2 \Vert \vhH \Vert ^2_{0,\E}  \right) = \Vert \vhH \Vert_{1,\k,\taun}^2 ,
\end{split}
\]
whence the assertion follows.
\end{proof}

Assuming~\eqref{continuity:stab:Helmholtz}, we also get the continuity of the discrete sesquilinear form~$\anH(\cdot,\cdot)$ in~\eqref{discrete-bf-Helmholtz}.
\begin{prop} \label{prop:continuity:Helmholtz}
Under assumption~\eqref{continuity:stab:Helmholtz}, the discrete
sesquilinear form $\anH(\cdot , \cdot)$ in~\eqref{discrete-bf-Helmholtz} satisfies
\begin{equation} \label{continuity:bf:Helmholtz}
\anH(\uhH, \vhH) \le \frac{1+ C_C(1+\betamin) ^2}{\betamin^2} \Vert \uhH \Vert_{1,k,\taun} \Vert \vhH \Vert_{1,k,\taun},
\end{equation}
where~$C_C$ is the constant in~\eqref{continuity:stab:Helmholtz} and~$\betamin$ is defined in~\eqref{betamin}.
\end{prop}
\begin{proof}
We have
\[
      \begin{split}
        \anH(\uhH,  \vhH) &= \sum_{\E \in \taun} \left\{  \aE (\PiHp \uhH, \PiHp \vhH)   \right.\\
&\left. \qquad\qquad+   \SEH \left( (I-\PiHp)\uhH , (I-\PiHp) \vhH  \right)    \right\}   \\
& \overset{\eqref{continuity:stab:Helmholtz}}{\le} \sum_{\E \in \taun} \left\{  \Vert \PiHp \uhH \Vert_{1,\k,\E}  \Vert \PiHp \vhH \Vert_{1,\k,\E} \right.\\
&\left. \qquad\qquad+ C_C \Vert (I-\PiHp) \uhH \Vert_{1,\k,\E}  \Vert (I-\PiHp) \vhH \Vert_{1,\k,\E}\right\} \\
& \overset{\eqref{continuity:projection:HelmholtzNEW}}{\le}  \sum_{\E \in \taun} \frac{1+ C_C(1+\betamin)^2}{\betamin^2} \Vert \uhH \Vert_{1,\k,\E} \Vert \vhH \Vert_{1,\k,\E}  
					\\
&\;\le \frac{1+ C_C(1+\betamin)^2}{\betamin^2} \Vert \uhH \Vert_{1,\k,\taun} \Vert \vhH \Vert_{1,\k,\taun}.\\
\end{split}
\]
\end{proof}

\begin{remark} \label{remark:old-Helmholtz-setting-stab}
In~\cite{TVEM_Helmholtz}, the assumption on the stabilization~$\SEH(\cdot, \cdot)$ was slightly stronger than~\eqref{weak-stab:coercivity}, namely we required
\begin{equation}\label{eq:strongerGarding}
\SEH(\vhH, \vhH)\ge   \vert \vhH \vert ^2_{1,\E} - \k^2 \Vert \vhH \Vert^2_{0,\E} \quad \quad \forall \vhH \in \ker(\PiHp).
\end{equation}
This assumption results in the following stronger version of the G\r arding inequality:
\[
\RE[\anH(\vhH ,\vhH) + \bnH(\vhH , \vhH)]  + 2\k^2 \Vert \vhH \Vert^2_{0,\Omega}  \ge \Vert\vhH \Vert^2_{1,\k,\taun}    \; \quad \forall \vhH \in \VhH.
\]
As we will see in Theorem~\ref{theorem:abstract-Helmholtz} below, the present weaker setting still allows us to derive an abstract error analysis for the method~\eqref{ncTVEM:Helmholtz} below.
The advantage of the new setting is that the design of a computable stabilization~$\SEH(\cdot, \cdot)$ becomes easier.

For a stronger version of Theorem~\ref{theorem:abstract-Helmholtz} below, relying on \text{the} assumption~\eqref{eq:strongerGarding} instead of~\eqref{weak:Garding}, we refer to~\cite[Theorem~4.3]{TVEM_Helmholtz}.
\end{remark}

\paragraph*{The method.}
We have introduced all the ingredients needed for the design of the nonconforming Trefftz VEM for the Helmholtz problem:
\begin{equation} \label{ncTVEM:Helmholtz}
\begin{cases}
\text{find } \uhH \in \VhH \text{ such that}\\
\anH(\uhH, \vhH) + \bnH(\uhH, \vhH) = (g, \PieH \vhH) \quad \quad \forall \vhH \in \VhH.
\end{cases}
\end{equation}
The well-posedness of the method follows by using a Schatz-type argument, as detailed in Theorem~\ref{theorem:abstract-Helmholtz} below.

\paragraph*{Convergence analysis.}
In the following theorem, we prove well-posedness and abstract error estimates for method~\eqref{ncTVEM:Helmholtz}.
In particular, the error of the method is controlled by two terms: a best approximation estimate in discontinuous, piecewise plane wave spaces and an estimate of the approximation of the boundary condition~$\g$.
For simplicity, an additional term involving the nonconformity of the method, which is hidden in the proof, is not explicitly reported.
Proposition~\ref{proposition:best-interpolation-Helmholtz} is used in order to absorb an interpolation error term within the best approximation in discontinuous, piecewise plane wave spaces.

\begin{thm} \label{theorem:abstract-Helmholtz}
Let the solution~$u$ to~\eqref{Laplace:weak} be in~$H^2(\Omega)$. Let the number of local plane wave directions in~\eqref{plane-wave-and-directions} be~$2\p+1$, with~$\p\ge 2$. 
Let \text{the} assumptions~\eqref{eq:shape-regularity} and~\eqref{eq:additional_assumption} on the meshes, \text{the} assumption~\eqref{eq:ass_directions} on the local plane wave directions,
and \text{the} assumptions~\eqref{weak-stab:coercivity} and~\eqref{continuity:stab:Helmholtz} on the local stabilization forms be valid.
Additionally, we require that~$\h\k^2$ is sufficiently small; see~\cite[eqt. (4.65)]{TVEM_Helmholtz}.
Then, there exists a unique solution~$\uhH$ to the method~\eqref{ncTVEM:Helmholtz}, and the following a priori estimate is valid:
\[
\begin{split}
\Vert u - \uhH \Vert_{1,\k, \taun} 	& \lesssim c_1(\h,\k) \Vert u - \wellE \Vert _{1,\k,\taun} + \h\, c_2(\h,\k) \vert u - \wellE \vert_{2,\taun}  \\
							& \quad + \h^{\frac{1}{2}} c_2(\h,\k) \Vert \g - \PieH \g \Vert_{0,\partial \Omega} \quad \forall \wellE \in \PWp(\Omega, \taun),
\end{split}
\]
where~$c_1$ and~$c_2$ are two constants, which remain bounded as~$\h\k^2 \to 0$.
Indeed, we can express such constants explicitly; see the statement of~\cite[Theorem~4.3]{TVEM_Helmholtz}.
\end{thm}
\begin{proof}
The proof follows along the same lines as that of~\cite[Theorem~4.3]{TVEM_Helmholtz}.
For this reason, we only present the modifications that are due to the validity of the weaker
G\r arding-type inequality~\eqref{weak:Garding}; see Remark~\ref{remark:old-Helmholtz-setting-stab}.

We first observe that, for all~$\uIH \in \VhH$,
\[
\Vert u - \uhH \Vert_{1,\k,\taun} \le \Vert u - \uIH \Vert_{1,\k,\taun}  +  \Vert \uIH - \uhH \Vert_{1,\k,\taun}.
\]
We focus on the second term on the right-hand side. For the sake of simplicity, write~$\deltah:= \uIH - \uhH$.
By applying the G\r arding-type inequality~\eqref{weak:Garding}, we deduce
\begin{equation} \label{an:estimate:abstract:Helmholtz}
  \begin{split}
& \Vert \deltah \Vert_{1,\k,\taun}^2 \le \RE (\anH(\deltah,\deltah) + \bnH(\deltah, \deltah)) + 2\k^2 \Vert \deltah \Vert^2_{0,\Omega} \\
&\quad \quad \quad \quad  + C_S \k^2 \sum_{\E\in\taun} \Vert (I-\PiHp) \deltah \Vert^2_{0,\E}  =: I + II + III.
\end{split}
\end{equation}
The terms~$I$ and~$II$ are dealt with exactly as
in~\cite{TVEM_Helmholtz}. As for the term~$III$, we proceed as
follows:
\[
\begin{split}
III 
& =  C_S \k^2 \sum_{\E\in\taun}\Vert (I-\PiHp) \deltah \Vert^2_{0,\E}  \le C_S  \Vert (I-\PiHp) \deltah \Vert^2_{1,\k,\taun} \\
& \le  2 C_S (\Vert \deltah \Vert^2_{1,\k,\taun}  +   \Vert \PiHp \deltah \Vert^2_{1,\k,\taun})  \overset{\eqref{continuity:projection:Helmholtz}}{\le} 2 C_S \left( \frac{1}{\beta ^2} +1 \right)  \Vert \deltah \Vert^2_{1,\k,\taun}.
\end{split}
\]
Recall that~$C_S = C_S(\h) \to 0$ as~$\h \to 0$. Then, for a sufficiently small~$\h$, we get
\[
III \le \frac{1}{4} \Vert \deltah \Vert^2_{1,\k,\taun}.
\]
Thus, we absorb the term~$III$ within the left-hand side of~\eqref{an:estimate:abstract:Helmholtz} yielding
\[
\frac{3}{4} \Vert \deltah \Vert_{1,\k,\taun}^2 \le I + II.
\]
The proof follows then along the same lines as that of~\cite[Theorem~4.3]{TVEM_Helmholtz} with~$\alpha_h = 3/4$.
Therefore, we omit further details.
\end{proof}

\begin{remark}
In view of Theorem~\ref{theorem:abstract-Helmholtz}, the approximation properties in Proposition~\ref{proposition:best-PW} and some algebra,
we deduce that the optimal $\h$-convergence is valid under suitable
regularity assumptions on the solution~$u$ to problem~\eqref{Helmholtz:weak} and on the boundary datum~$\g$.
We refer to~\cite[Theorem~4.4]{TVEM_Helmholtz} for a precise statement.
\end{remark}

\emph{Overall, the nonconforming Trefftz VEM for the Helmholtz problem is a modification of the standard nonconforming VEM, in the sense that it encodes certain properties of the solution to the problem within the defintion of the VE spaces.
 The resulting method has significantly fewer degrees of freedom than a standard VEM based on complete polynomial spaces, yet keeping the same convergence properties.
 Differently from the case of the Laplace problem, we need to resort to nonpolynomial underlying spaces (plane wave spaces, in our presentation).}

\section{Stability and dispersion analysis for the nonconforming Trefftz VEM for the Helmholtz equation} \label{section:NR-Helmholtz}
Here, we address the issue of the dispersion analysis for the nonconforming Trefftz VEM for the Helmholtz equation detailed in Section~\ref{section:Helmholtz}.

Amongst the difficulties in approximating time-harmonic wave propagation problems, we highlight the so-called \textit{pollution effect}~\cite{sauterpollution},
which describes the widening discrepancy between the best approximation error and the discretization error for large values of the wave number~$\k$.

This effect is directly linked to \textit{numerical dispersion}, representing the failure of the numerical method to reproduce the correct oscillating behaviour of the analytical solution.
More precisely, for a given wave number~$\k$, a continuous problem with plane wave solution is considered. Its numerical approximation delivers an approximate solution, which can be interpreted as a wave with a deviated wave number~$\kn$.
We can measure this mismatch of the continuous and discrete wave numbers $\k$ and~$\kn$ separately in terms of the real part and the imaginary part with the following interpretation.
The term~$|\Real{(k-\kn)}|$ represents the deviation (shift) of the phase (\textit{dispersion}), and the term $|\Imag{(k-\kn)}|=|\Imag{(\kn)}|$ refers to the damping of the amplitude (\textit{dissipation}) of the computed discrete solution.
Moreover, the difference $|k-\kn|$ measures the total amount of dispersion and dissipation and is sometimes referred to as \textit{total dispersion} or \textit{total error}.

We summarize the general strategy for a dispersion analysis in the following two steps:
\begin{enumerate}
\item Consider the discretization scheme of the numerical method applied to $-\Delta u -\k^2 u=0$ in~$\mathbb R^2$ using infinite meshes which are invariant under a discrete group of translations. Due to translation invariance, it is then possible to reduce the infinite mesh to a finite one.
\item Given a plane wave with wave number~$k$ travelling in a fixed direction, seek a so-called discrete \textit{Bloch wave solution},
which can be regarded as a generalization of the given continuous plane wave based on the underlying approximating spaces, and determine for which (discrete) wave number $\kn$ this Bloch wave solution actually solves the discrete variational formulation.
This procedure leads to small nonlinear eigenvalue problems, which need to be solved. 
\end{enumerate}

In the framework of standard conforming finite element methods (FEM) for the Helmholtz problem, a full dispersion analysis was done in~\cite{deraemaeker} for dimensions one to three.
Furthermore, in~\cite{sauterpollution} it was shown that the pollution effect can be avoided in 1D, but not in higher dimensions, and a generalized pollution-free FEM in 1D was constructed. Moreover, we highlight the work in~\cite{ihlenburg1995dispersion},
where a link between the results of the dispersion analysis and the numerical analysis was established for finite elements,
and the work in~\cite{ainsworth}, where quantitative, fully explicit estimates for the behaviour and decay rates of the dispersion error were derived in dependence on the order of the method relative to the mesh size and the wave number.
Also in the context of nonconforming methods, dispersion analyses have been performed for the discontinuous Galerkin (DG)-FEM~\cite{ainsworth2004dispersive,ainsworth2006dispersive},
the discontinuous Petrov-Galerkin (DPG) method~\cite{gopalakrishnan2014dispersive}, and the plane wave discontinuous Galerkin method (PWDG)~\cite{gittelson}.
Recently, a dispersion analysis for hybridized DG (HDG)-methods has been carried out in~\cite{gopalakrishnan2018dispersion}, including an explicit derivation of the wave number error for lowest order single face HDG methods.

Here, we numerically investigate the dispersion and dissipation
properties of the nonconforming Trefftz VEM (ncTVEM), and compare the results to those obtained in~\cite{gittelson} for PWDG, and to those for standard polynomial based FEM.

The remainder of the section is organized as follows.
In Section~\ref{subsection:abstract-dispersion}, we describe the abstract setting for the dispersion analysis.
Then, in Section~\ref{subsection:minimal-subspaces}, we specify the set of basis functions and the sesquilinear forms defining the numerical discretization schemes for the ncTVEM.
Finally, in Section~\ref{subsection:numerical-results}, we numerically study the dispersion and dissipation and we compare the results with those obtained with other methods.

\subsection{Abstract dispersion analysis} \label{subsection:abstract-dispersion}
In this section, we fix the abstract setting for the dispersion analysis employing the notation of~\cite{gittelson}.

To this purpose, in order to remove possible dependencies of the dispersion on the boundary conditions of the problem, we consider the Helmholtz problem~\eqref{Helmholtz:weak} on the unbounded domain $\Omega=\mathbb R^2$.
Let $\taun:=\{\E\}$ be a translation-invariant partition of $\Omega$ into polygons with mesh size $h:=\max_{\E \in \taun} \hE$, where $\hE:=\diam(\E)$,
i.e., there exists a set of elements $\widehat{\E}_1,\dots,\widehat{\E}_r$, $r \in \mathbb N$, such that the whole infinite mesh can be covered in a nonoverlapping way by shifts of the ``reference'' patch $\widehat{\E}:=\bigcup_{j=1}^r \widehat{\E}_j$.
In other words, this assumption implies the existence of translation vectors $\xibold_1, \xibold_2 \in \R^2$, such that every element $\E \in \taun$ can be written as a linear combination with coefficients in $\N_0$ of one of the reference polygons $\widehat{\E}_\ell$, $\ell=1,\dots,r$.
Some examples for translation-invariant meshes are shown in Figure~\ref{fig:meshes}. Moreover, we denote by $\EE$ the set of edges belonging to $\E$.

\begin{figure}[h]
\begin{center}
\begin{minipage}{0.28\textwidth} 
\centering
\includegraphics[width=\textwidth]{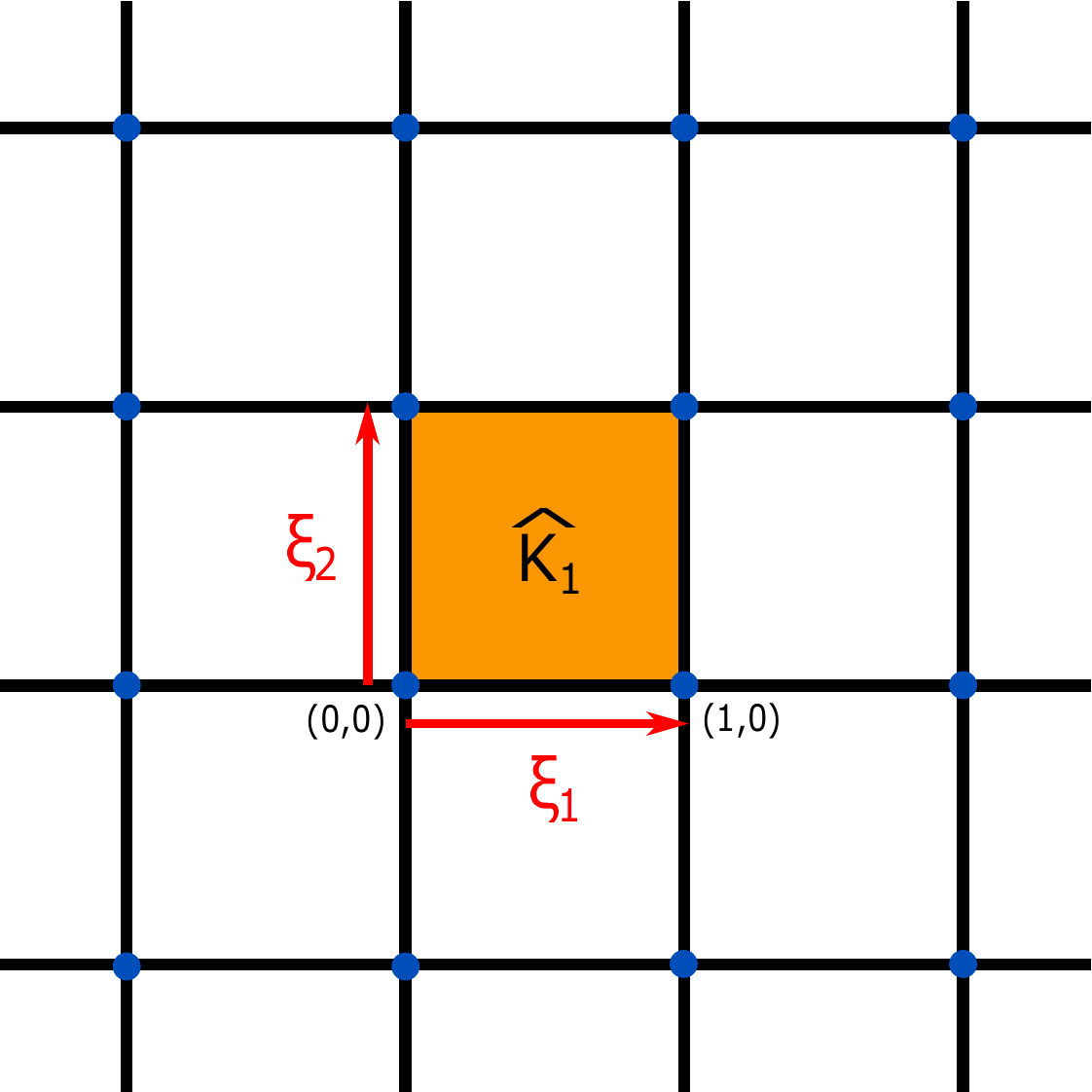}
\end{minipage}
\hfill
\begin{minipage}{0.35\textwidth}
\centering
\includegraphics[width=\textwidth]{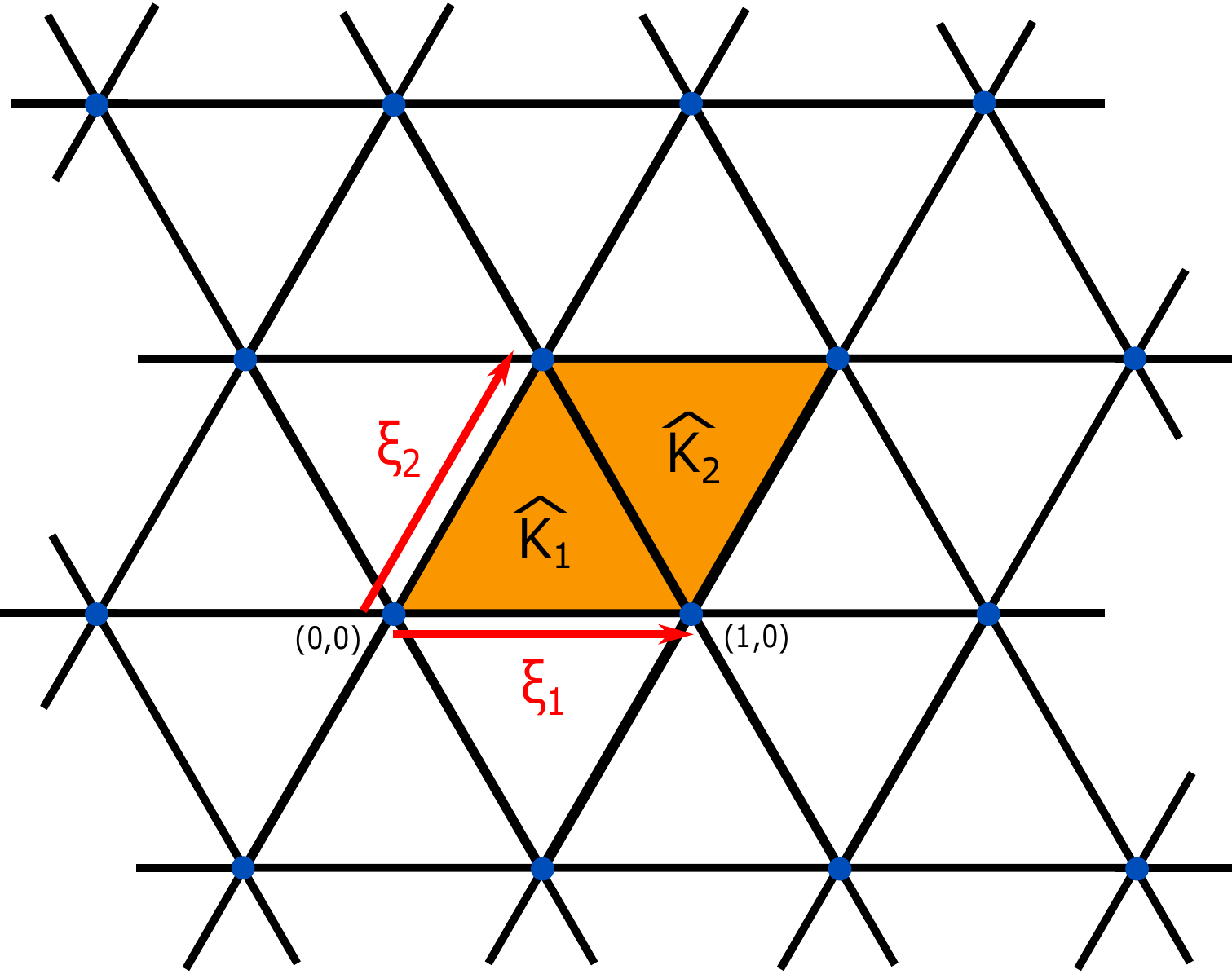}
\end{minipage}
\hfill
\begin{minipage}{0.28\textwidth}
\centering
\includegraphics[width=\textwidth]{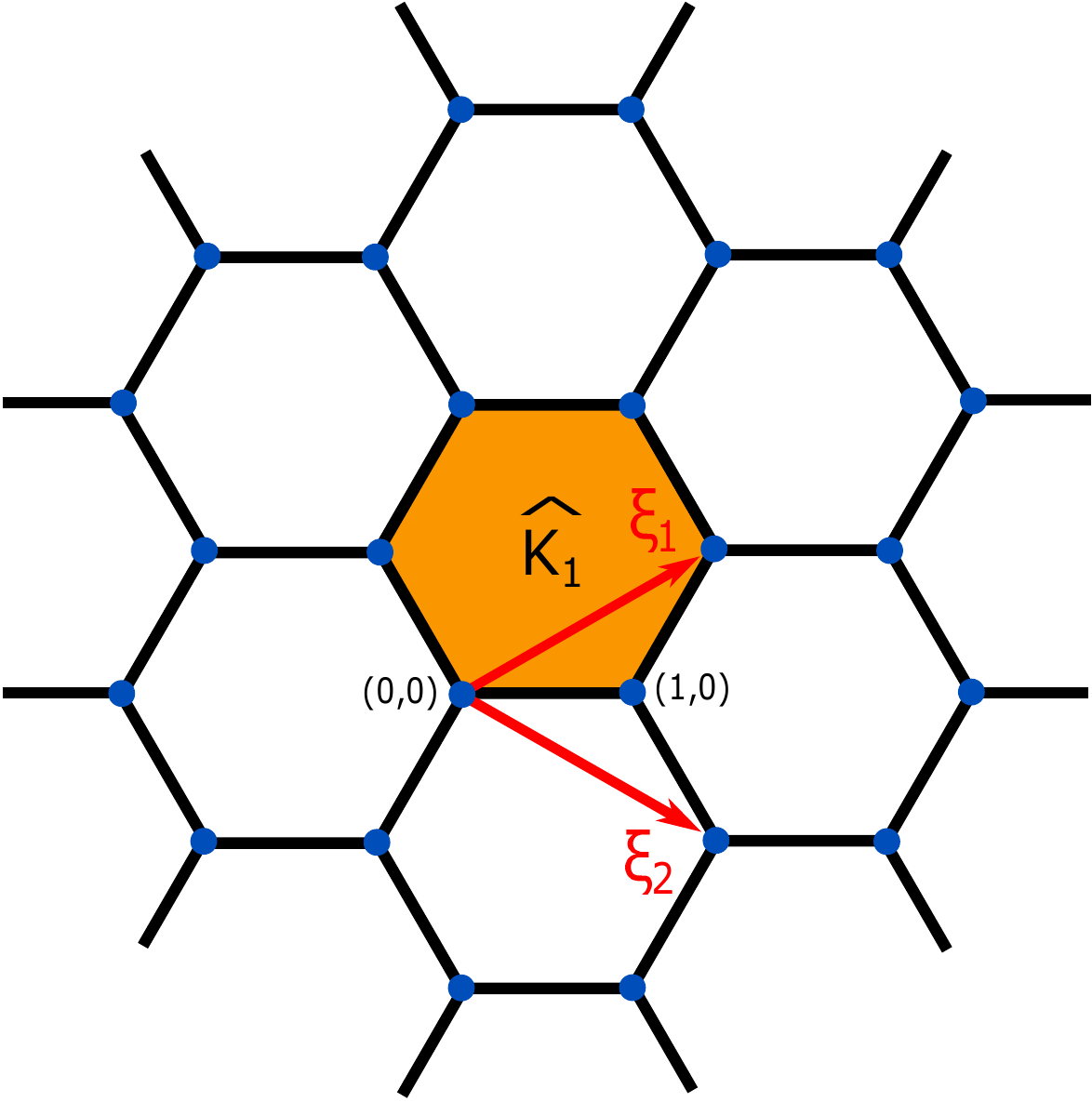}
\end{minipage}
\end{center}
\caption{Examples of translation-invariant meshes with the corresponding translation vectors $\xibold_1$ and $\xibold_2$: regular Cartesian mesh, triangular mesh, and hexagonal mesh, from \textit{left} to \textit{right}.}
\label{fig:meshes} 
\end{figure}

Let $u(\textbf{x})=e^{\im k \textbf{d} \cdot \textbf{x}}$, $\textbf{d} \in \R^2$ with $|\dbold|=1$ be a plane wave with wave number $k$ and traveling in direction~$\dbold$.
We denote by~$\Vn$ the global approximation space resulting from the discretization of~\eqref{Helmholtz:weak} using a Galerkin based numerical method, and by~$\Vnh \subset \Vn$ a minimal subspace generating~$\Vn$ by translations with
\begin{equation} \label{translation-xi_n}
\xibold_{\n}:=n_1 \xibold_1+n_2 \xibold_2, \quad \n=(n_1,n_2) \in \Z^2. 
\end{equation}
More precisely, depending on the structure of the method, $\Vnh$ is determined as follows. 
\begin{enumerate}
\item \textit{Vertex-related} basis functions: In this case, $\Vnh$ is the span of all basis functions related to a minimal set of vertices $\{\nu_i\}_{i=1}^{\lambdazero}$, $\lambdazero \in \N$,
such that all the other mesh vertices are obtained by translations with $\xibold_n$ of the form~\eqref{translation-xi_n}. An example is the FEM.
\item \textit{Edge-related} basis functions: Similarly as above, the space $\Vnh$ is in this case the span of all basis functions related to a minimal set of edges $\{\eta_i\}_{i=1}^{\lambdaone}$, $\lambdaone \in \N$, such that all the other edges of the mesh are obtained by translations with $\xibold_n$ of the form~\eqref{translation-xi_n}. This is, for instance, the case of the ncTVEM~\cite{TVEM_Helmholtz, TVEM_Helmholtz_num}.
\item \textit{Element-related} basis functions: Here, the space $\Vnh$ is simply given as the span of all basis functions related to a minimal set of elements $\{\sigma_i\}_{i=1}^{\lambdatwo}$, $\lambdatwo \in \N$, such that all other elements of the mesh are obtained by a translation with a vector $\xibold_n$ of the form~\eqref{translation-xi_n}. One representative of this category is the PWDG~\cite{GHP_PWDGFEM_hversion,TDGPW_pversion}.
\end{enumerate}
In the following, we will refer to these minimal sets of vertices $\{\nu_i\}_{i=1}^{\lambdazero}$, edges $\{\eta_i\}_{i=1}^{\lambdaone}$, and elements $\{\sigma_i\}_{i=1}^{\lambdatwo}$ as \textit{fundamental sets} of vertices, edges, and elements, respectively.

As a direct consequence, every~$\vn \in \Vn$ can be written as
\begin{equation*}
\vn(\x)=\sum_{\n \in \Z^2} \vnh(\x-\xibold_{\n}), \quad \vnh \in \Vnh.
\end{equation*}
Next, we define the discrete \textit{Bloch wave} with wave number~$\kn$ and traveling in direction~$\dbold$ by
\begin{equation} \label{Bloch_wave}
\un(\x)= \sum_{\n \in \Z^2} e^{\im \kn \dbold \cdot \xibold_{\n}} \unh(\x-\xibold_{\n}),
\end{equation} 
where $\unh \in \Vnh$, and $\kn \in \C$ with $\Real(\kn)>0$. Note that, since $\unh \in \Vnh$, the infinite sum in~\eqref{Bloch_wave} is in fact finite. Furthermore, given $\textup{\textbf{d}} \in \R^2$ with $|\textup{\textbf{d}}|=1$, the Bloch wave $\un$ in~\eqref{Bloch_wave} satisfies
\begin{equation*}
\un(\x+\xibold_{\lbold})=e^{\im \kn \dbold \cdot \xibold_{\lbold}} \un(\x),
\end{equation*}
for all $\lbold \in \Z^2$. This property follows directly by using the definition of the Bloch wave:
\begin{equation*}
\begin{split}
\un(\x+\xibold_{\lbold})&=
\sum_{\n \in \Z^2} e^{\im \kn \dbold \cdot \xibold_{\n}} \unh(\x+\xibold_{\lbold}-\xibold_{\n})=
\sum_{n \in \Z^2} e^{\im \kn \dbold \cdot \xibold_{\n}} \unh(\x-\xibold_{\n-\lbold}) \\
&=e^{\im \kn \dbold \cdot \xibold_{\lbold}} \sum_{\mbold \in \Z^2} e^{\im \kn \dbold \cdot \xibold_{\mbold}} \unh(\x-\xibold_{\mbold})
=e^{\im \kn \dbold \cdot \xibold_{\lbold}} \un(\x).
\end{split}  
\end{equation*}
Therefore, Bloch waves can be regarded as discrete counterparts, based on the approximation spaces, of continuous plane waves. 

We introduce the global (continuous) sesquilinear form
{\begin{equation} \label{definition aK}
a(u,v)\!:=\!\!\sum_{\E \in \taun}\!\! \aE(u,v) \!:=\!\!\sum_{\E \in \taun} \!\!\bigg[\int_{\E} \!\!\nabla u \cdot \overline{\nabla v} - \k^2 \!\!\int_{\E}\!\! u \overline{v} \, \bigg]  \ \, \forall u,v \in H^1(\R^2),
\end{equation}
and we denote by $\an(\cdot,\cdot)$ the global discrete sesquilinear form defining the numerical method under consideration.
In Section~\ref{subsection:minimal-subspaces} below, we will specify~$\Vnh$ and~$\an(\cdot,\cdot)$ for the ncTVEM  and the PWDG. 

Next, we define the discrete wave number~$\kn \in \C$ as follows.
\begin{defn} \label{def discrete wave number}
Given~$\k>0$ and $\textup{\textbf{d}} \in \R^2$ with $|\textup{\textbf{d}}|=1$, the \textit{discrete wave number} $\kn \in \C$ is the number with minimal~$|k-\kn|$, for which a discrete Bloch wave $\un$ of the form~\eqref{Bloch_wave} is a solution to the discrete problem
\begin{equation} \label{var_form}
\an(\un,\vnh)=0 \quad \forall \vnh \in \Vnh.
\end{equation}
\end{defn}
\noindent
Due to the scaling invariance of the mesh, we can assume that $\h=1$.
The wave number~$k$ on a mesh with $h=1$ corresponds to the wave number $k_0=\frac{k}{h_0}$ on a mesh with mesh size $h_0$.

Having this, the general procedure in the dispersion analysis now consists in finding those discrete wave numbers~$\kn \in \C$ and coefficients~$\unh \in \Vnh$, for which a Bloch wave solution of the form~\eqref{Bloch_wave} satisfies~\eqref{var_form},
and to measure the deviation of~$\kn$ from~$\k$ afterwards. This strategy results in solving small nonlinear eigenvalue problems. In fact, by plugging the Bloch wave ansatz~\eqref{Bloch_wave} into~\eqref{var_form} and using the sesquilinearity of~$\an(\cdot,\cdot)$, we obtain
\begin{equation} \label{relation 1} 
\sum_{\n \in \Z^2} e^{\im \kn \dbold \cdot \xibold_{\n}} \an(\unh(\cdot-\xibold_{\n}),\vnh)=0 \quad \forall \vnh \in \Vnh.
\end{equation}
Let $\{\chih_s \}_{s=1}^{\Xi} \subset \Vnh$ be a set of basis functions for the space $\Vnh$ that are related to fundamental elements, vertices, or edges, depending on the method. Then, we can expand $\unh$ in terms of this basis as
\begin{equation*}
\unh = \sum_{t=1}^{\Xi} u_t \chih_t.
\end{equation*} 
Plugging this ansatz into~\eqref{relation 1}, testing with $\chih_s$, $s=1,\dots,\Xi$, and interchanging the sums (this can be done since the infinite sum over $\n$ is in fact finite) yields
\begin{equation} \label{NEP:basis:fcts}
\sum_{t=1}^{\Xi} u_t \left(\sum_{\n \in \Z^2} e^{\im \kn \dbold \cdot \xibold_{\n}} \an(\chih_t (\cdot-\xibold_{\n}),\chih_s) \right)=0 \quad \forall s=1,\dots,\Xi,
\end{equation}
which can be represented as
\begin{equation} \label{relation 2}
\sum_{t=1}^{\Xi} \boldsymbol{T}_{s,t}(\kn) u_t = 0 \quad \forall s=1,\dots,\Xi,
\end{equation}
with
\begin{equation} \label{matrix T}
\boldsymbol{T}_{s,t}(\kn):=\sum_{\n \in \Z^2} e^{\im \kn \dbold \cdot \xibold_{\n}} \an(\chih_t (\cdot-\xibold_{\n}),\chih_s).
\end{equation}
The matrix problem corresponding to~\eqref{relation 2} has the form 
\begin{equation} \label{NEP}
\boldsymbol{T}(\kn) \textbf{u} = \boldsymbol{0},
\end{equation}
where $\boldsymbol{T}: \, \C \to \C^{\Xi \times \Xi}$ is defined via~\eqref{matrix T}, and $\boldsymbol{u}=(u_1,\dots,u_{\Xi})^T \in \C^{\Xi}$.
We highlight that $\boldsymbol{T}$ is a holomorphic map and~\eqref{NEP} is a small nonlinear eigenvalue problem, which can be solved using, e.g., an iterative method as done in~\cite{gittelson},
or a direct method based on a rational interpolation procedure~\cite{xiao2017solving} or on a contour integral approach~\cite{beyn2012integral, asakura2009numerical}.
For the numerical experiments presented in Section~\ref{subsection:numerical-results}, we will make use of the latter, which we will denote by \textit{contour integral method} (CIM) in the sequel.
Due to the use of plane wave related basis functions, deriving an exact analytical solution to~\eqref{NEP} is not possible even for the lowest order case.  

\subsection{Minimal generating subspaces} \label{subsection:minimal-subspaces}
In this section, we specify the minimal generating subspaces~$\Vnh$, the corresponding sets of basis functions $\{\chih_s \}_{s=1}^{\Xi}$, and the sequilinear forms $\an(\cdot,\cdot)$ for the ncTVEM and the PWDG~\cite{GHP_PWDGFEM_hversion,TDGPW_pversion}.
The basis functions for these two methods are edge-related and element-related, respectively.
In Figures~\ref{fig:dispersion-squares}-\ref{fig:dispersion-hexagons}, the stencils related to the fundamental sets of vertices, edges, and elements are depicted for these three methods and the meshes in Figure~\ref{fig:meshes}.

Before doing that, we recall some notation from Section~\ref{section:Helmholtz}. Let $\{\djj\}_{j=1}^p$, $p=2q+1$, $q \in \N$, be a set of equidistributed plane wave directions. We denote by
\begin{equation} \label{pw}
w_j(\x):=e^{\im k \djj \cdot \x}, \quad j=1,\dots,p,
\end{equation}
the plane wave with wave number~$\k$ and traveling along the direction~$\djj$. Furthermore, for every~$\E \in \taun$, we set $w_j^{\E}:={w_j}_{|_{\E}}$, and we introduce the bulk place waves space
\[
\PWE:=\lin\{ \wjE, \, j=1,\dots,p \}.
\]

\paragraph*{The case of ncTVEM.}
Let now $\{ \eta_i \}_{i=1}^{\lambdaone}$ be a fundamental set of edges.
Then, the set of basis functions $\{\chih_s^{(1)} \}_{s=1}^{\Xi}$ spanning the minimal generating subspace $\Vnh^{(1)}$ is given by the union of the canonical basis functions related to $\{ \eta_i \}_{i=1}^{\lambdaone}$.
More precisely, for $s \leftrightarrow (i,j)$, $i \in \{1,\dots,\lambdaone\}$ and $j \in \Jcal_{\eta_i}$, i.e. we identify $s$ with the edge index~$i$ and the index~$j$ associated with the $j$-th plane wave basis function on this edge as above,
\begin{equation*}
\chih_{s}^{(1)}=\chih_{(i,j)}^{(1)}:=\Psi_{(\eta_i,j)},
\end{equation*}
where $\Psi_{(\eta_i,j)}$ is defined elementwise as follows.
If~$\E \in \taun$ is an element abutting the edge~$\eta_i$, then~${\Psi_{(\eta_i,j)}}_{|_{\eta_i}}$ coincides with the local canonical basis function associated with the (global) edge~$\eta_i$ and the $j$-th orthogonalized edge plane wave basis function;
otherwise $\Psi_{(\eta_i,j)}$ is zero. Clearly, $\Xi=\sum_{i=1}^{\lambdaone} \widetilde{p}_{\eta_i}$.

As for the sesquilinear form~$\an^{(1)}(\cdot,\cdot)$, it coincides with that in~\eqref{discrete-bf-Helmholtz}, where the projector $\Pip$ is defined in~\eqref{projector:L2-edge-Helmholtz}.

\paragraph*{The case of PWDG.}
For PWDG, we refer to~\cite{gittelson} for a full dispersion analysis. For the sake of completeness, we shortly recall here the definitions of the minimal generating subspace and the sesquilinear form adapted to our setting.

The global approximation space $\Vn^{(2)}$ is given by
\begin{equation*}
\Vn^{(2)}:=\{\vn \in L^2(\R^2):\, {\vn}_{|_K} \in \PWE \quad \forall \E \in \taun\}.
\end{equation*}
Moreover, the global sesquilinear form $\an^{(2)}(\cdot,\cdot)$ is defined by
\begin{equation} \label{an-PWDG}
\begin{split}
\an^{(2)}(\un,\vn):=&\sum_{\E \in \taun} \aE(\un,\vn) - \int_{\En} \llbracket \un \rrbracket \cdot \{\!\!\{\overline{\nabla_n \vn}\}\!\!\} \\
&- \frac{\beta}{\im k} \int_{\En} \llbracket \nabla_n \un \rrbracket \, \llbracket \overline{\nabla_n \vn} \rrbracket  
- \int_{\En} \{\!\!\{\nabla_n \un\}\!\!\} \cdot \llbracket \overline{\vn} \rrbracket \\
&+ \im k \alpha \int_{\En} \llbracket \un \rrbracket \cdot \llbracket \overline{\vn} \rrbracket \qquad\qquad \forall \un, \vn \in \Vn^{(2)}.
\end{split}
\end{equation}
where $\aE(\cdot,\cdot)$ is given in~\eqref{definition aK}, $\En$ is the mesh skeleton, $\alpha, \beta>0$ are the flux parameters, and $\nabla_n$ is the broken gradient.
For~$v=\un$ or $\vn$, $\llbracket v \rrbracket$ is the standard trace jump as defined in~\eqref{jump}, and, on a given edge~$e$, denoting by~$\E^-$ and~$\E^+$ its adjacent elements, 
\begin{equation*} 
\{\!\!\{\nabla_n v \}\!\!\}:=\frac{1}{2} \left(\nabla  v_{|_{K^+}}+\nabla v_{|_{K^-}}\right) ,\quad \llbracket \nabla_n v \rrbracket:=\nabla   {v}_{|_{K^+}}\cdot\n_{{K^+}}+\nabla{v}_{|_{K^-}}\cdot\n_{{K^-}}
\end{equation*}
are the trace average and normal jump, respectively, of~$\nabla_n v$.
Recall that~$\taun$ is a partition of $\Omega=\R^2$, thus all edges in $\En$ are shared by two elements.

Let now $\{\sigma_i\}_{i=1}^{\lambdatwo}$ be a fundamental set of elements.
Then, the basis functions $\{\chih_s^{(2)} \}_{s=1}^{\Xi}$ are given by $\{w_j^{\sigma_i}\}_{i=1,\dots,\lambdatwo, j=1\dots,p}$, where $s \leftrightarrow (i,j)$, i.e. $s$ is identified with the element index $i$ and the plane wave direction index $j$, and $\Xi=\lambdatwo p$.
As mentioned above, the minimal generating subspace $\Vnh^{(2)} \subset \Vn^{(2)}$ is simply the span of the basis functions $\{\chih_s^{(2)} \}_{s=1}^{\Xi}$, and the sesquilinear form $\an^{(2)}(\cdot,\cdot)$ is given in~\eqref{an-PWDG}.

\paragraph*{Overview of the stencils generating the minimal subspaces.}
In Figures~\ref{fig:dispersion-squares}-\ref{fig:dispersion-hexagons}, we illustrate the stencils of the basis functions for the ncTVEM and the PWDG, employing the meshes made of squares, triangles, and hexagons, respectively, depicted in Figure~\ref{fig:meshes}.
The fundamental sets of vertices, edges, and elements are displayed in dark-blue, and the translation vectors~$\xibold_1$ and~$\xibold_2$ in red.
Furthermore, the supports of the basis functions spanning the minimal generating subspaces are coloured in light-blue for the ncTVEM.
Due to the locality of the basis functions, only those associated with the vertices, edges, and elements displayed in dark-blue and dark-yellow contribute to the sum~\eqref{NEP:basis:fcts}.
Integration only has to be performed over the elements $\E^\zeta$ and the adjacent edges.

\begin{figure}[h]
\begin{center}
\begin{minipage}{0.3\textwidth}
\centering
\includegraphics[width=\textwidth]{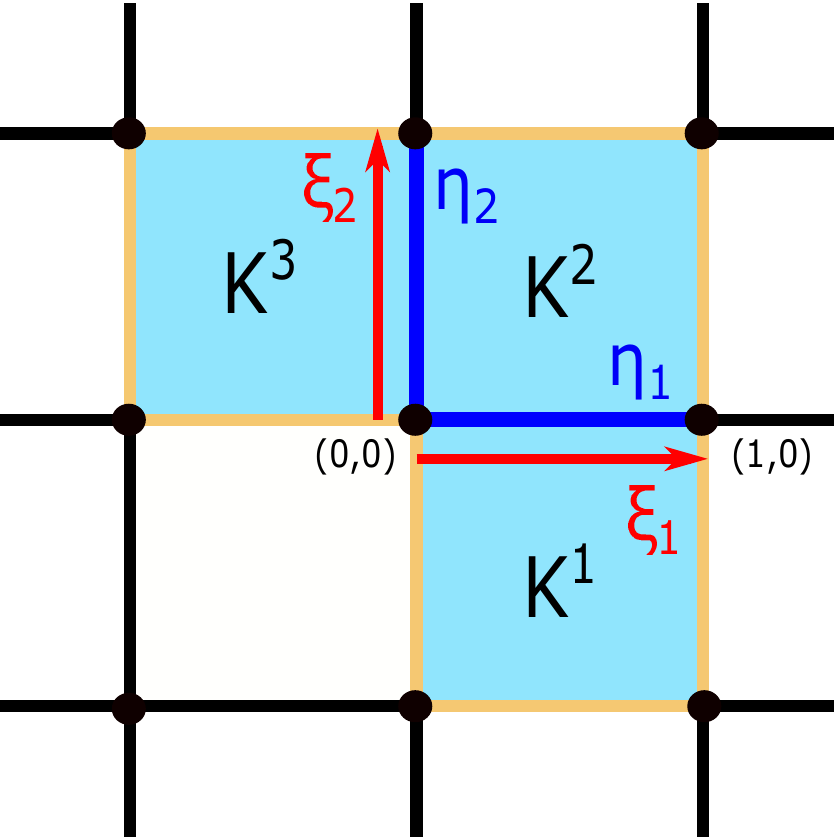}
\end{minipage}
\quad\quad\quad
\begin{minipage}{0.3\textwidth} 
\centering
\includegraphics[width=\textwidth]{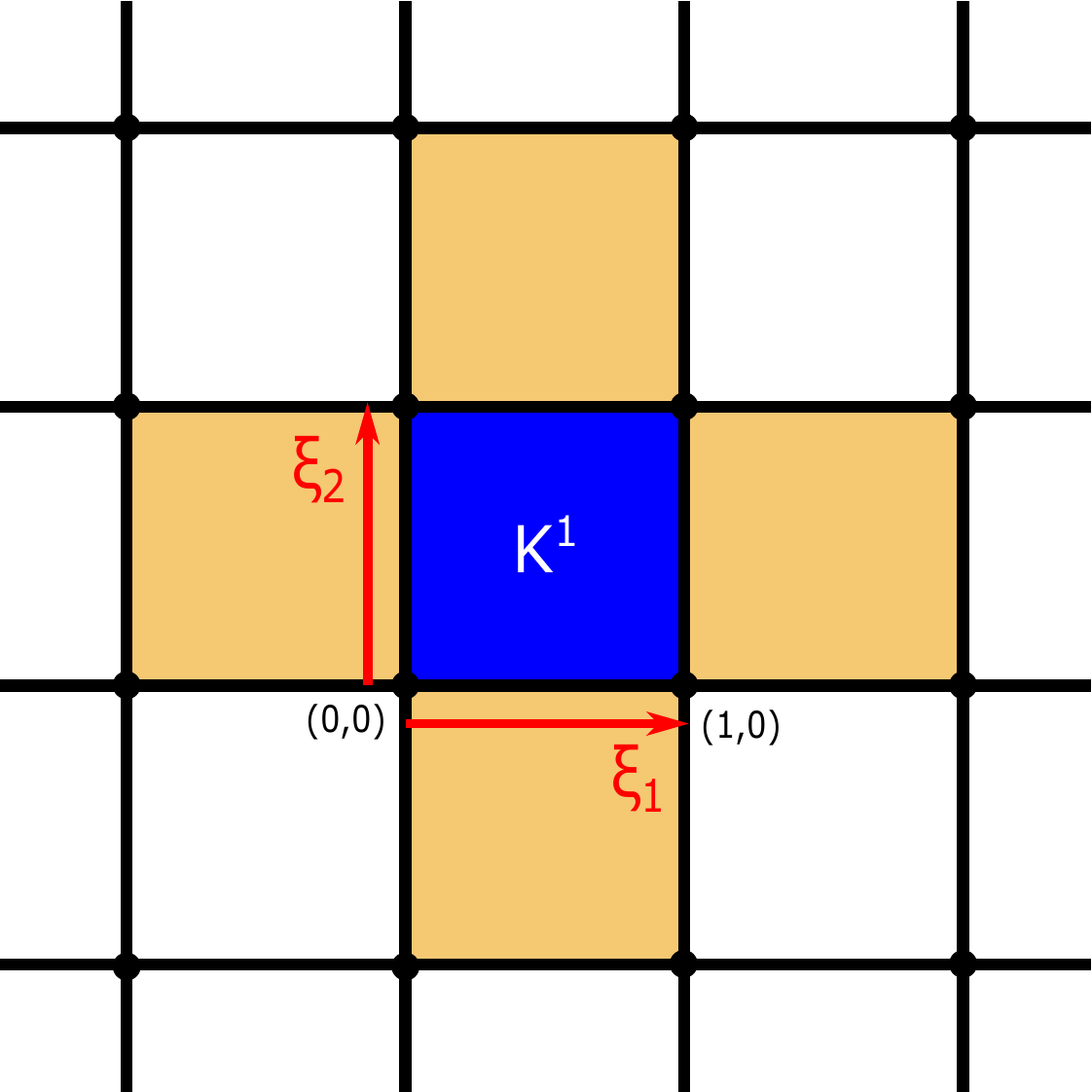}
\end{minipage}
\end{center}
\caption{Stencils of the basis functions related to the fundamental sets of edges (ncTVEM) and elements (PWDG), respectively, from \textit{left} to \textit{right}, when employing the meshes made of squares in Figure~\ref{fig:meshes}.}
\label{fig:dispersion-squares} 
\end{figure}

\begin{figure}[h]
\begin{center}
\begin{minipage}{0.3\textwidth}
\centering
\includegraphics[width=\textwidth]{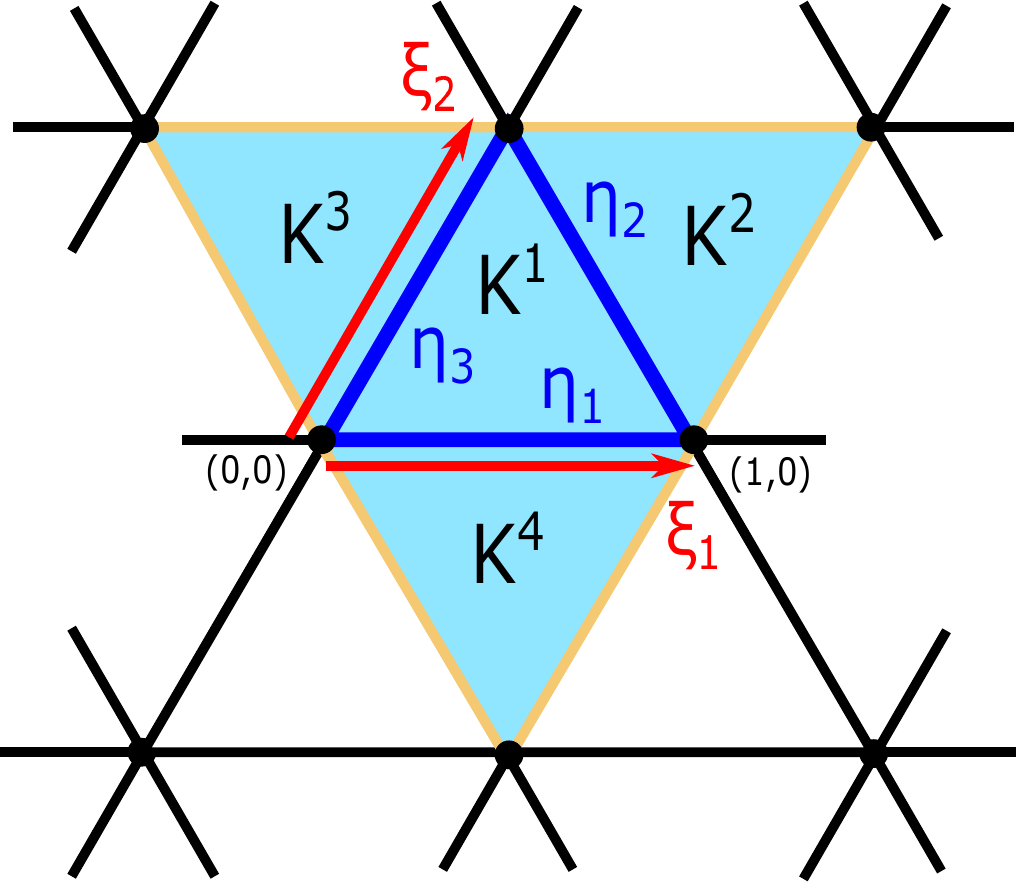}
\end{minipage}
\quad\quad\quad
\begin{minipage}{0.3\textwidth} 
\centering
\includegraphics[width=\textwidth]{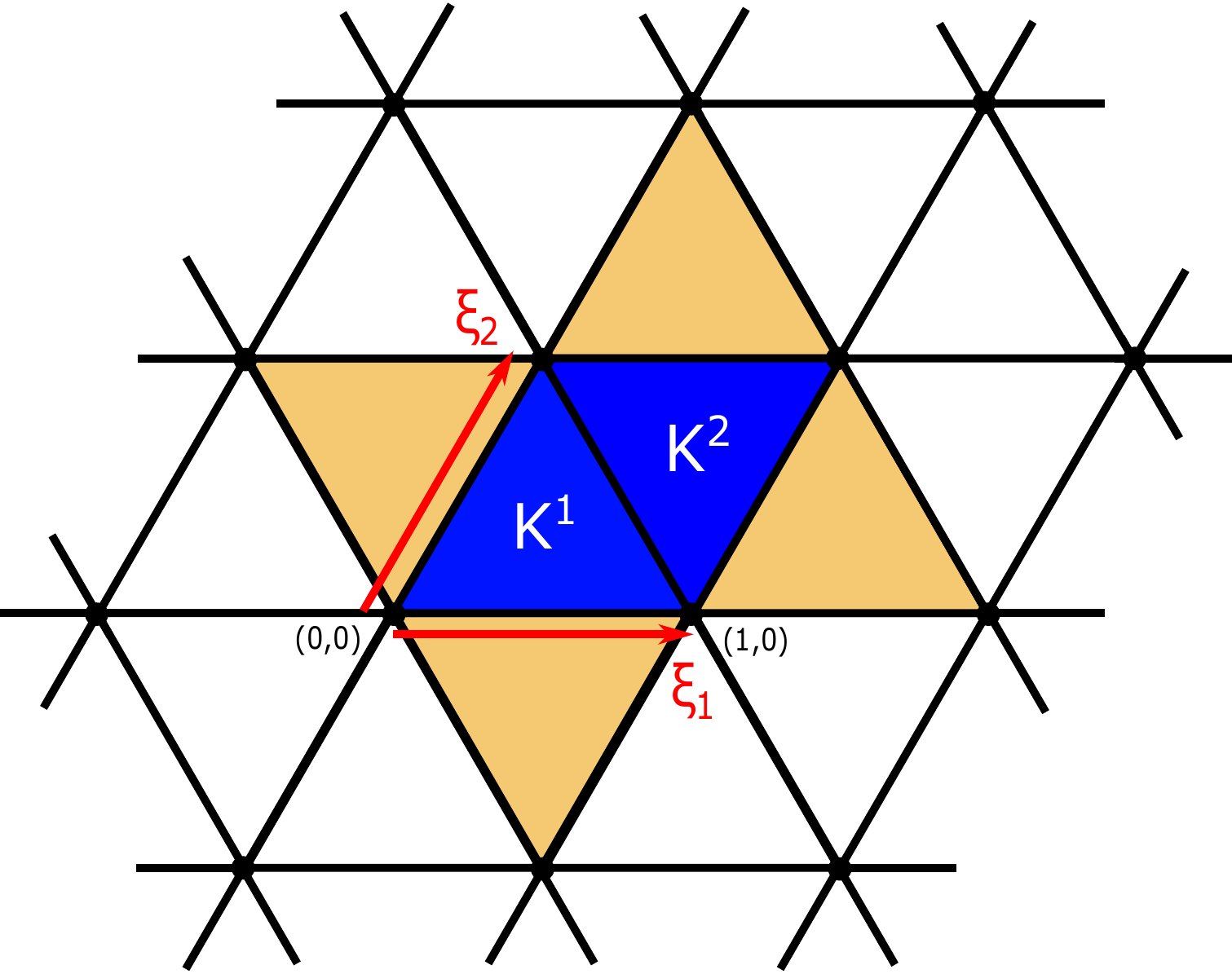}
\end{minipage}
\end{center}
\caption{Stencils of the basis functions related to the fundamental sets of edges (ncTVEM) and elements (PWDG), respectively, from \textit{left} to \textit{right}, when employing the meshes made of triangles in Figure~\ref{fig:meshes}.}
\label{fig:dispersion-triangles} 
\end{figure}

\begin{figure}[h]
\begin{center}
\begin{minipage}{0.32\textwidth}
\centering
\includegraphics[width=\textwidth]{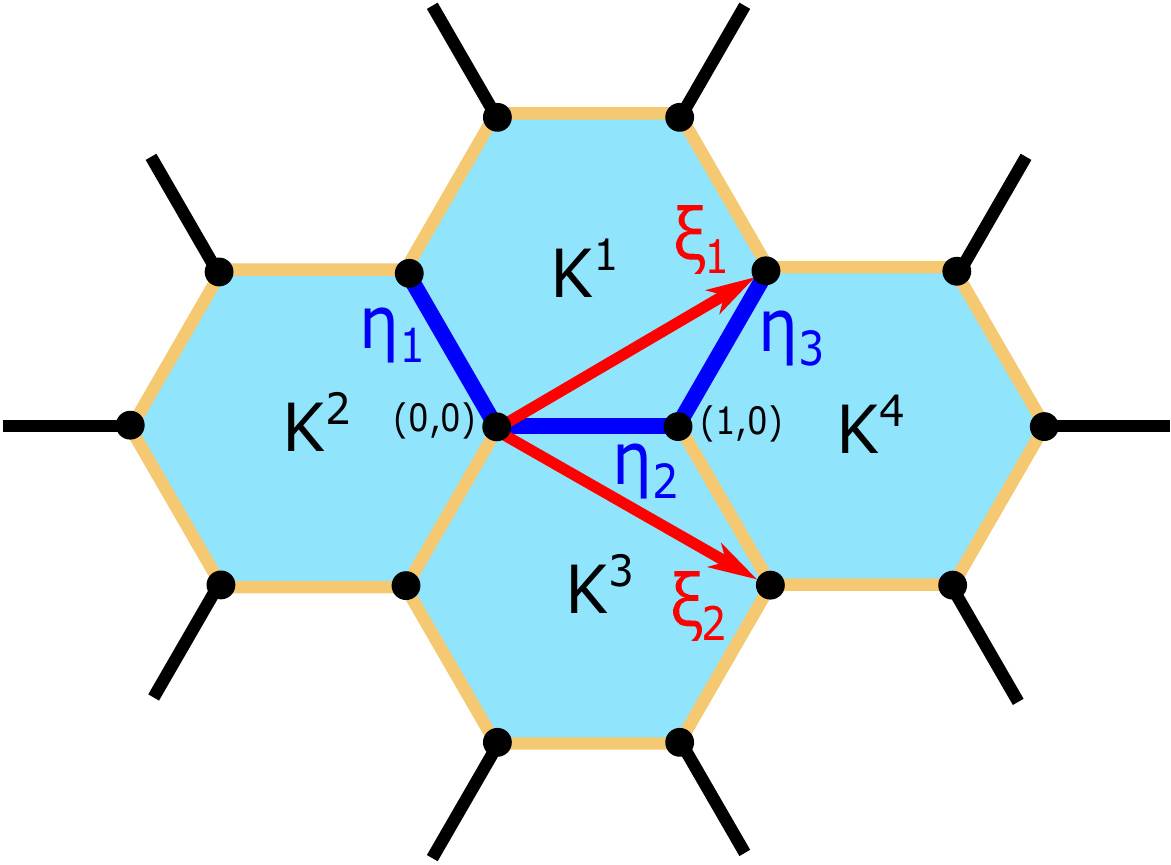}
\end{minipage}
\quad\quad\quad
\begin{minipage}{0.3\textwidth} 
\centering
\includegraphics[width=\textwidth]{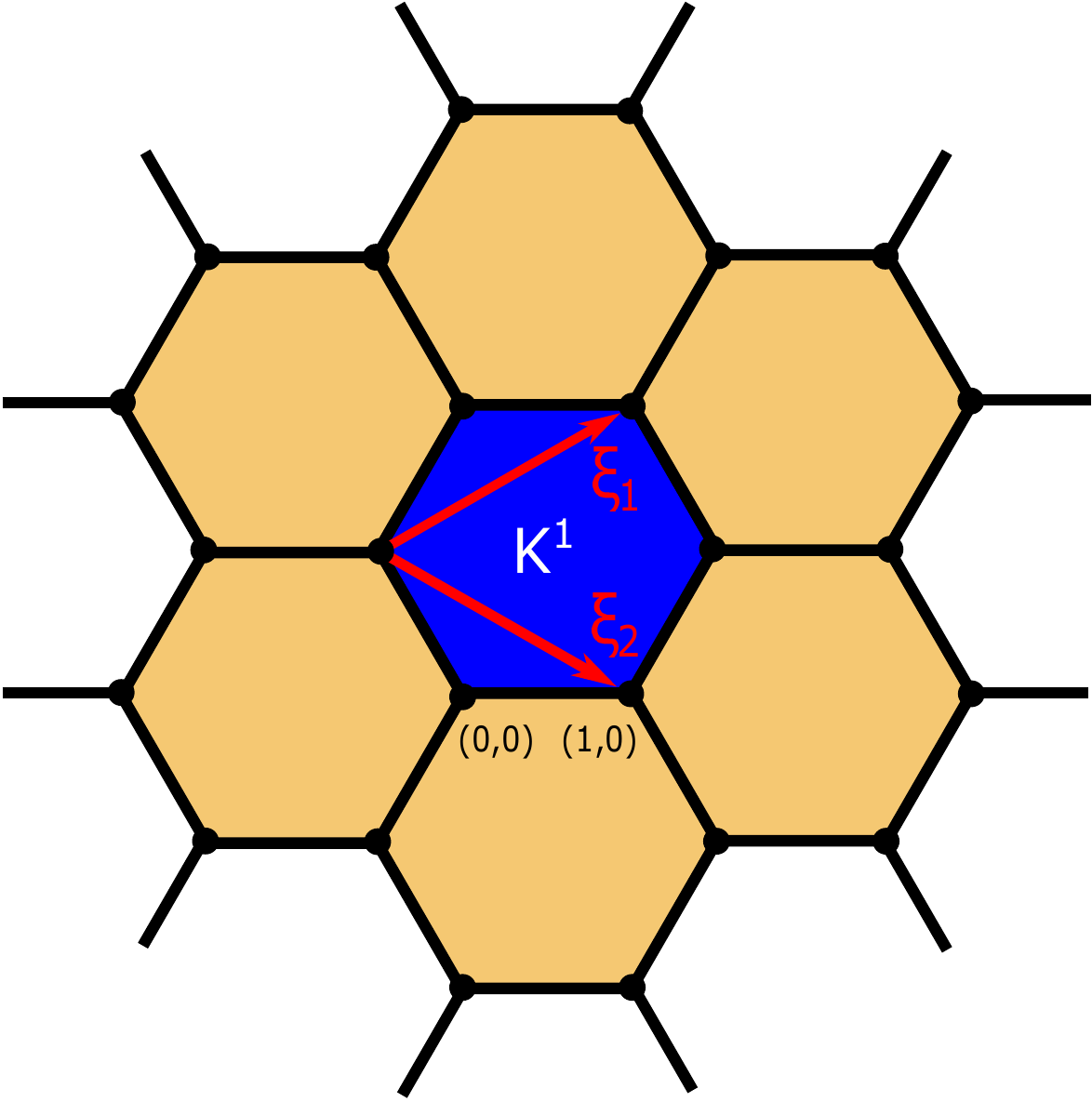}
\end{minipage}
\end{center}
\caption{Stencils of the basis functions related to the fundamental sets of edges (ncTVEM) and elements (PWDG), respectively, from \textit{left} to \textit{right}, when employing the meshes made of hexagons in Figure~\ref{fig:meshes}.}
\label{fig:dispersion-hexagons} 
\end{figure}

\subsection{Numerical results} \label{subsection:numerical-results}
In this section, after fixing some parameters for the different methods in Section~\ref{section:NR-Helmholtz} and specifying the quantities to be compared,
we present a series of numerical tests using the meshes portrayed in Figure~\ref{fig:meshes}.
Firstly, in Section~\ref{subsubsection:dependence_Bloch_angle}, we investigate the qualitative behaviour of dispersion and dissipation depending on the Bloch wave angle~$\theta$ in Definition~\ref{Bloch_wave}.
Then, in Section~\ref{subsubsection:disp_vs_q}, we compare the dispersion and dissipation errors against the effective plane wave degree~$q$ and against the dimensions of the minimal generating subspaces.
Finally, in Section~\ref{subsubsection:disp_vs_k}, the dependence of the errors on the wave number is studied.

\paragraph*{Choice of the parameters in PWDG and the stabilizations in the ncTVEM.}
We use the choice of the flux parameters of the ultra weak variational formulation (UWVF), i.e. $\alpha=\beta=1/2$, in PWDG, and we employ the stabilization terms suggested in~\cite{TVEM_Helmholtz, TVEM_Helmholtz_num} for the ncTVEM.

As for the ncTVEM, we employ a modified D-recipe stabilization detailed in~\cite[Section~4]{TVEM_Helmholtz_num}.
More precisely, for all~$\E \in \taun$, consider the set of local canonical basis function~$\{\varphi_i^H\}$ of~$\VhHE$.
For all~$\varphi_i^H$ and~$\varphi_j^H$ basis functions, we consider
\[
\SEH(\varphi_i^H, \varphi_j^H) = \aE(\PiHp \varphi_i^H, \PiHp \varphi_j^H).
\]
An essential element in the implementation of the method is the or\-tho\-go\-na\-li\-za\-tion-and-filtering process detailed in~\cite[Algorithm~2]{TVEM_Helmholtz_num}.
The basic idea is that the plane waves on each edge $\e$ used in the definition of the degrees of freedom are first orthogonalized in $L^2(\e)$.
Then, all combinations of plane waves that are close to be linearly dependent to others are eliminated.
This is explained in Algorithm~\ref{algorithm:ortho}.

\begin{algorithm}[h]
\caption{}
\label{algorithm:ortho}
Let $\sigma>0$ be a given threshold.
\begin{enumerate}
\item For all edges~$\e\in{\mathcal E}_\h$:
\begin{enumerate}
\item Assemble the matrix~${\mathbf G}^\e$ associated with the the
  $L^2(\e)$ inner product:
\[
({\mathbf G}^\e)_{j,\ell}= (\welle,w_j^\e)_{0,e} \quad \forall
j,\ell=1,\dots, \NPWe,
\]
where we recall that $\NPWe := \dim(\PWpe)$, and
$\{w_r^\e\}_{r=1}^{\NPWe}$ denotes the original basis of $\PWpe$.
\item Compute the eigenvalue/eigenvector decomposition of~${\mathbf G}^\e$:
\begin{equation*}
{\mathbf G}^\e {\mathbf Q}^\e =   {\mathbf Q}^\e {\mathbf \Lambda}^\e,
\end{equation*}
where~${\mathbf Q}^\e$ is a matrix whose columns are
right-eigenvectors of ${\mathbf G}^\e$, and~${\mathbf \Lambda}^\e$ is a diagonal matrix of the corresponding eigenvalues. 
\item Remove the columns of~${\mathbf Q}^\e $ corresponding to the eigenvalues with absolute value smaller than~$\sigma$.
Denote by $ \widehat \NPWe\le \NPWe $ the number of remaining columns, and re-label them by~$1,\dots, \widehat \NPWe$.
\item For $\ell=1,\dots, \widehat \NPWe$, set
\begin{equation*} \label{new_basis_fcts}
\widehat{w}_\ell^\e:=\sum_{r=1}^{\NPWe} {\mathbf Q}^\e_{r,\ell} \, w_r^\e.
\end{equation*}
The new, \emph{filtered basis} $\{\widehat{w}_\ell^\e\}_{\ell=1}^{\widehat \NPWe}$ is $L^2(\e)$ orthogonal.
\end{enumerate}
\item For all $\E\in\taun$, the new basis of $\VhHE$ is built by using the filtered basis
  $\{\widehat{w}_\ell^\e\}_{\ell=1}^{\widehat \NPWe}$ instead of
  the original basis $\{w_r^\e\}_{r=1}^{\NPWe}$ for each $\e\in\EE$.
\end{enumerate}
\end{algorithm}

A consequence of this approach is that, after few steps of both the $\h$- and $\p$- versions of the method,
the accuracy improves with an extremely slow growth of the number of degrees of freedom.
This results in the so called \textit{cliff-edge effect}, which was observed in~\cite{TVEM_Helmholtz_num, fluidfluid_ncVEM};
we shall exhibit such fast decay of the error notably for the $\p$-version in Figures~\ref{fig:disp_diss_q_sq}, \ref{fig:disp_diss_q_tri}, \ref{fig:rel_dispersion_nr_dofs}, and~\ref{fig:disp_diss_q_sq FEM}.

\paragraph*{Numerical quantities.}
Given a wave number~$\k>0$ and~$\kn$ the discrete wave number in Definition~\ref{def discrete wave number}, we will study the following quantities:
\begin{itemize}
\item the \textit{dispersion error} $|\Real{(\k-\kn)}|$, which describes the difference of the propagation velocities of the continuous and discrete plane wave solutions;
\item the \textit{dissipation error} $|\Imag{(\kn)}|=|\Imag{(\k-\kn)}|$, which represents the difference of the amplitudes (damping) of the continuous and discrete plane wave solution;
\item the \textit{total error} $|\k-\kn|$, which measures the total deviation of the continuous and discrete wave numbers.
\end{itemize}

\subsubsection{Dependence of dispersion and dissipation on the Bloch wave angle} \label{subsubsection:dependence_Bloch_angle}
In this section, we study dispersion and dissipation of the different methods in dependence on the angle~$\theta$ of the direction~$\dir$ in the definition of the Bloch wave in~\eqref{Bloch_wave}.
Importantly, we are interested in a qualitative comparison of the methods. A quantitative comparison should be performed in terms of the dimensions of the minimal generating subspaces instead of the effective degrees, and is discussed in Section~\ref{subsubsection:disp_vs_q} below.

To this purpose, in Figures~\ref{fig:dispersion-angle-sq}-\ref{fig:dispersion-angle-hex}, the numerical quantities $|\Real{(k-\kn)}|$ and $|\Imag{(\kn)}|$ are plotted against~$\theta$ for the meshes made of squares, triangles, and hexagons, respectively, shown in Figure~\ref{fig:meshes}.
We took~$\k=3$ and~$q=7$ for all those types of meshes (Figures~\ref{fig:dispersion-angle-sq}-\ref{fig:dispersion-angle-hex}, left).
Moreover, for~$\k=10$, we chose~$q=10$ for the squares (Figure~\ref{fig:dispersion-angle-sq}, right) and the triangles (Figure~\ref{fig:dispersion-angle-tri}, right), and~$q=13$ for the hexagons (Figure~\ref{fig:dispersion-angle-hex}, right).
We remark that the latter choice for~$q$ on the meshes made of hexagons is purely for demonstration purposes, in order to obtain a reasonable range for the errors, where one can see the behaviour more clearly.
Moreover, the wave number~$k$ here (mesh size~$h=1$) corresponds to the wave number~$k_0=\frac{k}{h_0}$ on a mesh with mesh size~$h_0$.

\begin{figure}[h]
\begin{center}
\begin{minipage}{0.475\textwidth} 
\centering
\includegraphics[width=\textwidth]{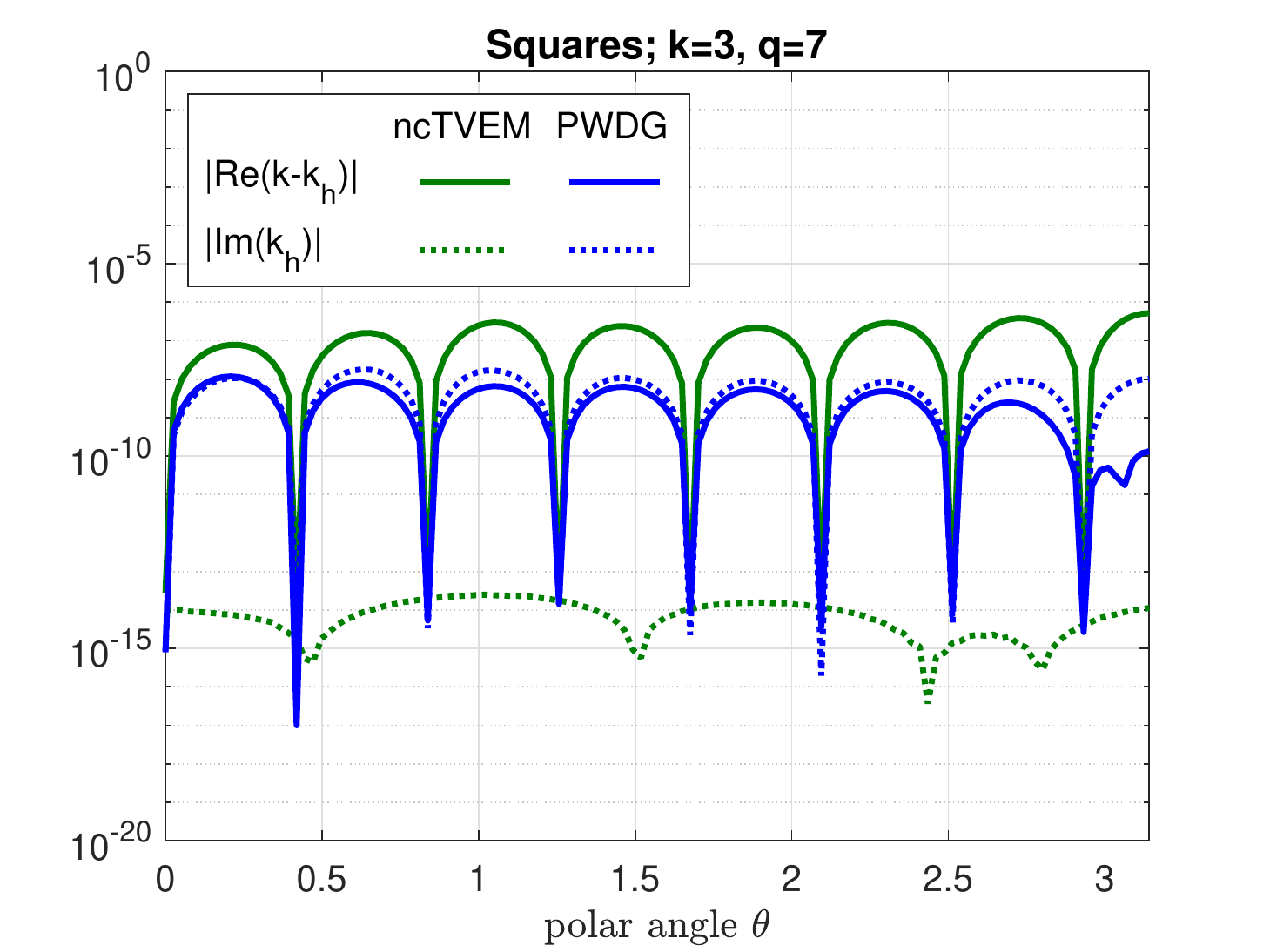}
\end{minipage}
\hfill
\begin{minipage}{0.475\textwidth}
\centering
\includegraphics[width=\textwidth]{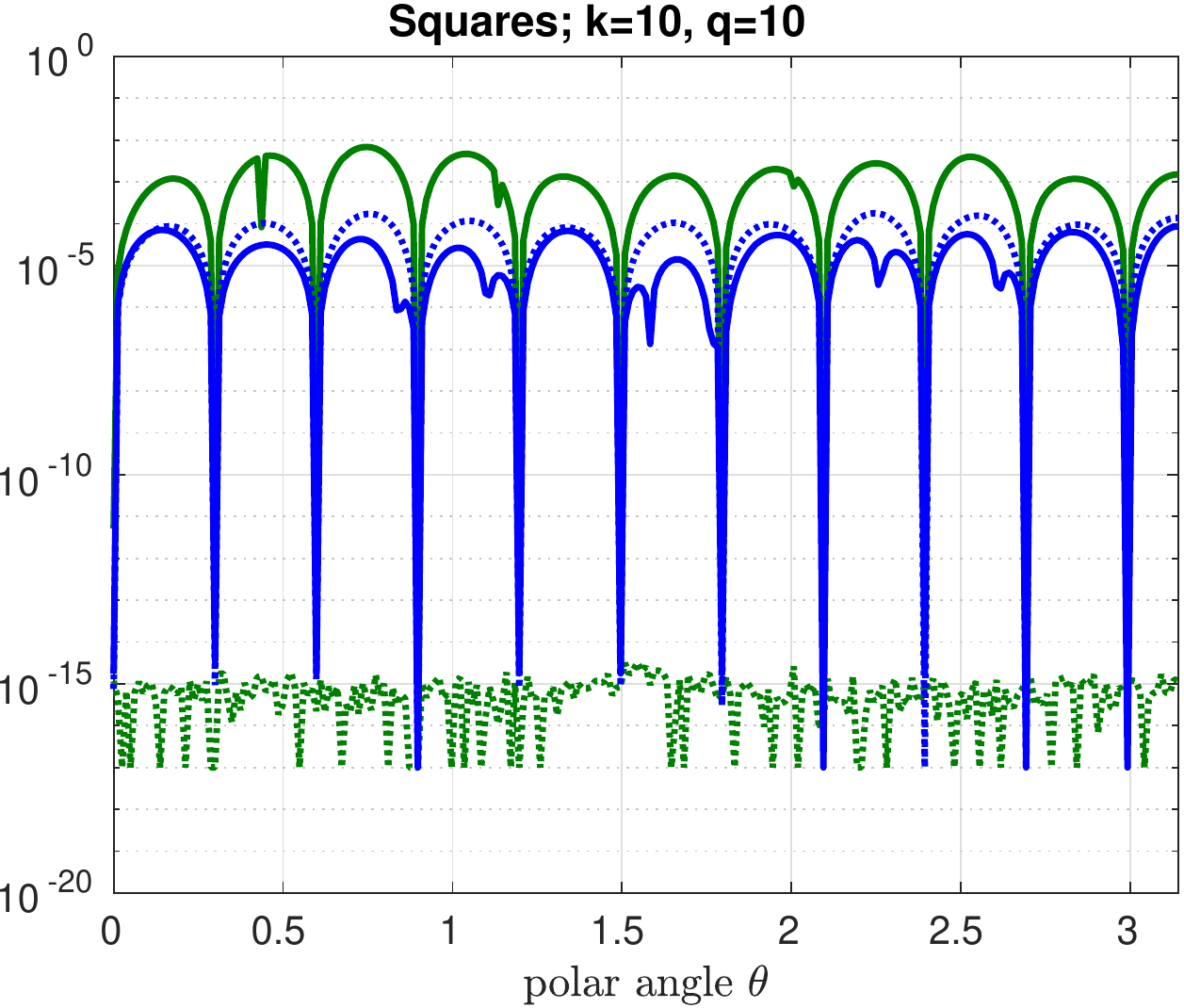}
\end{minipage}
\end{center}
\caption{Dispersive and dissipative behaviour of PWDG and ncTVEM in dependence on the polar angle~$\theta$ of the Bloch wave direction~$\dir$ in~\eqref{Bloch_wave} on the meshes made of squares in Figure~\ref{fig:meshes},
with~$\k=3$ and~$q=7$ (\textit{left}), and~$k=10$ and~$q=10$ (\textit{right}).}
\label{fig:dispersion-angle-sq} 
\end{figure}

\begin{figure}[h]
\begin{center}
\begin{minipage}{0.475\textwidth} 
\centering
\includegraphics[width=\textwidth]{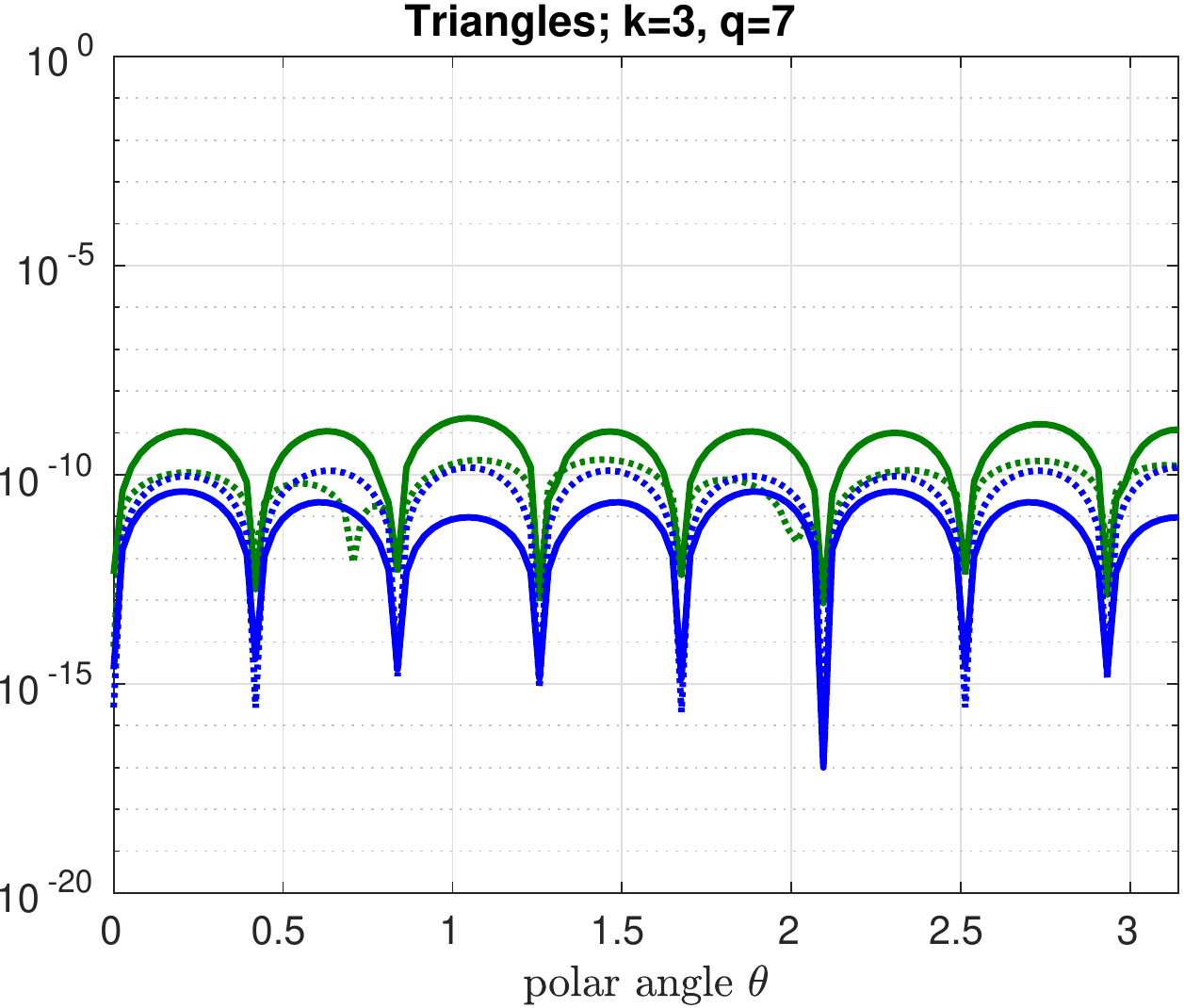}
\end{minipage}
\hfill
\begin{minipage}{0.475\textwidth}
\centering
\includegraphics[width=\textwidth]{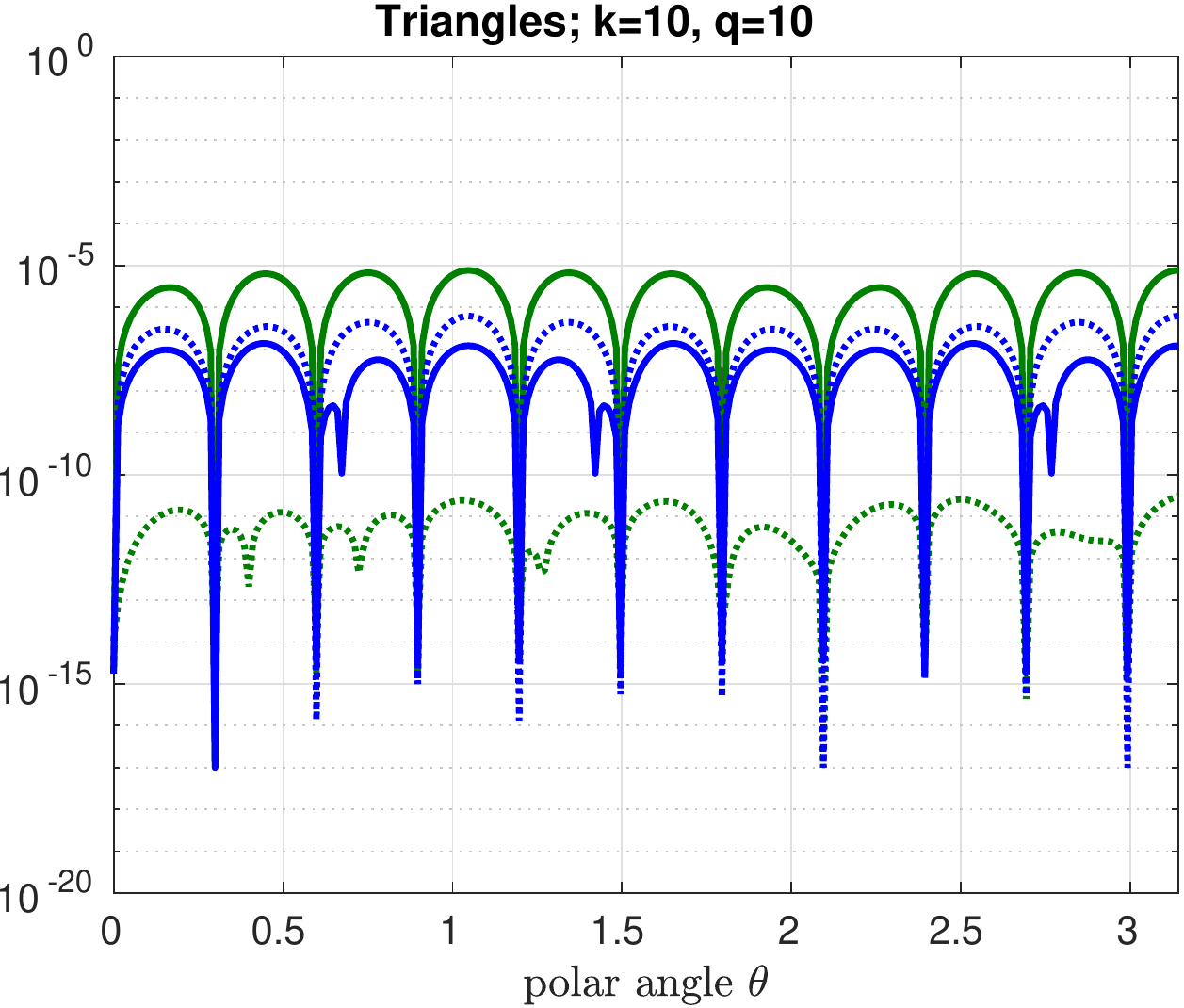}
\end{minipage}
\end{center}
\caption{Dispersive and dissipative behaviour of PWDG and ncTVEM in dependence on the polar angle $\theta$ of the Bloch wave direction $\dir$ in~\eqref{Bloch_wave} on the meshes made of triangles in Figure~\ref{fig:meshes}, with $k=3$ and $q=7$ (\textit{left}), and $k=10$ and $q=1$ (\textit{right}). The colour legend is the same as in Figure~\ref{fig:dispersion-angle-sq}.}
\label{fig:dispersion-angle-tri} 
\end{figure}

\begin{figure}[h]
\begin{center}
\begin{minipage}{0.475\textwidth} 
\centering
\includegraphics[width=\textwidth]{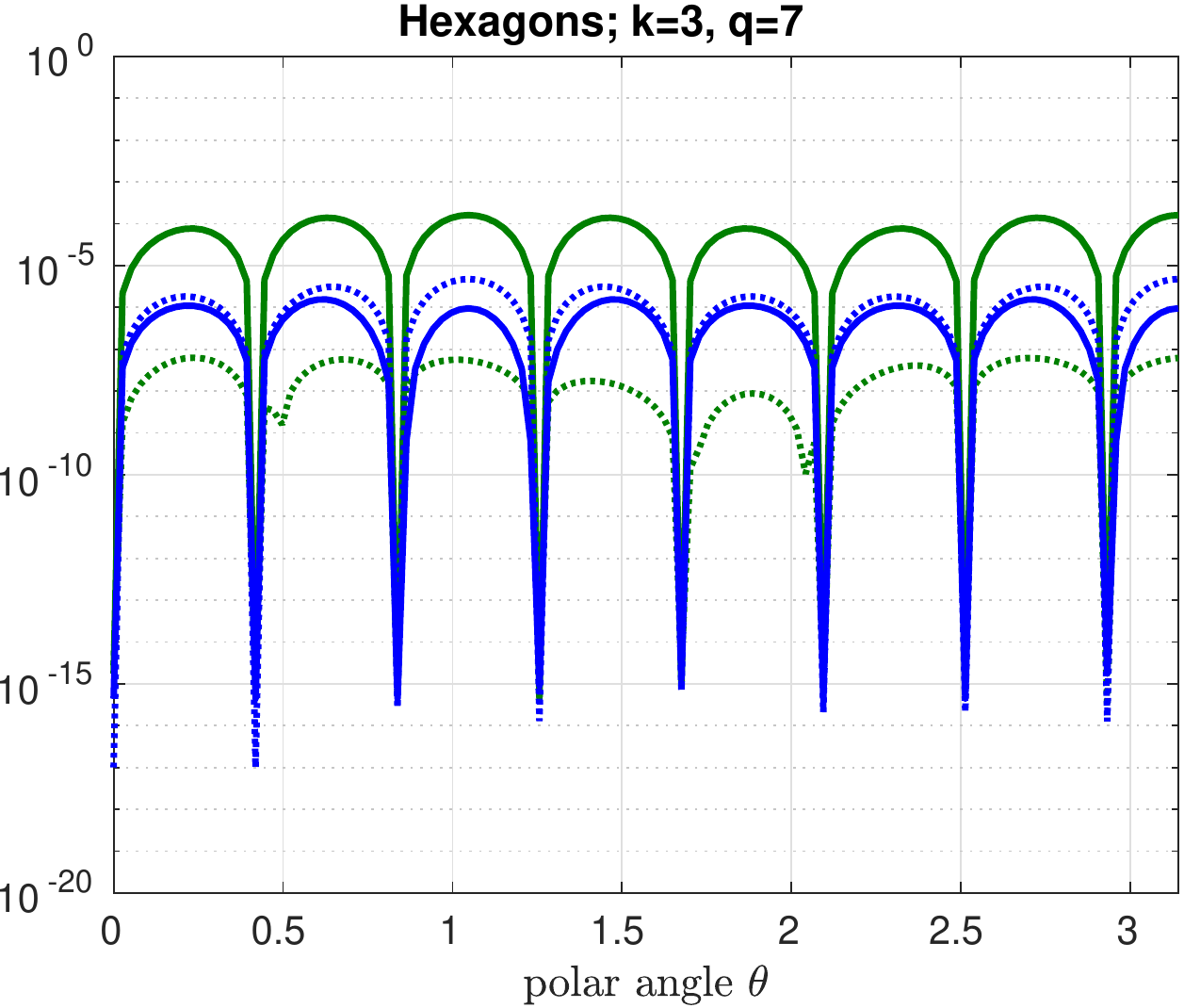}
\end{minipage}
\hfill
\begin{minipage}{0.475\textwidth}
\centering
\includegraphics[width=\textwidth]{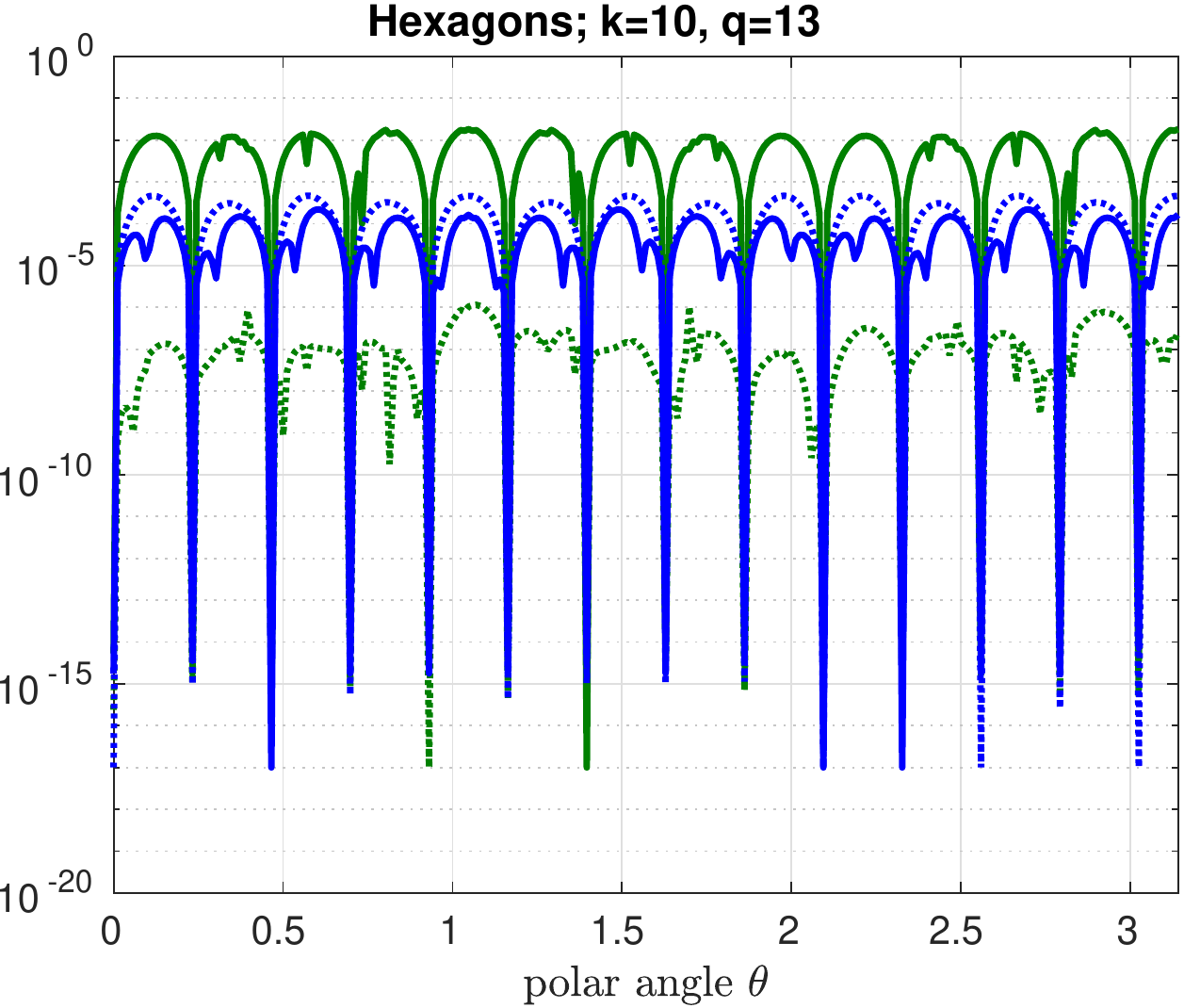}
\end{minipage}
\end{center}
\caption{Dispersive and dissipative behaviour of PWDG and ncTVEM in dependence on the polar angle~$\theta$ of the Bloch wave direction~$\dir$ in~\eqref{Bloch_wave} on the meshes made of hexagons in Figure~\ref{fig:meshes}, with~$\k=3$ and $q=7$ (\textit{left}),
and~$\k=10$ and~$q=13$ (\textit{right}). The colour legend is the same as in Figure~\ref{fig:dispersion-angle-sq}.}
\label{fig:dispersion-angle-hex} 
\end{figure}

The dispersion and dissipation are zero, up to machine precision, for choices of the Bloch wave direction~$\dir$ in~\eqref{Bloch_wave} coinciding with one of the plane wave directions $\{\djj\}_{j=1}^p$
(here we always took equidistributed directions $\djj$, where $\textup{\textbf{d}}_1=(1,0)$).
This follows directly from the fact that, in this case, the Bloch wave satisfying~\eqref{var_form} coincides with the corresponding plane wave traveling along the direction $\dir$.
Moreover, for the ncTVEM, the dispersion error dominates the dissipation error, whereas for PWDG the dissipation dominates the dispersion.

\subsubsection{Exponential convergence of the dispersion error against the effective degree~$q$} \label{subsubsection:disp_vs_q}
Here, we investigate the dependence of dispersion and dissipation on the effective plane wave degree~$q$ (namely, $p=2q+1$ bulk plane waves).
For fixed wave number~$\k$, we will observe exponential convergence of the total error for increasing~$q$, as already seen in~\cite{gittelson} for PWDG.
This result is not unexpected since also the $\p$-versions for the discretization errors have exponential convergence, provided that the exact analytical solution is smooth;
see~\cite{TDGPW_pversion} for PWDG, and the numerical experiments in~\cite{TVEM_Helmholtz} for the ncTVEM, respectively.
Moreover, we will make a comparison of these methods in terms of the total error versus the dimensions of the minimal generating subspaces.

To this purpose, we consider the following range for the wave number: $k \in \{2,3,4,5\}$. We recall again that~$k$ corresponds to~$k_0=\frac{k}{h_0}$ on a mesh with mesh size~$h_0$. 

\paragraph*{Dispersion and dissipation vs. effective degree $q$.}

In Figures~\ref{fig:disp_diss_q_sq}-\ref{fig:disp_diss_q_hex}, the relative dispersion error $|\Real(k-\kn)|/k$ and the relative damping error $|\Imag(\kn)|/k$ are displayed against~$q$, for the meshes made of squares, triangles, and hexagons, respectively.
The maxima of the relative dispersion and the relative dissipation, respectively, are taken over a large set of Bloch wave directions~$\dir$.
After a preasymptotic regime, we observe exponential convergence of the dispersion error for all methods and the dissipation error for the PWDG.
The dispersion error is consistently smaller for the PWDG than for the ncTVEM.

\begin{figure}[h]
\begin{center}
\begin{minipage}{0.485\textwidth} 
\centering
\includegraphics[width=\textwidth]{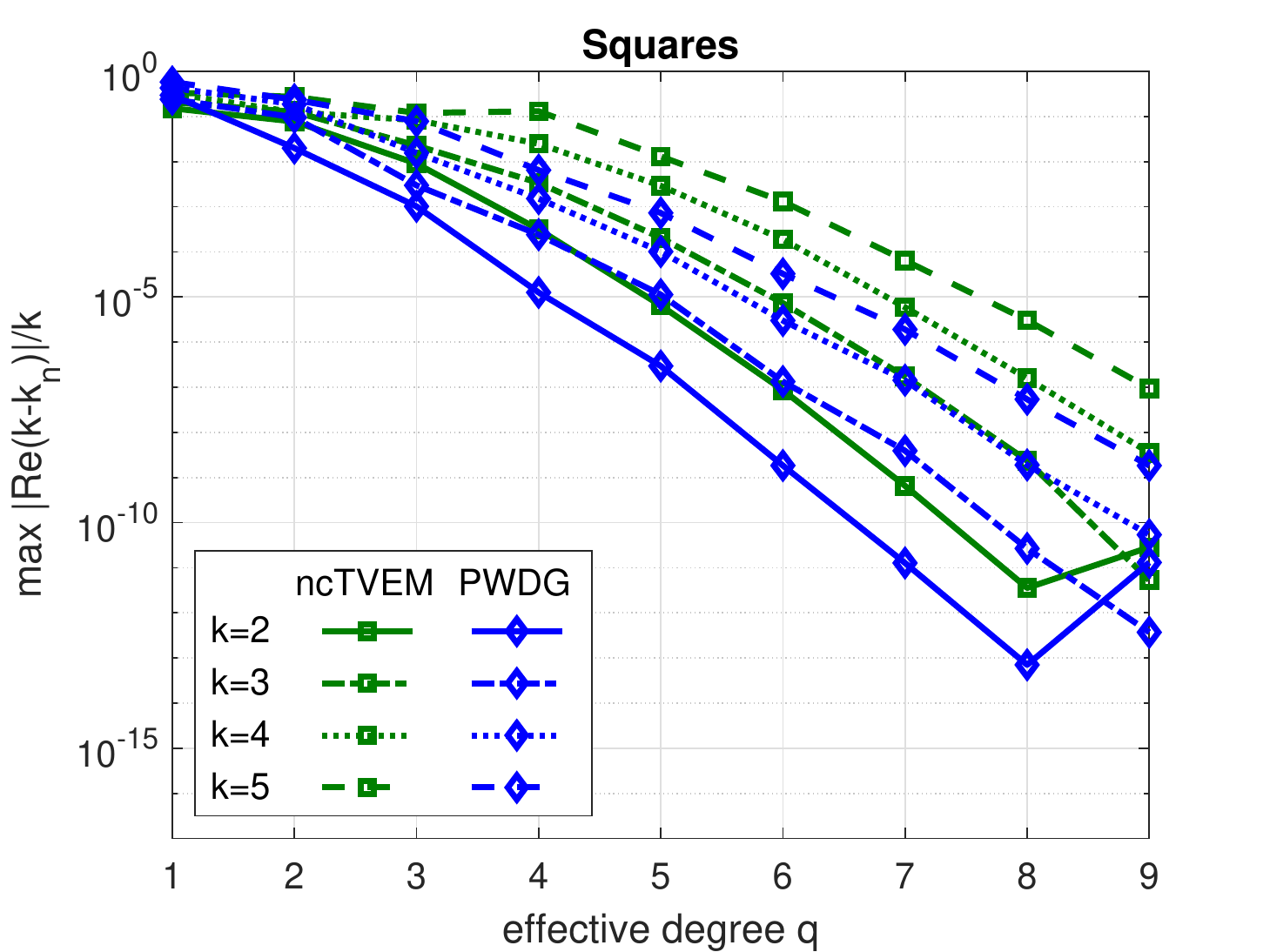}
\end{minipage}
\hfill
\begin{minipage}{0.485\textwidth}
\centering
\includegraphics[width=\textwidth]{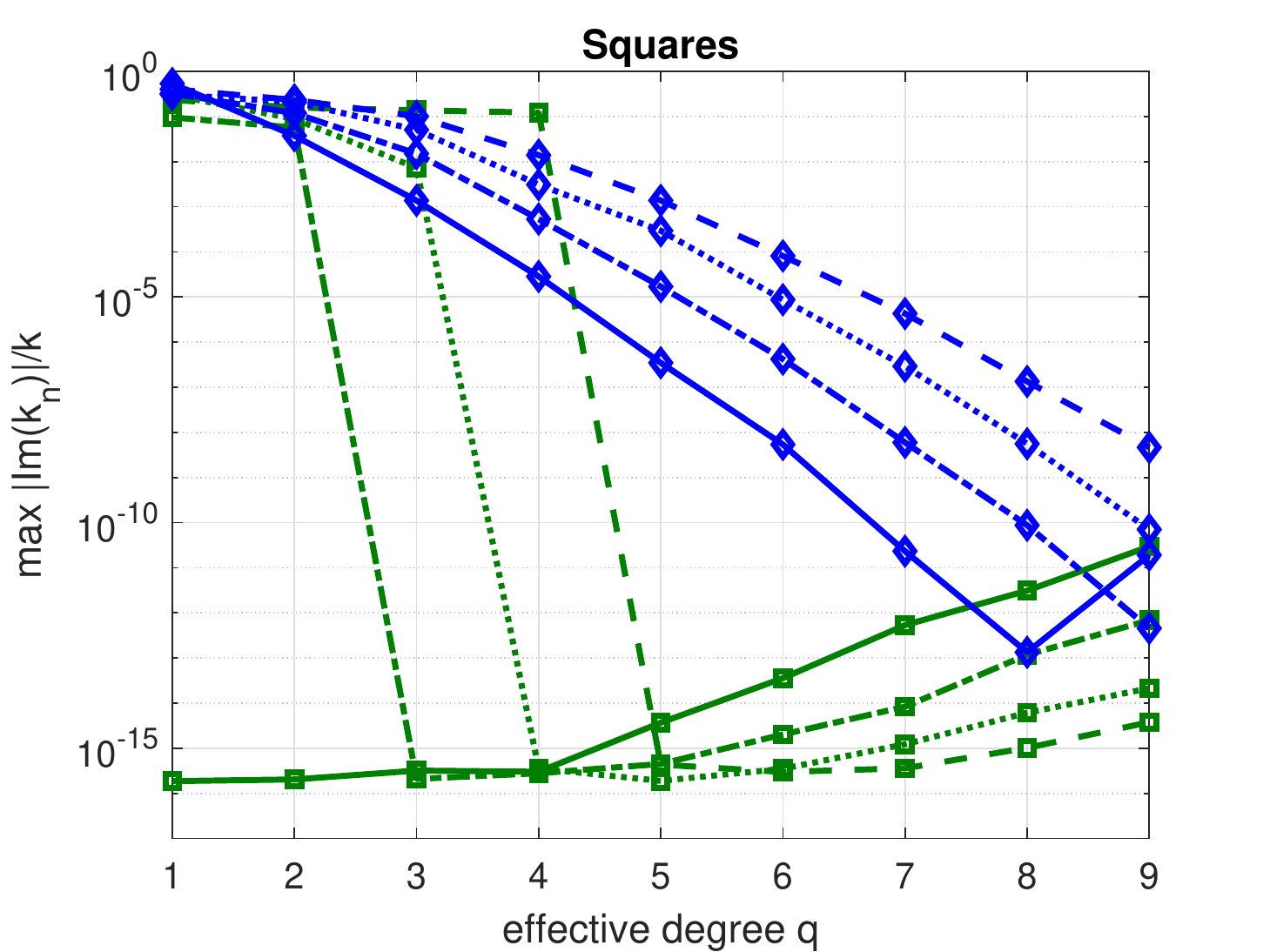}
\end{minipage}
\end{center}
\caption{Relative dispersion (\textit{left}) and relative dissipation (\textit{right}) for the different methods in dependence on the effective degree $q$ and the wave numbers $k=2,\dots,5$ on the meshes made of squares in Figure~\ref{fig:meshes}.
The maxima over a large set of Bloch wave directions $\dir$ are taken.}
\label{fig:disp_diss_q_sq} 
\end{figure}

\begin{figure}[h]
\begin{center}
\begin{minipage}{0.485\textwidth} 
\centering
\includegraphics[width=\textwidth]{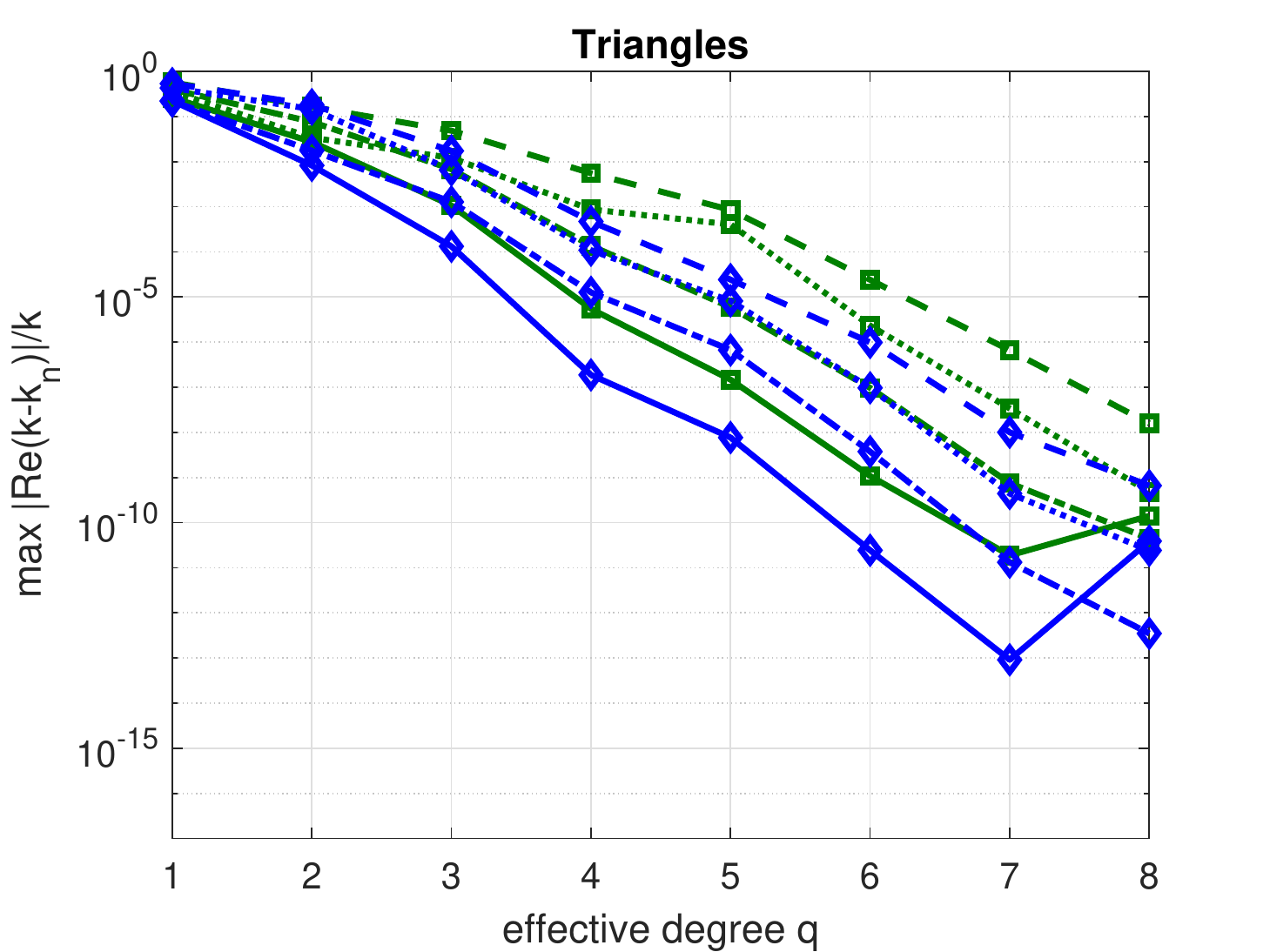}
\end{minipage}
\hfill
\begin{minipage}{0.485\textwidth}
\centering
\includegraphics[width=\textwidth]{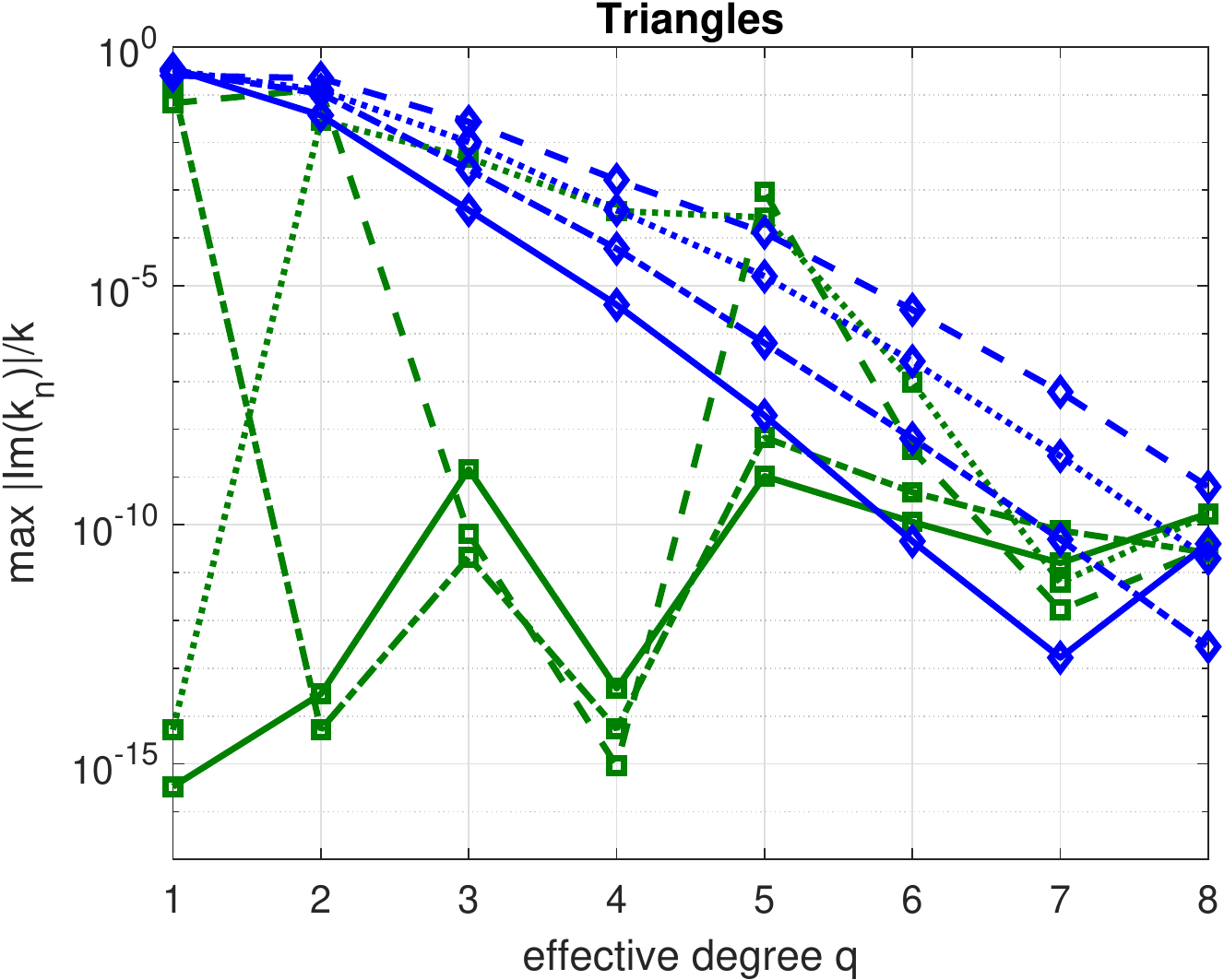}
\end{minipage}
\end{center}
\caption{Relative dispersion (\textit{left}) and relative dissipation (\textit{right}) for the different methods in dependence on the effective degree $q$ and the wave numbers $k=2,\dots,5$ on the meshes made of triangles in Figure~\ref{fig:meshes}.
The maxima over a large set of Bloch wave directions $\dir$ are taken. The colour legend is the same as in Figure~\ref{fig:disp_diss_q_sq}.}
\label{fig:disp_diss_q_tri} 
\end{figure}

\begin{figure}[h]
\begin{center}
\begin{minipage}{0.485\textwidth} 
\centering
\includegraphics[width=\textwidth]{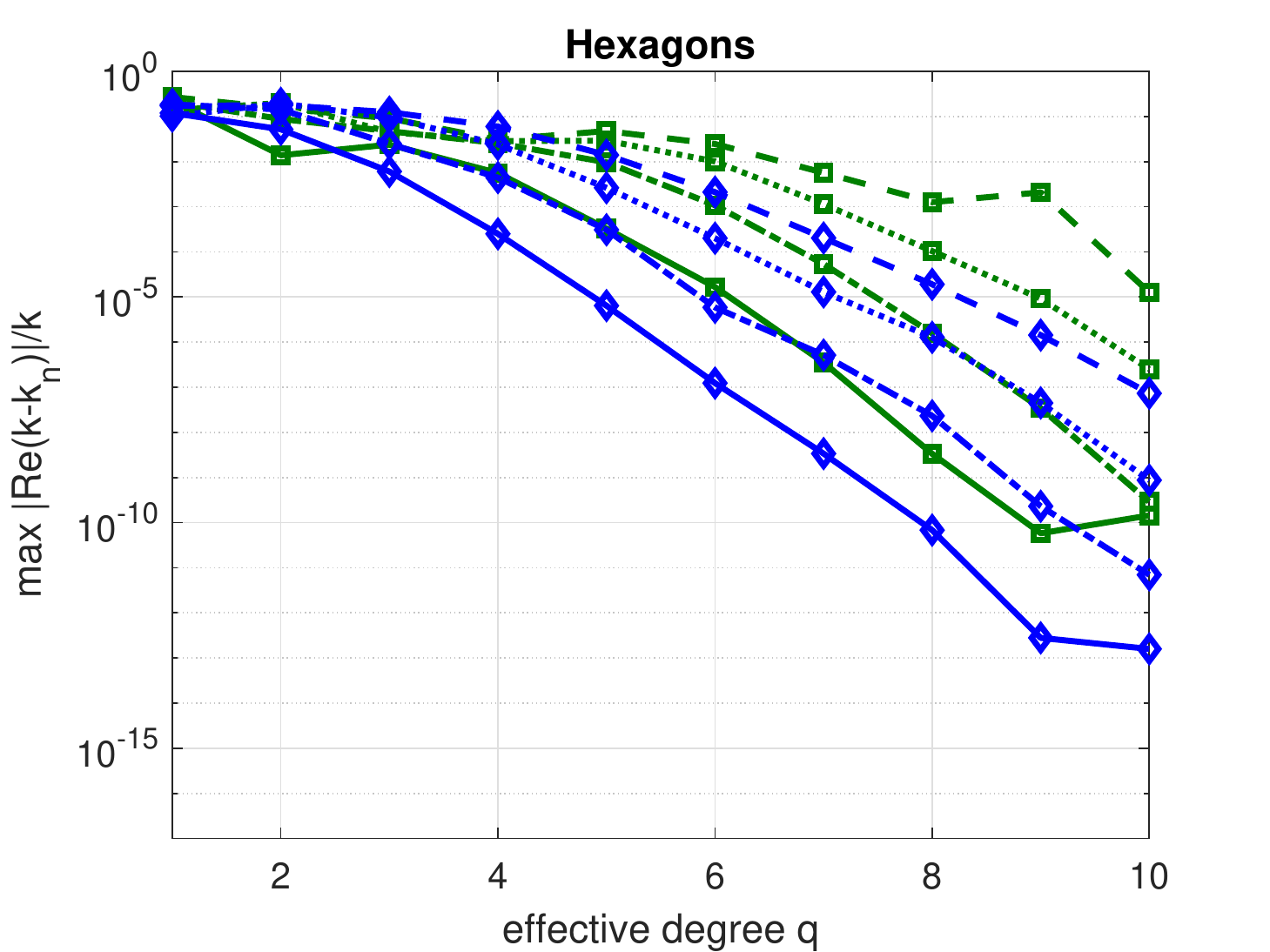}
\end{minipage}
\hfill
\begin{minipage}{0.485\textwidth}
\centering
\includegraphics[width=\textwidth]{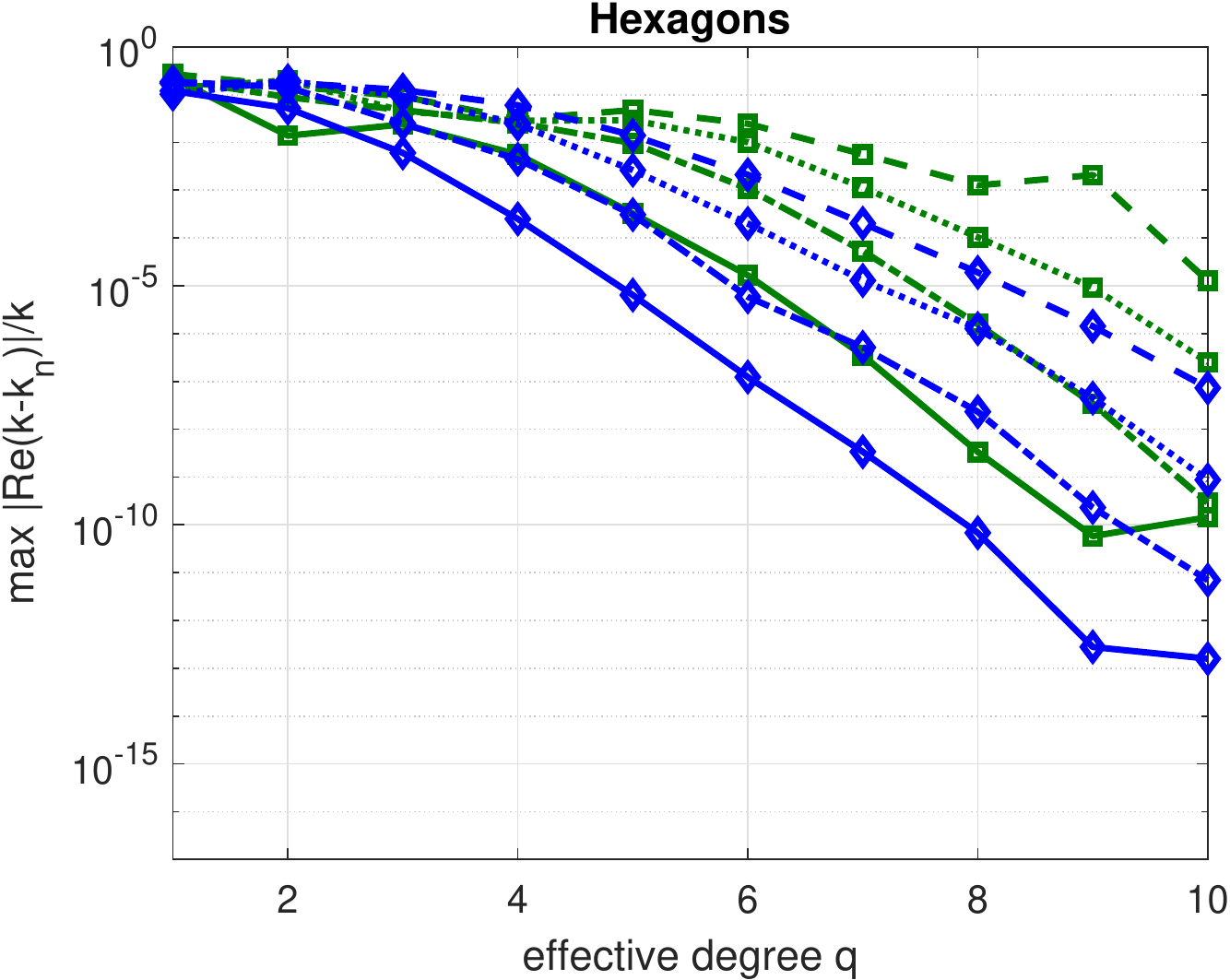}
\end{minipage}
\end{center}
\caption{Relative dispersion (\textit{left}) and relative dissipation (\textit{right}) for the different methods in dependence on the effective degree $q$ and the wave numbers $k=2,\dots,5$ on the meshes made of hexagons in Figure~\ref{fig:meshes}.
The maxima over a large set of Bloch wave directions $\dir$ are taken. The colour legend is the same as in Figure~\ref{fig:disp_diss_q_sq}.}
\label{fig:disp_diss_q_hex} 
\end{figure}

\paragraph*{Dispersion and dissipation vs. dimensions of minimal generating subspaces.}
From a computational point of view, it is also important to consider a comparison of the dispersion errors in terms of the dimensions of the minimal generating subspaces (density of the degrees of freedom).
We directly compare the relative total dispersion errors $|\kn-\k|/\k$, thus measuring the total deviation of the discrete wave number from the continuous one.
As above, the maxima over a large set of Bloch wave directions are taken. In Figure~\ref{fig:rel_dispersion_nr_dofs}, those errors are displayed for the meshes in Figure~\ref{fig:meshes}.
For the ncTVEM, we can recognize the \textit{cliff-edge effect}, meaning that, at some point, the dispersion error decreases without increase of the dimension of the minimal generating subspace.
This effect has already been remarked in~\cite{TVEM_Helmholtz_num, fluidfluid_ncVEM} for the discretization error and is a peculiarity of the orthogonalization-and-filtering process mentioned in~\cite[Algorithm~1]{TVEM_Helmholtz_num}.
Moreover, we can observe a direct correlation between the density of the degrees of freedom, which depends on the shape of the meshes, see Figures~\ref{fig:dispersion-squares}-\ref{fig:dispersion-hexagons}, 
and the error plots (larger cardinalities of the fundamental sets lead to larger errors; as mentioned above, for ncTVEM, the filtering process leads to dimensionality reductions).

\begin{figure}[h]
\centering
\includegraphics[width=0.485\textwidth]{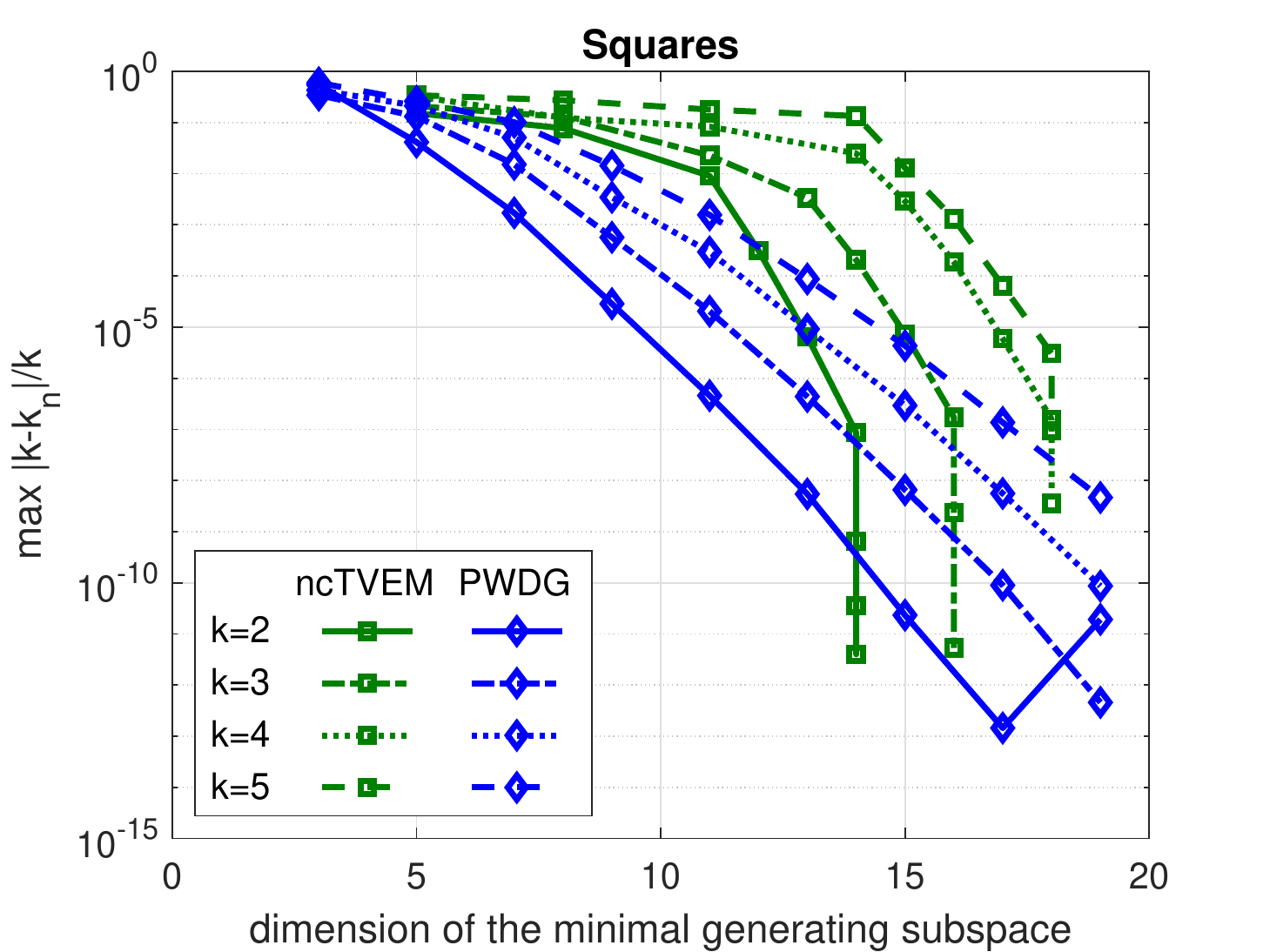}
\hfill
\includegraphics[width=0.485\textwidth]{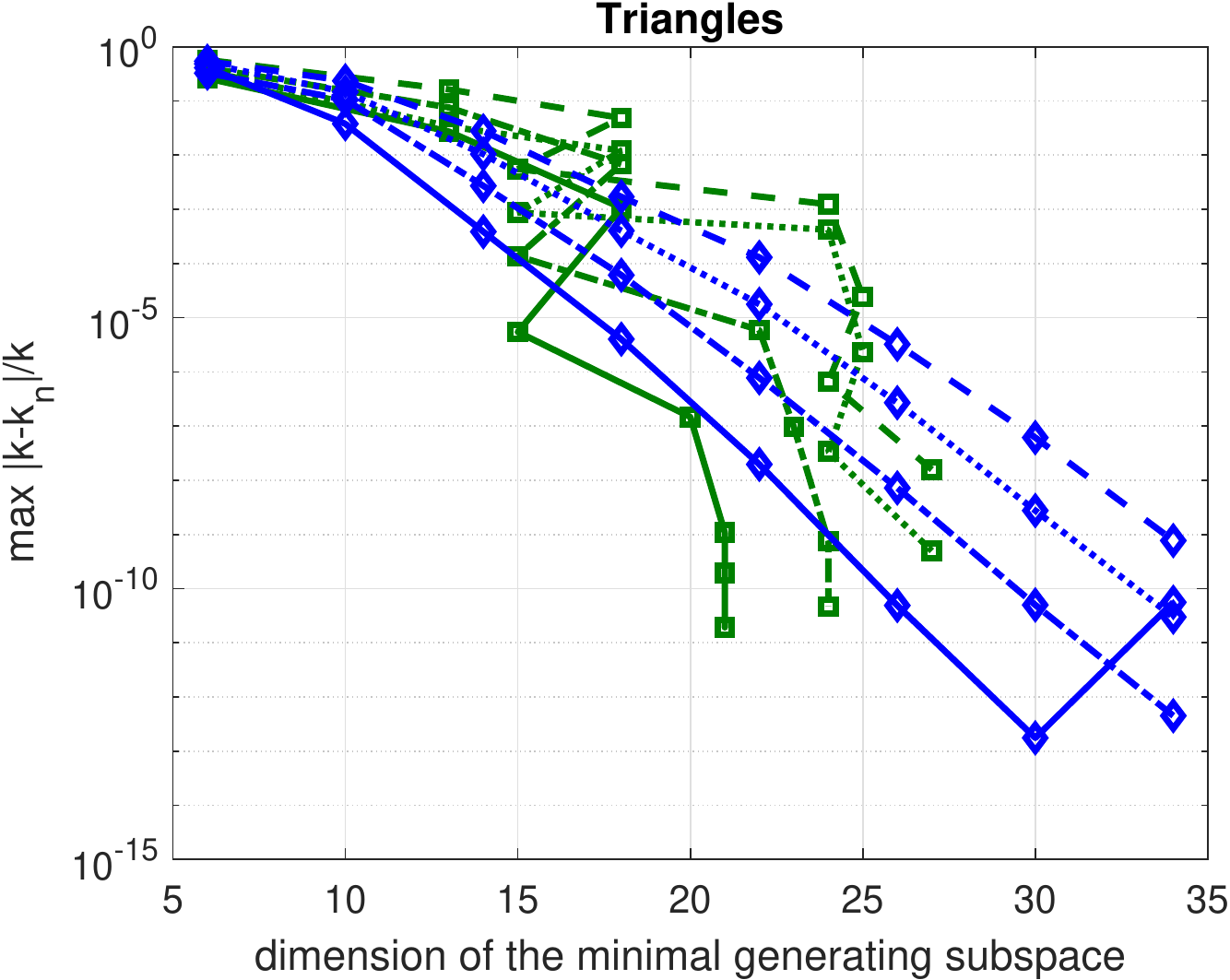}
\hfill
\includegraphics[width=0.485\textwidth]{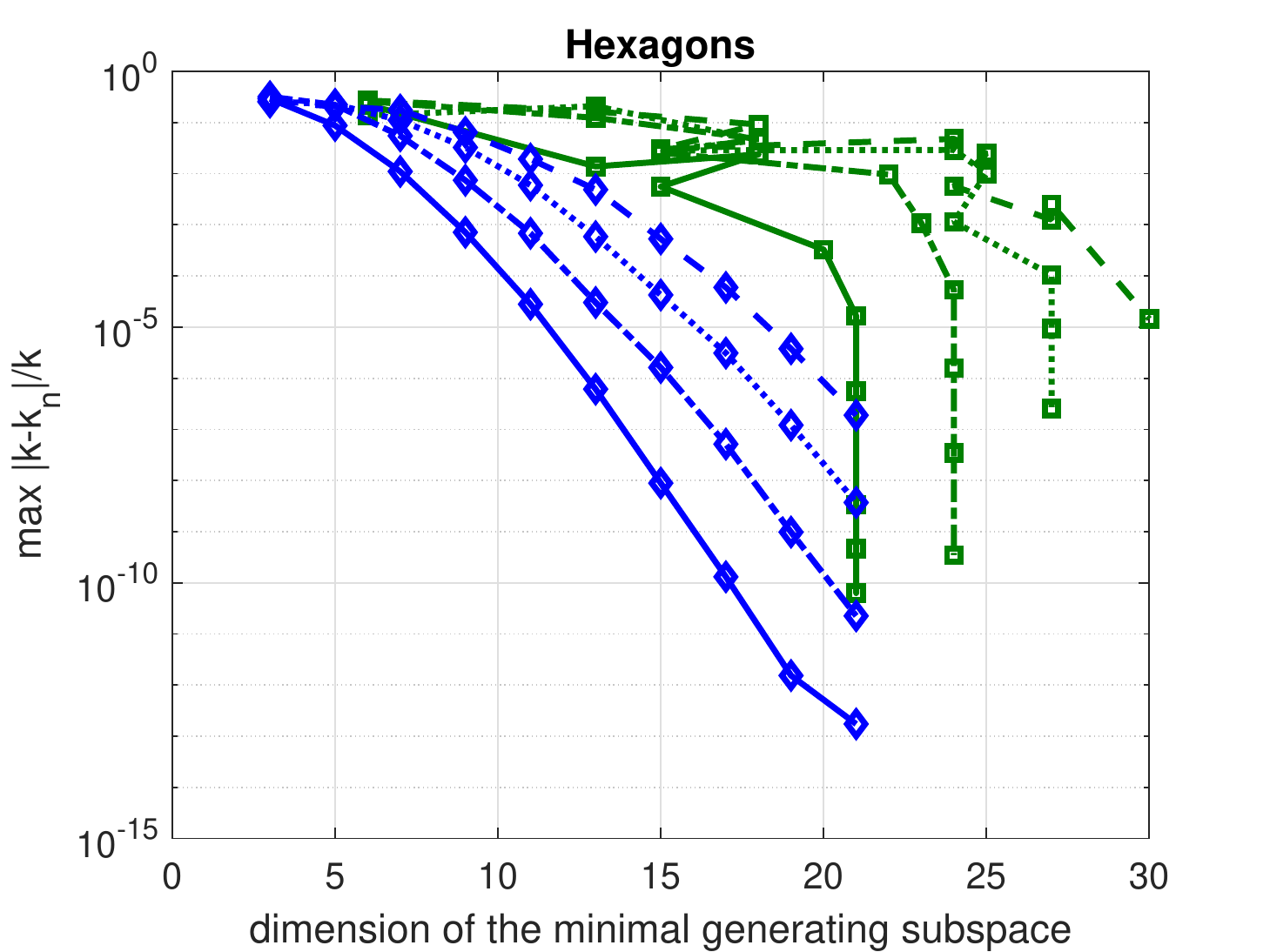}
\caption{Relative total dispersion error in dependence on the dimensions of the minimal generating subspaces for different values of~$k$ on the meshes in Figure~\ref{fig:meshes}.}
\label{fig:rel_dispersion_nr_dofs}
\end{figure}

\paragraph*{Comparison with the standard FEM.}

Here, we highlight the advantages of using full Trefftz methods (ncTVEM, PWDG) in comparison to standard polynomial based methods, such as the FEM,
whose dispersion properties were studied in, e.g., \cite{deraemaeker,sauterpollution,ihlenburg1995dispersion,ainsworth}.
For simplicity, we focus on the meshes made of squares in Figure~\ref{fig:meshes}, since, in this case, the basis functions in the FEM
have a tensor product structure and an explicit dispersion relation can be derived~\cite[Theorem 3.1]{ainsworth}:
\begin{equation} \label{eq:disp_rel_FEM}
\cos(\kn)=R_q(k),
\end{equation}
where, denoting by $[\cdot/\cdot]_{z \cot z}$ and $[\cdot/\cdot]_{z \tan z}$ the Pad\'e approximants to the functions $z \cot z$ and $z \tan z$, respectively,
\begin{equation*}
R_q(2 z):=\frac{[2N_0/2N_0-2]_{z \cot z}-[2N_e+2/2N_e]_{z \tan z}}{[2N_0/2N_0-2]_{z \cot z}+[2N_e+2/2N_e]_{z \tan z}},
\end{equation*}
with $N_0:=\lfloor (q+1)/2 \rfloor$ and $N_e:=\lfloor q/2 \rfloor$.
From~\eqref{eq:disp_rel_FEM}, one can see that only dispersion plays a role in the FEM.
In Figure~\ref{fig:disp_diss_q_sq FEM}, we display the relative total dispersion errors against the effective degree~$q$ (left) and against the dimensions of the minimal generating subspaces (right) for fixed $k=3$.
Similar results are obtained for other values of $k$ and are not shown. One can clearly notice that the dispersion error for the FEM is lower than for the other methods,
when comparing it in terms of~$q$, but higher, when comparing it in terms of the dimensions of the minimal generating subspaces.   

\begin{figure}[h]
\begin{center}
\begin{minipage}{0.485\textwidth} 
\centering
\includegraphics[width=\textwidth]{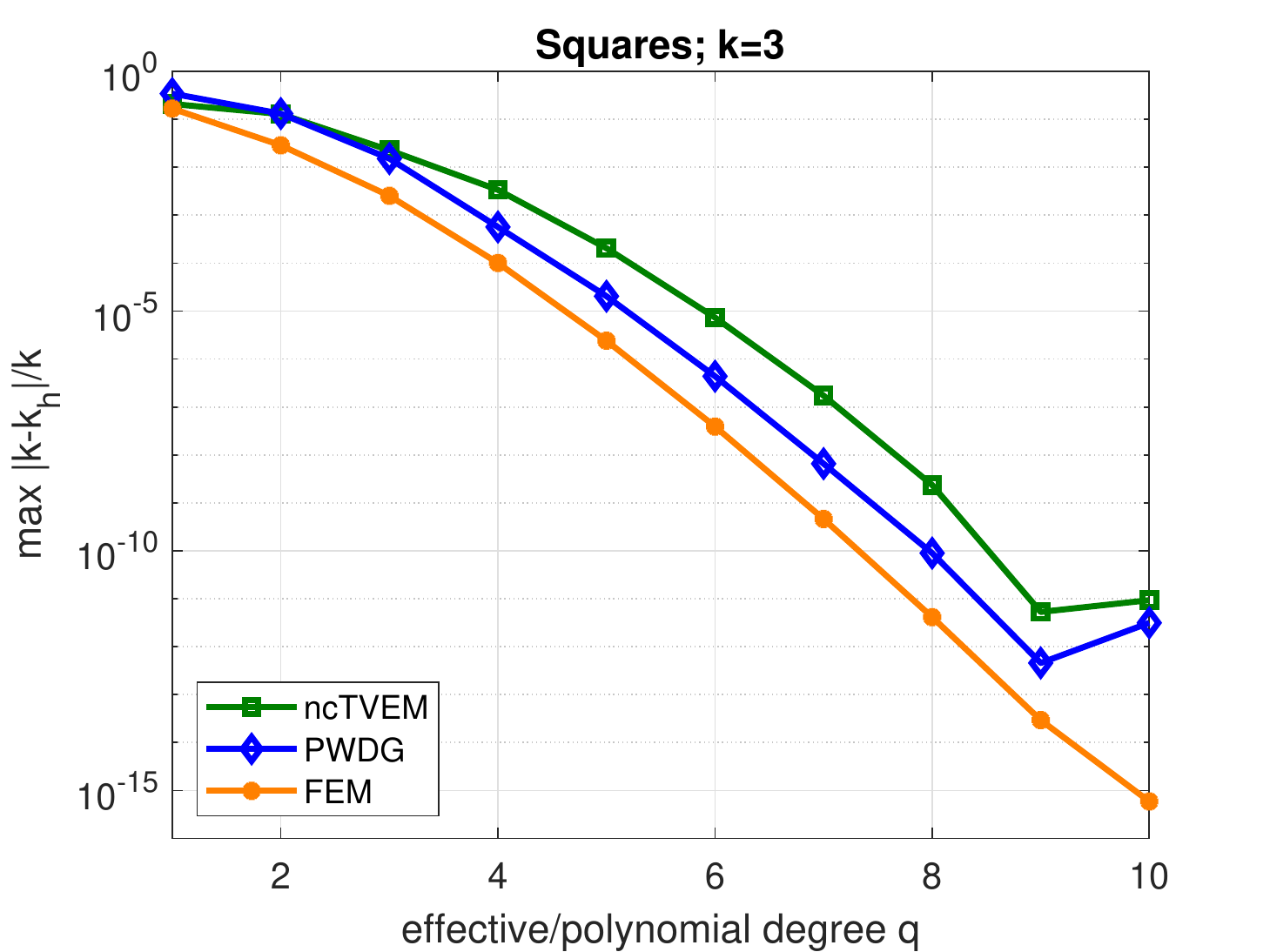}
\end{minipage}
\hfill
\begin{minipage}{0.485\textwidth}
\centering
\includegraphics[width=\textwidth]{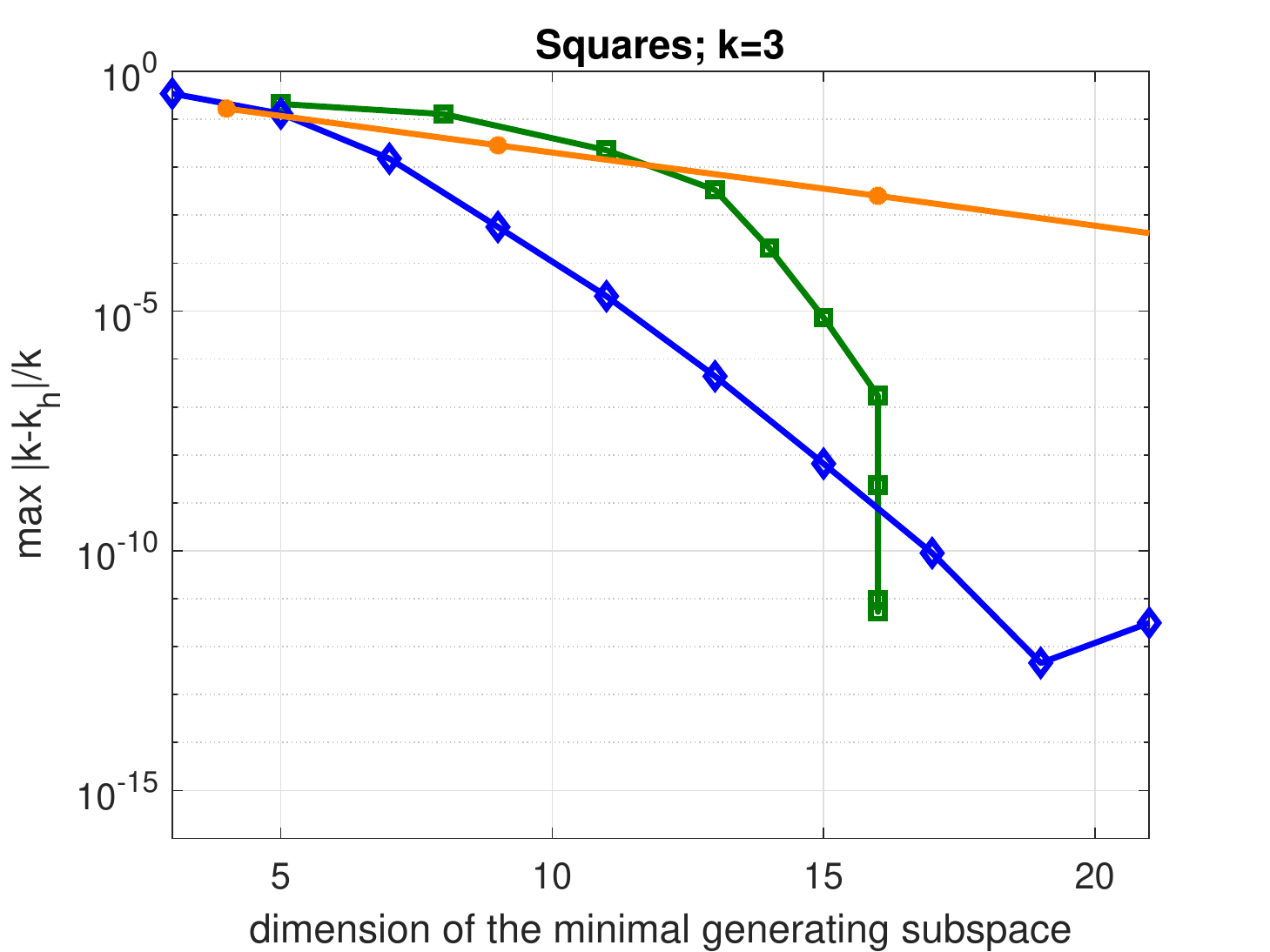}
\end{minipage}
\end{center}
\caption{Comparison of the relative total dispersion errors for ncTVEM, PWDG, and the standard polynomial based FEM on a mesh made of squares as in Figure~\ref{fig:meshes} for fixed wave number $k=3$,
in dependence on the effective/polynomial degree~$q$ (\textit{left}) and the dimension of the minimal generating subspaces (\textit{right}). The maxima over a large set of Bloch wave directions~$\dir$ are taken.}
\label{fig:disp_diss_q_sq FEM} 
\end{figure}

\subsubsection{Algebraic convergence of the dispersion error against the wave number~$\k$} \label{subsubsection:disp_vs_k}
We study the dispersion and dissipation properties of the three methods with respect to the wave number~$k$.
Due to the fact that~$h=1$, and~$k$ is related to the wave number~$k_0$ on a mesh with mesh size $h_0$ by $k=kh=k_0 h_0$, the limit $k \to 0$ corresponds in fact to an $h$-version with $h_0 \to 0$ for fixed $k_0$.
We will observe algebraic convergence of the total dispersion error in terms of~$k$.
This resembles the algebraic convergence of the discretization error in the $h$-version, proved in~\cite{TVEM_Helmholtz_num} and~\cite{GHP_PWDGFEM_hversion} for the ncTVEM and the PWDG, respectively.  

For the numerical experiments, we fix the effective degrees $q=3,5,7$. We employ once again the meshes made of squares and triangles in Figure~\ref{fig:meshes}. Similar results have been obtained on the mesh made of hexagons. 
In Figure~\ref{fig:comparison k}, the relative total errors $|k-\kn|/k$ determined over a large set of Bloch wave directions $\dir$ are depicted against~$k$. Algebraic convergence can be observed. Furthermore, larger values of~$q$ lead to smaller errors.
The peaks occurring in the convergence regions of the ncTVEM could be related to the presence of Neumann eigenvalues, and Dirichlet and Neumann eigenvalues,
that have to be excluded in the construction of the ncTVEM, respectively, in order to have a well-posed variational formulation.
Moreover, the oscillations for larger and smaller values of $k$ are related to the pre-asymptotic regime and the instability regime, which are typical of wave based methods.

In Table~\ref{tab:comparison-k}, we list some relative total errors for different values of~$k$. They indicate a convergence behaviour of
\begin{equation} \label{eq:k_kn_keta}
\max \frac{|k-\kn|}{|k|} \approx \mathcal{O}(k^{\eta}), \quad k \to 0,
\end{equation}
where $\eta \in [2q-1,2q]$. This was already observed in~\cite{gittelson} for PWDG.

\begin{table}[htb]
	\centering
	\begin{tabular}{c|c||c|c|c|c|c||c|c|c|c|c|}
		& \multirow{2}{*}{method} & \multicolumn{5}{|c||}{squares} & \multicolumn{5}{|c|}{triangles} \\ 
		\cline{3-12}
		& & $k$ & $\frac{|k-\kn|}{k}$ & $k$ & $\frac{|k-\kn|}{k}$ & rate & $k$ & $\frac{|k-\kn|}{k}$ & $k$ & $\frac{|k-\kn|}{k}$ & rate \\ 
		\hline
		\multirow{3}{*}{$q=3$} & PWVEM & 2 & 1.50e-03 & 0.3 & 4.59e-08 & 5.48 & 2 & 2.71e-04 & 0.3 & 3.42e-09 & 5.95 \\ 
		\cline{2-12}
		& ncTVEM & 2 & 9.04e-03  & 0.3 & 3.69e-07 & 5.33 & 2 & 1.07e-03 & 0.3 & 4.09e-08 & 5.36 \\ 
		\cline{2-12}
		& PWDG & 2 & 1.71e-03 & 0.3 & 1.04e-07 & 5.11 & 2 & 3.87e-04 & 0.3 & 3.04e-08 & 4.98 \\ 
		\hline \hline 
		\multirow{3}{*}{$q=5$} & PWVEM & 2 & 3.68e-06 & 0.8 & 5.09e-10 & 9.70 & 3 & 2.17e-05 & 2 & 4.54e-07 & 9.53 \\ 
		\cline{2-12}
		& ncTVEM & 2 & 6.48e-06 & 0.8 & 1.21e-09 & 9.37 & 3 & 5.91e-06 & 2 & 1.47e-07 & 9.11 \\ 
		\cline{2-12}
		& PWDG & 2 & 4.56e-07 & 0.8 & 1.47e-10 & 8.77 & 3 & 7.75e-07 & 2 & 1.97e-08 & 9.06 \\ 
		\hline \hline 
		\multirow{3}{*}{$q=7$} & PWVEM & 4 & 1.55e-05 & 2 & 2.23e-09 & 12.76 & 6 & 7.79e-05 & 4 & 5.57e-07 & 12.19 \\ 
		\cline{2-12}
		& ncTVEM & 4 & 5.93e-06 & 2 & 6.54e-10 & 13.15 & 6 & 6.01e-06 & 4 & 3.39e-08 & 12.77 \\ 
		\cline{2-12}
		& PWDG & 4 & 2.92e-07 & 2 & 2.33e-11 & 13.62 & 6 & 7.10e-07 & 4 & 2.76e-09 & 13.69 \\ 
	\end{tabular} 
	\caption{Rates of the relative total error for $k \to 0$.} \label{tab:comparison-k}
\end{table}

\begin{figure}[h]
\begin{center}
\begin{minipage}{0.485\textwidth} 
\centering
\includegraphics[width=\textwidth]{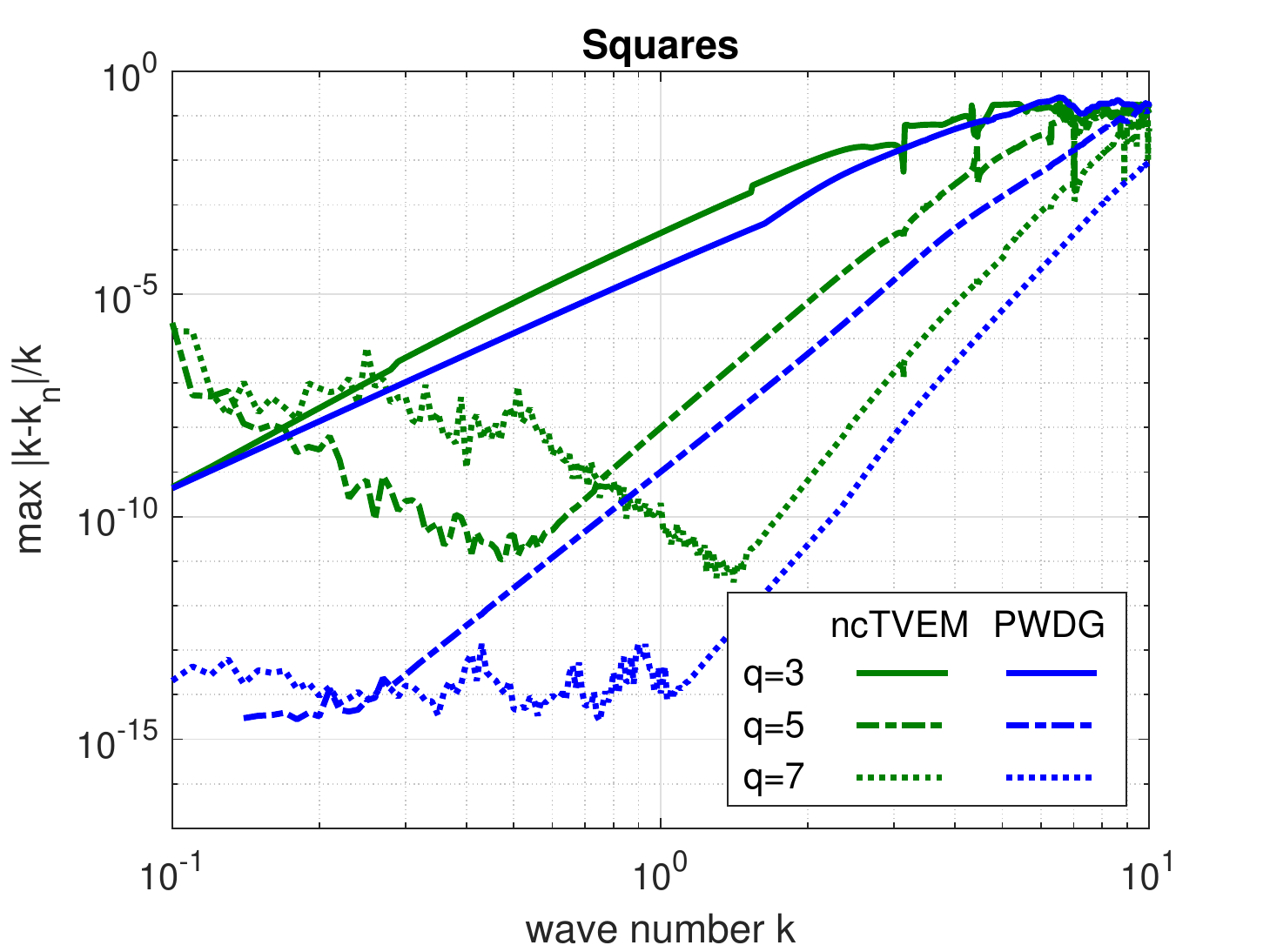}
\end{minipage}
\hfill
\begin{minipage}{0.485\textwidth}
\centering
\includegraphics[width=\textwidth]{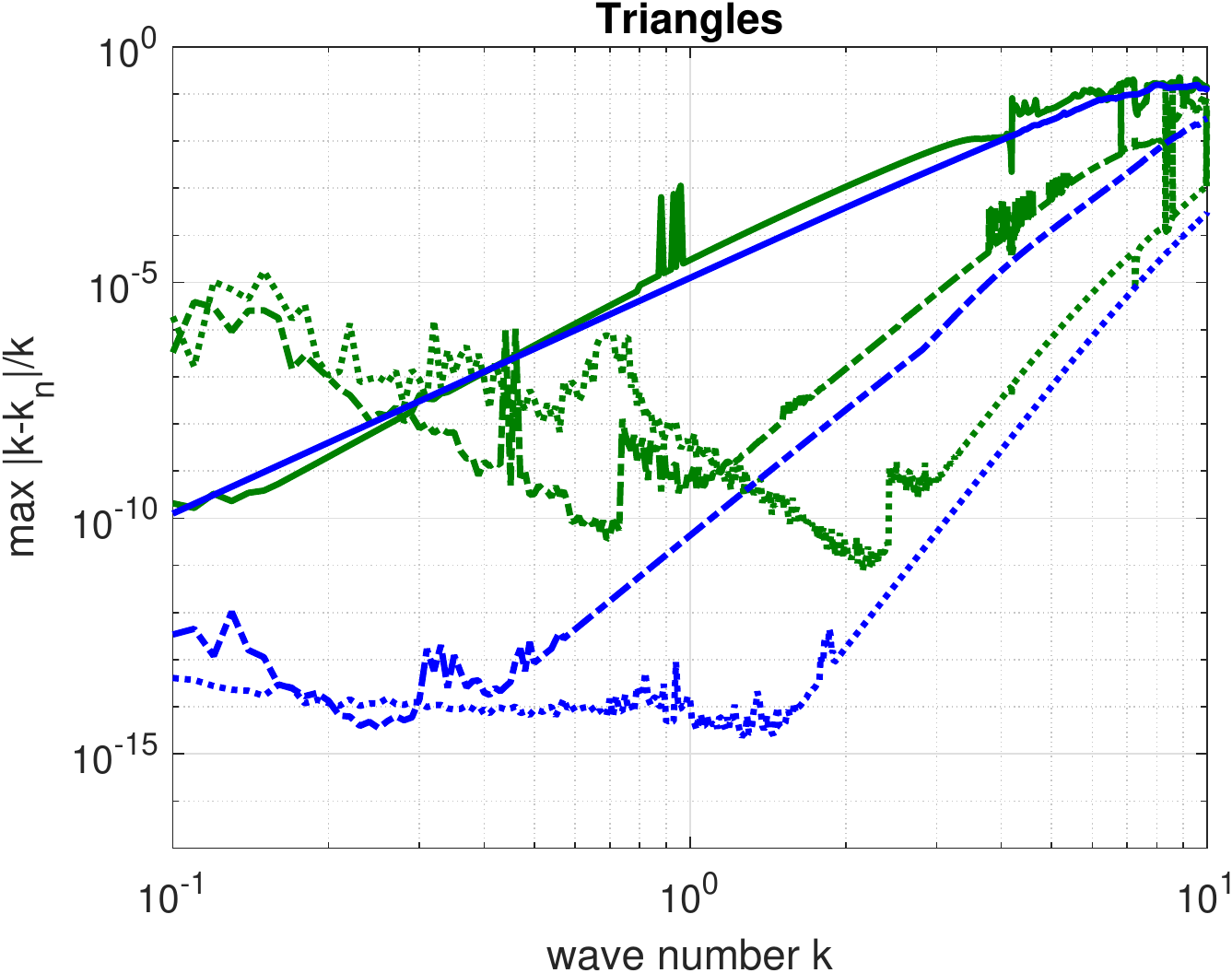}
\end{minipage}
\end{center}
\caption{Relative total dispersion in dependence on the wave number~$k$ for fixed effective degrees $q=3,5,7$.
The maxima over a large set of Bloch wave directions $\dir$ are taken. As meshes, those made of squares (\textit{left}) and triangles (\textit{right}) in Figure~\ref{fig:meshes} are employed.}
\label{fig:comparison k} 
\end{figure}

\begin{remark}
Clearly, similarly as above, dispersion and dissipation can be investigated again separately from each other. Here, we only show the results, depicted in Figure~\ref{fig:comparison k_disp_diss_q3}, for fixed~$q=5$ and varying~$k$ on the meshes made of squares.
As already observed, one can deduce that the ncTVEM are dispersion dominated, whereas dissipation plays a major role for the PWDG.
\end{remark}

\begin{figure}[h]
\begin{center}
\begin{minipage}{0.485\textwidth} 
\centering
\includegraphics[width=\textwidth]{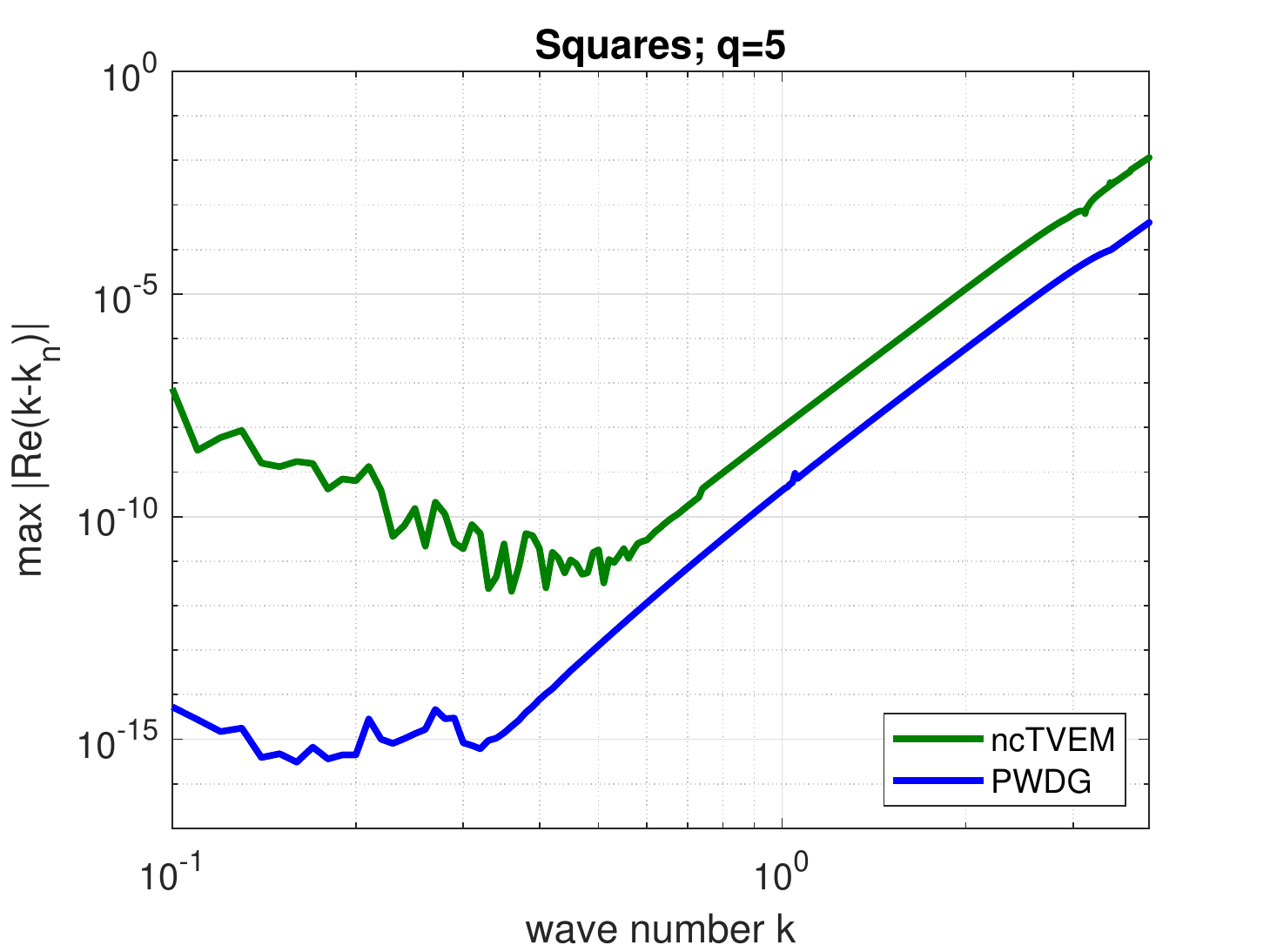}
\end{minipage}
\hfill
\begin{minipage}{0.485\textwidth}
\centering
\includegraphics[width=\textwidth]{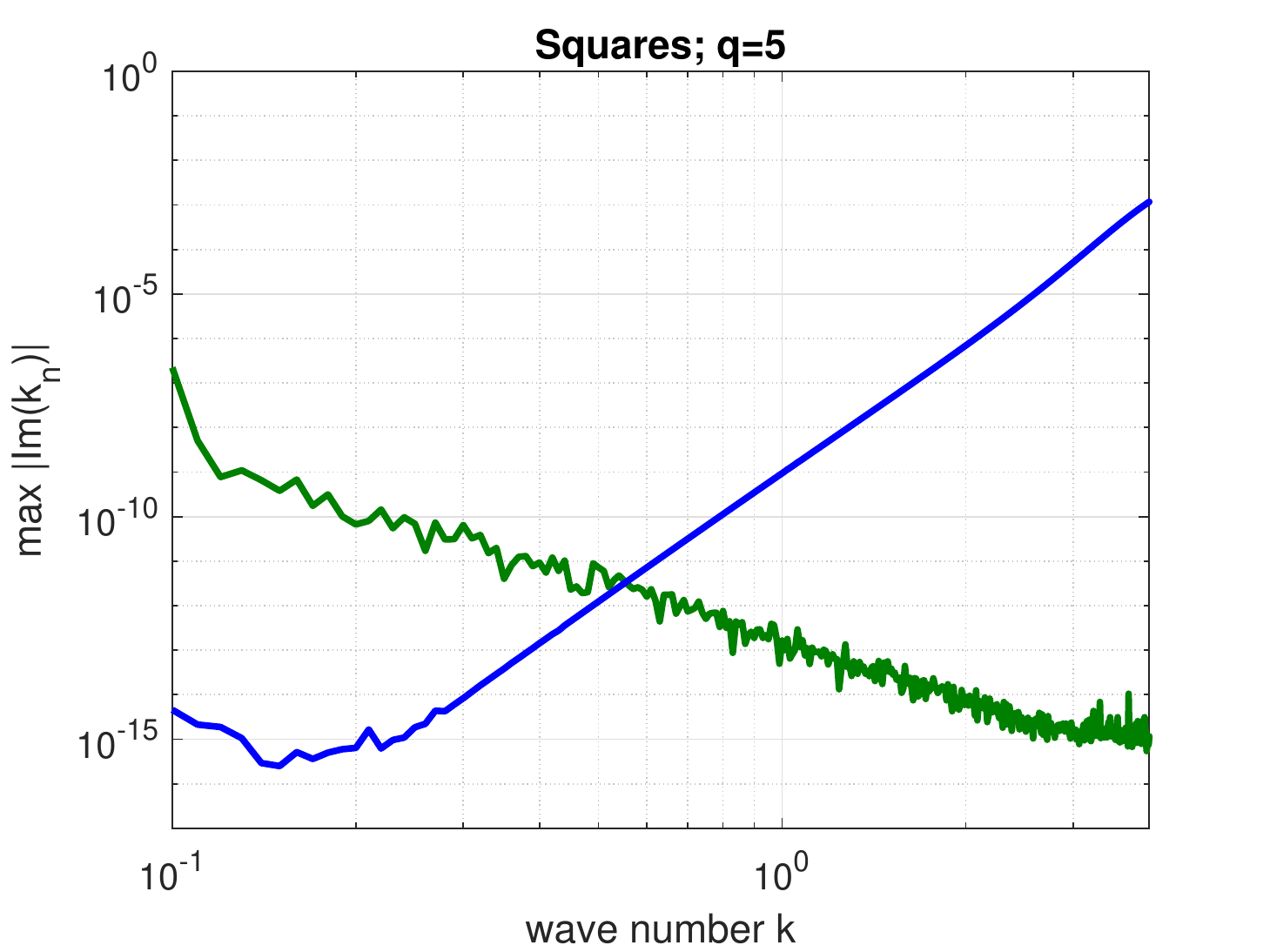}
\end{minipage}
\end{center}
\caption{Relative dispersion (\textit{left}) and relative dissipation (\textit{right}) in dependence on the wave number~$k$ for fixed $q=5$ on the meshes made of squares in Figure~\ref{fig:meshes}. The maxima over a large set of Bloch wave directions $\dir$ are taken.}
\label{fig:comparison k_disp_diss_q3} 
\end{figure}

\section*{Acknowledgements}
I. Perugia has been funded by the Austrian Science Fund (FWF) through the projects F~65 and P~29197-N32.
L. Mascotto acknowledges the support of the Austrian Science Fund (FWF) through the project P~33477.

{\footnotesize
\bibliography{bibliogr}
}
\bibliographystyle{plain}

\end{document}